\documentclass{tac}
\usepackage{subfiles}
\usepackage{mathtools}
\usepackage[T1]{fontenc}
\usepackage{inputenc}
\usepackage{microtype}
\usepackage[usenames,dvipsnames]{xcolor}
\usepackage{enumitem}
\usepackage{booktabs}
\usepackage{array}
\usepackage{pbox}
\usepackage{tikz}
\usepackage[textsize=footnotesize]{todonotes}
\usepackage{caption}
\usepackage{multirow,comment}
\usepackage{hyperref}
\hypersetup{colorlinks,linkcolor={blue!50!black},citecolor={red!40!black},urlcolor={blue!75!black}}
\usepackage{amssymb}
\usepackage[theoremfont]{newpxtext}
\usepackage[varg,bigdelims,nosymbolsc,noamssymbols]{newpxmath}
\usepackage[cal=euler]{mathalfa}

\usepackage{bm}
\usepackage[capitalize]{cleveref}

\crefname{equation}{}{}
\crefname{condition}{condition}{conditions}
\crefname{item}{}{}
\crefname{warning}{Warning}{Warnings}

\usetikzlibrary{decorations.markings,arrows.meta,calc,fit,quotes,cd,math}
\hypersetup{final}

\setlist{nosep,leftmargin=*,labelsep=.5em}
\setlist[1]{labelindent=\parindent}

\newtheorem{theoremold}{Theorem}
\newtheorem{propositionold}{Proposition}
\newtheorem{corollaryold}{Corollary}
\newtheorem{lemmaold}{Lemma}

\newtheoremrm{definitionold}{Definition}
\newtheoremrm{remarkold}{Remark}
\newtheoremrm{exampleold}{Example}
\newtheoremrm{warningold}{Warning}

\renewenvironment{theorem}{%
  \let\labelold\label%
  \def\label##1{\labelold[theorem]{##1}}%
  \begin{theoremold}%
}{\end{theoremold}}
\renewenvironment{proposition}{%
  \let\labelold\label%
  \def\label##1{\labelold[proposition]{##1}}%
  \begin{propositionold}%
}{\end{propositionold}}
\renewenvironment{corollary}{%
  \let\labelold\label%
  \def\label##1{\labelold[corollary]{##1}}%
  \begin{corollaryold}%
}{\end{corollaryold}}
\renewenvironment{lemma}{%
  \let\labelold\label%
  \def\label##1{\labelold[lemma]{##1}}%
  \begin{lemmaold}%
}{\end{lemmaold}}
\renewenvironment{definition}{%
  \let\labelold\label%
  \def\label##1{\labelold[definition]{##1}}%
  \begin{definitionold}%
}{\end{definitionold}}
\renewenvironment{remark}{%
  \let\labelold\label%
  \def\label##1{\labelold[remark]{##1}}%
  \begin{remarkold}%
}{\end{remarkold}}
\renewenvironment{example}{%
  \let\labelold\label%
  \def\label##1{\labelold[example]{##1}}%
  \begin{exampleold}%
}{\end{exampleold}}
\newenvironment{warning}{%
  \let\labelold\label%
  \def\label##1{\labelold[warning]{##1}}%
  \begin{warningold}%
}{\end{warningold}}

\tikzset{
  entity/.style = {
    circle,
    draw,
    minimum size=7mm,
    fill=black!20,
    inner sep=0pt,
    font={\scriptsize\ttfamily}
  },
  type/.style = {
    circle,
    draw,
    minimum size=7mm,
    inner sep=0pt,
    font={\scriptsize\sffamily}
  },
  schema/.style = {
    x=2cm, y=2cm,
    node distance=1 and 1,
    on grid,
    shorten >=1pt,
    bend angle=25,
    every edge quotes/.style = {auto, font={\tiny\ttfamily}},
    every edge/.style = {draw, ->},
    every fit/.style = {draw, semithick, inner sep=9pt},
    baseline=(current bounding box.center)
  },
  schema eqs/.style = {font={\tiny},align=left,inner sep=0pt},
  schema name/.style = {node distance=9pt and 9pt}
}
\newcommand{\simrightarrow}{\xrightarrow{\raisebox{-4pt}[0pt][0pt]{\ensuremath{\sim}}}}

\DeclareMathOperator{\id}{id}
\DeclareMathOperator{\dom}{dom}
\DeclareMathOperator{\cod}{cod}
\DeclareMathOperator{\colim}{colim}

\DeclareMathOperator{\Hom}{Hom}
\DeclareMathOperator{\Ob}{Ob}

\newcommand{\iso}{\cong}
\renewcommand{\equiv}{\simeq}

\newcommand{\Ran}{\mathrm{Ran}}
\newcommand{\Lan}{\mathrm{Lan}}
\newcommand{\Fun}[1]{[#1]}
\newcommand{\BigFun}[1]{\left[#1\right]}
\newcommand{\subsub}[3]{{}_{#2}#1_{#3}}

\mathchardef\mhyphen="2D 
\def\dash{\unskip\kern.16667em---\penalty\exhyphenpenalty\hskip.16667em\relax}

\newcommand{\cat}[1]{\mathscr{#1}} 
\newcommand{\ncat}[1]{\mathbf{#1}} 
\newcommand{\ccat}[1]{\bm{\mathcal{#1}}} 
\newcommand{\dcat}[1]{\vmathbb{#1}} 
\newcommand{\nfun}[1]{\mathbf{#1}}  
\newcommand{\nncat}[2]{\ccat{#1}\ncat{#2}} 
\newcommand{\ndcat}[2]{\dcat{#1}\ncat{#2}}

\newcommand{\Type}{\ncat{Type}}
\newcommand{\schema}[1]{\mathbf{#1}} 
\newcommand{\snodes}[1]{#1_{\tn{e}}} 
\newcommand{\satts}[1]{#1_{\tn{o}}} 
\newcommand{\scol}[1]{\widetilde{#1}} 
\newcommand{\incnodes}[1]{i_{#1}}
\newcommand{\inctypes}{i_{\ncat{T}}}
\newcommand{\TypeAlg}{\Type\alg}

\newcommand{\map}[1]{\bm{\mathrm{#1}}} 
\newcommand{\mnodes}[1]{#1_{\tn{e}}} 
\newcommand{\matts}[1]{#1_{\tn{o}}} 
\newcommand{\mcol}[1]{\widetilde{#1}} 

\newcommand{\inst}[1]{\mathbf{#1}} 
\newcommand{\inodes}[1]{#1_{\tn{e}}} 
\newcommand{\iterms}[1]{#1_{\tn{t}}} 
\newcommand{\iatts}[1]{#1_{\tn{o}}} 

\newcommand{\trans}[1]{\bm{\mathrm{#1}}} 
\newcommand{\tnodes}[1]{#1_{\tn{e}}} 
\newcommand{\tterms}[1]{#1_{\tn{t}}} 

\newcommand{\bimod}[1]{\mathbf{#1}} 
\newcommand{\bnodes}[1]{#1_{\tn{e}}}
\newcommand{\bterms}[1]{#1_{\tn{t}}}
\newcommand{\batts}[1]{#1_{\tn{o}}}
\newcommand{\bret}[1]{#1_{\tn{r}}}
\newcommand{\bcol}[1]{\widetilde{#1}} 
\newcommand{\bint}[1]{#1_{\tn{o}}} 

\newcommand{\twoCell}[1]{\bm{#1}} 
\newcommand{\twonodes}[1]{#1_{\tn{e}}} 
\newcommand{\twoterms}[1]{#1_{\tn{t}}} 
\newcommand{\twocol}[1]{\widetilde{#1}} 

\newcommand{\eqnodes}[1]{#1_{\tn{e}}} 
\newcommand{\eqatts}[1]{#1_{\tn{o}}} 

\newcommand{\query}[1]{#1}
\newcommand{\qfor}[1]{#1_{f}}
\newcommand{\qwhere}[1]{#1_{w}}
\newcommand{\qatts}[1]{#1_{a}}
\newcommand{\qret}[1]{#1_{r}}

\newcommand{\sinst}[1]{\schema{#1}\ncat{\mhyphen\mkern-.7muInst}} 

\newcommand{\Schema}{\ncat{Schema}}
\newcommand{\Data}{\ndcat{D}{ata}}

\newcommand{\ATh}{\ncat{ATh}}
\newcommand{\ASig}{\ncat{ASig}}

\newcommand{\Cxt}[1]{\ncat{Cxt}_{#1}}
\newcommand{\piCxt}[1]{\overline{\Cxt{#1}}}
\newcommand{\APr}{\ncat{APr}}
\newcommand{\Fr}{\mathrm{Fr}}

\newcommand{\ProfTimes}{\ncat{Prof}^{\times}}
\newcommand{\LambdaTimes}{\Lambda^{\times}}
\newcommand{\comp}[1]{\widehat{#1}}
\newcommand{\conj}[1]{\widecheck{#1}}

\newcommand{\One}{\mathsf{1}}
\newcommand{\emptyContext}{\varnothing}

\newcommand{\Ents}{\textsf{Entities}}
\newcommand{\Edges}{\textsf{Edges}}
\newcommand{\PathEqs}{\textsf{Path Eqs}}
\newcommand{\Atts}{\textsf{Attributes}}
\newcommand{\ObsEqs}{\textsf{Obs. Eqs}}

\newcommand{\Set}{\ncat{Set}} 
\newcommand{\Cat}{\ncat{Cat}} 
\newcommand{\CCat}{\nncat{C}{at}} 
\newcommand{\Prof}{\nncat{P}{rof}} 
\newcommand{\dProf}{\ndcat{P}{rof}} 

\newcommand{\EntOb}[1]{\ensuremath{\mathtt{#1}}}
\newcommand{\Mor}[1]{\ensuremath{\mathtt{#1}}}

\newcommand{\TypeOb}[1]{\mathsf{#1}}
\newcommand{\Int}{\TypeOb{Int}}
\newcommand{\Str}{\TypeOb{Str}}
\newcommand{\Bool}{\TypeOb{Bool}}

\newcommand{\unit}{\mathsf{U}}
\newcommand{\lframe}{\mathsf{L}}
\newcommand{\rframe}{\mathsf{R}}

\newcommand{\Ver}{\ncat{Vert}} 
\newcommand{\HHor}{\bm{\ccat{H}}} 

\newcommand{\ZZ}{\mathbb{Z}}
\newcommand{\Col}[1]{\mathbf{Col}(\mathbf{#1})}
\newcommand{\col}[1]{\mathrm{Col}(#1)}
\newcommand{\Colnodes}[1]{\snodes{\col{#1}}}
\newcommand{\Colatts}[1]{\satts{\col{#1}}}

\newcommand{\op}[1]{#1^{\mathrm{op}}}

\newcommand{\alg}{\ncat{\mhyphen\mkern-1.5muAlg}}
\newcommand{\T}[1][3]{\cat{T}\mkern-#1mu}
\newcommand{\Bimod}[2]{{}_{#1}\mathrm{Bimod}_{#2}}

\newcommand{\yoneda}{\mathbf{y}}
\newcommand{\LAdj}{\nncat{L}{Adj}}
\newcommand{\RAdj}{\nncat{R}{Adj}}
\newcommand{\Simp}[2]{{}_{#1}\ncat{Simp}_{#2}}
\newcommand{\Cocone}[2]{{}_{#1}\ncat{Cocone}_{#2}}
\newcommand{\Cone}[2]{{}_{#1}\ncat{Cone}_{#2}}
\newcommand{\res}[2]{{}_{#1}\mathrm{Res}_{#2}}
\newcommand{\FrAlg}[2][\kappa]{#1[#2]}
\newcommand{\FrCpsh}[1]{\langle#1\rangle}
\newcommand{\FrInst}[1]{\langle#1\rangle}

\newcommand{\RR}{\mathbb{R}}

\newcommand*{\cocolon}{
  \nobreak
  \mskip6mu plus1mu
  \mathpunct{}%
  \nonscript
  \mkern-\thinmuskip
  {:}%
  \mskip2mu
  \relax
}

\newcommand{\tn}[1]{\mathrm{#1}}
\newcommand{\singleton}[1][\ast]{\{#1\}}

\newcommand{\from}{\leftarrow}

\newcommand{\tickar}{\begin{tikzcd}[baseline=-0.5ex,cramped,sep=small,ampersand replacement=\&]{}\ar[r,tick]\&{}\end{tikzcd}}
\newcommand{\xtickar}[1]{\stackrel{#1}{\tickar}}

\newcommand{\tickxar}{\begin{tikzcd}[baseline=-0.5ex,cramped,sep=small,ampersand replacement=\&]{}\ar[r,tickx]\&{}\end{tikzcd}}
\newcommand{\xtickxar}[1]{\stackrel{#1}{\tickxar}}

\newcommand{\twocell}[3][]{\arrow[draw=none,to path={(dom#2.center)--(cod#2.center)\tikztonodes}]{}[anchor=center,#1]{\Downarrow #3}}
\newcommand{\twocellalt}[3][]{\arrow[draw=none,to path={(dom#2.center)--(cod#2.center)\tikztonodes}]{}[anchor=center,#1]{#3}}
\newcommand{\twocellA}[2][]{\twocell[#1]{A}{#2}}
\newcommand{\twocellB}[2][]{\twocell[#1]{B}{#2}}
\newcommand{\twocellC}[2][]{\twocell[#1]{C}{#2}}

\newcommand{\FPQL}{CQL}

\tikzcdset{
   arrow style=tikz,
   diagrams={>={Classical TikZ Rightarrow[angle=63:4pt, line width=.6pt]}},
   arrows={semithick}
}

\tikzset{
  tick/.style={postaction={
    decorate,
    decoration={markings, mark=at position 0.5 with {\draw[-] (0,.4ex) -- (0,-.4ex);}}}
  },
  tickx/.style={
    postaction={ decorate,
      decoration={markings,
        mark=at position 0.5 with {
          \fill circle [radius=.28ex];
        }
      }
    }
  }
}
\tikzset{
   dom/.style={append after command={coordinate[alias=dom#1]}},
   domA/.style={dom=A}, domB/.style={dom=B},
   domC/.style={dom=C}, domD/.style={dom=D},
   domE/.style={dom=E}, domF/.style={dom=F},
   cod/.style={append after command={coordinate[alias=cod#1]}},
   codA/.style={cod=A}, codB/.style={cod=B},
   codC/.style={cod=C}, codD/.style={cod=D},
   codE/.style={cod=E}, codF/.style={cod=F}
}

\DeclareFontFamily{U}{mathx}{\hyphenchar\font45}
\DeclareFontShape{U}{mathx}{m}{n}{
      <5> <6> <7> <8> <9> <10>
      <10.95> <12> <14.4> <17.28> <20.74> <24.88>
      mathx10
      }{}

 \makeatletter
 \DeclareRobustCommand\widecheck[1]{{\mathpalette\@widecheck{#1}}}
 \def\@widecheck#1#2{%
     \setbox\z@\hbox{\m@th$#1#2$}%
     \setbox\tw@\hbox{\m@th$#1%
        \widehat{%
           \vrule\@width\z@\@height\ht\z@
           \vrule\@height\z@\@width\wd\z@}$}%
     \dp\tw@-\ht\z@
     \@tempdima\ht\z@ \advance\@tempdima2\ht\tw@ \divide\@tempdima\thr@@
     \setbox\tw@\hbox{%
        \raise\@tempdima\hbox{\scalebox{1}[-1]{\lower\@tempdima\box
 \tw@}}}%
     {\ooalign{\box\tw@ \cr \box\z@}}}
 \makeatother

\allowdisplaybreaks[1]
\title{Algebraic\;\;Databases}
\author[Schultz,\;\;Spivak,\;\;Vasilakopoulou,\;\;Wisnesky]%
{Patrick Schultz,\;\;David I. Spivak,\;\;Christina Vasilakopoulou, \;and\;\;Ryan Wisnesky}
\thanks{Schultz, Spivak, and Vasilakopoulou were supported by AFOSR grant FA9550--14--1--0031, ONR grant N000141310260, and NASA grant NNL14AA05C.
Wisnesky was supported by NIST SBIR grant 70NANB15H290.}
\copyrightyear{\;2016}
\address{Massachusetts Institute of Technology, 77 Massachusetts Ave.\ Cambridge, MA 02139}
\eaddress{bekudno@gmail.com, david@topos.institute, cvasilak@math.ntua.gr, \\ ryan@wisnesky.net}


\begin{document}

\linespread{1}\selectfont
\maketitle

\begin{abstract}
  Databases have been studied category-theoretically for decades. While mathematically elegant,
  previous categorical models have typically struggled with representing concrete data such as
  integers or strings.

  In the present work, we propose an extension of the earlier set-valued functor model, making use of
  multi-sorted algebraic theories (a.k.a.\ Lawvere theories) to incorporate concrete data in a
  principled way. This approach easily handles missing information (null values), and also allows
  constraints and queries to make use of operations on data, such as multiplication or comparison of
  numbers, helping to bridge the gap between traditional databases and programming languages.

  We also show how all of the components of our model \dash including schemas, instances,
  change-of-schema functors, and queries \dash fit into a single double categorical structure called
  a proarrow equipment (a.k.a.\ framed bicategory).
\end{abstract}

\tableofcontents

\linespread{1.15}\selectfont

\section{Introduction}

Category-theoretic models of databases have been present for some time. For example
in~\cite{Rosebrugh.Wood:1992a,Fleming.Gunther.Rosebrugh:2003a,Johnson.Rosebrugh:2002a} databases
schemas are formalized as sketches of various sorts (e.g.\ EA sketches = finite limits +
coproducts). The data itself (called an \emph{instance}) is represented by a model of the sketch. In
this language, queries can be understood as limit cones in such a sketch. While different from the
traditional relational foundations of database theory~\cite{Abiteboul:1995a}, this is in general a
very natural and appealing idea.

In~\cite{Spivak:2013a}, Spivak puts emphasis on the ability to move data from one format, or
database schema, to another. To enable that, he proposes defining schemas to be mere categories
\dash or in other words trivial sketches (with no (co)limit cones). A schema morphism is just a
functor. Unlike the case for non-trivial sketches, a schema morphism induces three adjoint functors,
the pullback and its Kan extensions. These functors can be called \emph{data migration functors}
because they transfer data from one schema to another. In this formalism, queries can be recovered
as specific kinds of data migration.

Both of the above approaches give some secondary consideration to attributes, e.g.\ the name or
salary of an employee, taking values in some data type, such as strings, integers, or booleans.
Rosebrugh et al.\ formalized attributes in terms of infinite coproducts of a chosen terminal object,
whereas Spivak formalized them by slicing the category of copresheaves over a fixed object. However,
neither approach seemed to work convincingly in implementations \cite{Spivak:2015a}.

\subsection{The approach of this paper}

In the present paper, the goal of providing a principled and workable formalization of attributes is
a central concern. We consider attribute values as living in an algebra over a multi-sorted algebraic
theory, capturing operations such as comparing integers or concatenating strings. A database schema
is formalized as what we call an \emph{algebraic profunctor}, which is a profunctor from a category
to an algebraic theory that preserves the products of the theory. Each element of the profunctor
represents an observation of a given type (string, integer, boolean) that can be made on a certain
entity (employee, department). For example, if an entity has an observable for length and width, and
if the theory has a multiplication, then the entity has an observable for area.

We also focus on providing syntax for algebraic databases. We can present a schema, or an instance
on it, using a set of generators and relations. The generators act like the ``labelled nulls'' used in modern relational databases, easily handling unknown information, while the relations
are able to record constraints on missing data. In this sense, our approach can be related to
knowledge bases or ontologies \cite{Munn:2008a}. One can express that Pablo is an employee whose
salary is between 65 and 75, and deduce various facts; for example, if the schema expresses that
each employee's salary is at most that of his or her manager, one can deduce that Pablo's manager
makes at least 65.

Mathematically, this paper develops the theory of algebraic profunctors. An algebraic profunctor can
be regarded as a diagram of models for an algebraic theory $\T$, e.g.\ a presheaf of rings or
modules on a space. Algebraic profunctors to a fixed $\T[2]$ form the objects in a proarrow
equipment \dash a double category satisfying a certain fibrancy condition \dash which we call
$\Data$. This double category includes database schemas and schema morphisms, and we show that the
horizontal morphisms (which we call \emph{bimodules} between schemas) generalize both instances and
conjunctive queries.

We make heavy use of \emph{collages} of profunctors and bimodules. Collages are a kind of
double-categorical colimit which have been studied in various guises under various names \dash
\cite{Garner.Shulman:2013a} gives a good general treatment. We propose exactness properties which
the collage construction satisfies in some examples; we say that an equipment has \emph{extensive
collages} when these properties hold. This fits in with the work started in \cite{Schultz:2015a},
and may be of interest independent of the applications in this paper. Although the present work only
makes use of the properties of extensive collages in the equipment $\dProf$ of categories, functors,
and profunctors, we found more direct proofs of these properties in this case to be no easier and
less illuminating.

To connect the theory with practice, it is necessary to have a concrete syntax for presenting the
various categorical structures of interest. While it is mostly standard, we provide a self-contained
account of a type-theoretic syntax for categories, functors, profunctors, algebraic theories,
algebras over those theories, and algebraic profunctors. We use this syntax to consistently ground
the theoretical development with concrete examples in the context of databases, though the reader
need not have any background in that subject.

\subsection{Implementation}

The mathematical framework developed in this paper is implemented in an open-source
software system we call \FPQL, available for download at \url{http://categoricaldata.net}.  All
examples from this paper are included as built-in demonstrations in the \FPQL\ tool.  We
defer a detailed discussion of \FPQL\ until the end of the paper (\cref{Impl}), but two high-level
introductory remarks are in order.

First, we note that most constructions on finitely-presented categories require solving word problems
in categories and hence are not computable~\cite{Fleming.Gunther.Rosebrugh:2003a}.
Given a category presented by generators $G$ and relations (equations) $E$, the word problem asks if
two terms (words) in $G$ are equal under $E$.  Although not decidable in general, many approaches
to this problem have been proposed; we discuss our particular
approach in \cref{Impl}.  If we can solve the word problem for a particular category presentation, then we can use that
decision procedure to implement query evaluation, construct collages, and perform other tasks.

Second, we note that there are many connections between the mathematical framework presented here and
various non-categorical frameworks.  When restricted to a discrete algebraic theory, the query language
we discuss in \cref{Queries} corresponds exactly to relational algebra's unions of conjunctive queries
under bag semantics~\cite{Spivak:2015a}. This correspondence allows fragments of our framework to
be efficiently implemented using existing relational systems (MySQL, Oracle, etc), and our software has
indeed been used on various real-world examples~\cite{Spivak:2015b}.

\subsection{Outline}

In \cref{sec:profunctors_equipments} we review profunctors and use them to motivate the definition
of double categories and proarrow equipments. We also review, as well as refine, the notion of
collages, which exist in all of the equipments of interest in this paper. In
\cref{sec:algebraic_theories} we review multisorted algebraic theories, and we discuss profunctors
\dash from categories to algebraic theories \dash that preserve products in the appropriate way; we
call these algebraic profunctors. We save relevant database-style examples until
\cref{sec:presentations_syntax}, where we provide type-theoretic syntax for presenting theories,
categories and (algebraic) profunctors. This section serves as a foundation for the syntax used
throughout the paper, especially in examples, though it can be skipped by those who only want to
understand the category theoretic concepts.

We get to the heart of the new material in \cref{sec:schemas} and \cref{sec:instances}, where we
define schemas and instances for algebraic databases and give examples. Morphisms between schemas
induce three adjoint functors \dash called data migration functors \dash between their instance
categories, and we discuss this in \cref{ssec:DataMigrationFunctors}.

In \cref{ThedoublecategoryData} we wrap all of this into a double category (in fact a proarrow
equipment) $\Data$, in which schemas are objects, schema morphisms are vertical morphisms, and
schema bimodules \dash defined in this section \dash are horizontal morphisms. Instances are shown
to be bimodules of a special sort, and the data migration functors from the previous section are
shown to be obtained by composition and exponentiation of instance bimodules with representable
bimodules. In this way, we see that $\Data$ nicely packages all of the structures and operations of
interest.

Finally, in \cref{Queries} we discuss the well-known "Select-From-Where" queries of standard
database languages and show that they form a very special case of our data migration setup. We
conclude with a discussion of the implementation of our mathematical framework in \cref{Impl}.

\subsection{Notation}

In this paper we will adhere to the following notation. For named categories, such as the
category $\Set$ of sets, we use bold roman. For category variables \dash for instance "Let $\cat{C}$
be a category" \dash we use math script. 

Named bicategories or 2-categories, such as the 2-category $\CCat$ of small categories,
will be denoted similarly to named 1-categories except with calligraphic first letter. We use
the same notation for a variable bicategory $\ccat{B}$.

Double categories, such as the double category $\dProf$ of categories, functors, and profunctors, will be denoted like 1-categories except with blackboard bold first letter. We use the same notation for a variable double category $\dcat{D}$.

If $\cat{C}$ and $\cat{D}$ are categories, we sometimes denote the functor category $\CCat(\cat{C},\cat{D})$ by $\Fun{\cat{C},\cat{D}}$
or $\cat{D}^{\cat{C}}$.

\subsection{Acknowlegements}
The authors thank the anonymous referee for many helpful and questions and comments.

\section{Profunctors and proarrow equipments}\label{sec:profunctors_equipments}

We begin with a review of profunctors, which are sometimes called correspondences or distributors;
standard references include \cite{Borceux:1994a-1} and \cite{Benabou:2000a}. Together with
categories and functors, these fit into a proarrow equipment in the sense of Wood \cite{Wood:1982a,
Wood:1985a}, though we follow the formulation in terms of double categories called framed
bicategories (or fibrant double categories), due to Shulman \cite{Shulman:2008a,Shulman:2010a}. Eventually, in
\cref{ThedoublecategoryData}, we will produce an equipment $\Data$ that encompasses database
schemas, morphisms, instances, and queries.

\subsection{Profunctors}
  \label{ssec:profunctors}
Perhaps the most important example of an equipment is that of categories, functors, and profunctors.
We review profunctors here, as they will be a central player in our story.

Let $\cat{C}$ and $\cat{D}$ be categories. Recall that a \emph{profunctor} $M$ from $\cat{C}$ to
$\cat{D}$, written $M\colon\cat{C}\tickar\cat{D}$, is defined to be a functor
$M\colon\op{\cat{C}}\times\cat{D}\to\Set$.

\subsection{Profunctors as matrices}
  \label{ssec:profunctor_matrix}

It can be helpful to think of profunctors as something like matrices. Given finite sets $X$ and $Y$,
there is an equivalence between
\begin{itemize}[nosep]
\item $X\times Y$-matrices $A$ (i.e.\ functions $X\times Y\to\RR$),
\item functions $A\colon X\to\RR^Y$,
\item functions $A\colon Y\to\RR^X$,
\item linear maps $L_A\colon\RR^X\to\RR^Y$,
\item linear maps $L'_A\colon\RR^Y\to\RR^X$.
\end{itemize}
Similarly, there is an equivalence between
\begin{itemize}[nosep]
\item profunctors $M\colon\cat{C}\tickar\cat{D}$,
\item functors $M\colon\op{\cat{C}}\to\Set^{\cat{D}}$,
\item functors $M\colon\cat{D}\to\Set^{\op{\cat{C}}}$,
\item colimit-preserving functors $\Lambda_M\colon\Set^{\cat{C}}\to\Set^{\cat{D}}$,
\item colimit-preserving functors $\Lambda'_M\colon\Set^{\op{\cat{D}}}\to\Set^{\op{\cat{C}}}$.
\end{itemize}
The first three correspondences are straightforward by the cartesian monoidal closed structure of
$\Cat$. The last two follow from the fact that, just as $\RR^Y$ is the free real vector space
on the set $Y$, the category $\Set^{\op{\cat{D}}}$ is the free
completion of $\cat{D}$ under colimits, and similarly for $\Set^{\cat{C}}$.
By the equivalence between colimit-preserving functors $\Set^{\cat{C}}\to\cat{E}$ and functors
$\op{\cat{C}}\to\cat{E}$ for any cocomplete category $\cat{E}$, the functor $\Lambda_M$
is obtained by taking the left Kan extension of $M\colon\op{\cat{C}}\to\Set^{\cat{D}}$ along
the Yoneda embedding $\yoneda\colon\op{\cat{C}}\to\Set^{\cat{C}}$. Using the pointwise formula
for Kan extensions, this means that given any $I\colon\cat{C}\to\Set$, the functor $\Lambda_M(I)
\colon\cat{D}\to\Set$ is given by the coend formula
\begin{equation}
    \label{eqn:LambdaPointwise}
  (\Lambda_MI)(d) = \int^{c\in\cat{C}}I(c)\times M(c,d).
\end{equation}
This is analogous to the matrix formula $(L_Av)_y=\sum_{x\in X}v_xA_{x,y}$.

Alternatively, since colimits in $\Set^{\cat{D}}$ are computed pointwise, we can express
$\Lambda_MI$ itself as a coend in $\Set^{\cat{D}}$
\begin{equation}
    \label{eqn:LambdaNonPointwise}
  \Lambda_MI = \int^{c\in\cat{C}}I(c)\cdot M(c)
\end{equation}
where we think of $M$ as a functor $\op{\cat{C}}\to\Set^{\cat{D}}$. The symbol\; $\cdot$\; represents
the set-theoretic copower (see \cite{Kelly:1982a}), i.e.\ $I(c)\cdot M(c)$ is an $I(c)$-fold coproduct of copies
of $M(c)$. Formula~\cref{eqn:LambdaNonPointwise} is analogous to the matrix formula $L_Av=\sum_{x\in X}A(x)v_x$, where we
think of $A$ as a function $X\to\RR^Y$ and $A(x)v_x$ denotes scalar multiplication by $v_x\in\RR$.
The construction of $\Lambda_M'$ is very similar.

\subsection{Profunctors as bimodules}
  \label{ssec:profunctor_bimodule}

One can also think of a profunctor as a sort of graded bimodule: for each pair of objects
$c\in\cat{C}$ and $d\in\cat{D}$ there is a set $M(c,d)$ of elements in the bimodule, and given an
element $m\in M(c,d)$ and morphisms $f\colon c'\to c$ in $\cat{C}$ and $g\colon d\to d'$ in
$\cat{D}$, there are elements $g\cdot m\in M(c,d')$ and $m\cdot f\in M(c',d)$, such that the
equations $(g\cdot m)\cdot f=g\cdot(m\cdot f)$, $g'\cdot(g\cdot m)=(g'\circ g)\cdot m$, and $(m\cdot
f)\cdot f'=m\cdot(f\circ f')$ hold whenever they make sense.

\subsection{Representable profunctors}
  \label{sec:representable_profunctors}

Profunctors also act as generalized functors, just like relations $R\subseteq A\times B$ act as
generalized functions $A\to B$. Any functor $F\colon\cat{C}\to\cat{D}$ induces profunctors
$\cat{D}(F,-)\colon\cat{C}\tickar\cat{D}$ and $\cat{D}(-,F)\colon\cat{D}\tickar\cat{C}$, called the
profunctors \emph{represented} by $F$. These profunctors are defined by
\begin{equation}
    \label{repprofs}
  \cat{D}(F,-)(c,d)\coloneqq\cat{D}(Fc,d) \qquad \cat{D}(-,F)(d,c)\coloneqq\cat{D}(d,Fc).
\end{equation}

\subsection{Tensor product of profunctors}
  \label{ssec:profunctor_tensor}
Given two profunctors
\[
  \begin{tikzcd}
    \cat{C} \ar[r,tick,"M"] & \cat{D} \ar[r,tick,"N"] & \cat{E}
  \end{tikzcd}
\]
there is a tensor product $M\odot N\colon\cat{C}\tickar\cat{E}$, given by the coend formula
\begin{equation}
    \label{eqn:coendComp}
  (M\odot N)(c,e) = \int^{d\in\cat{D}}M(c,d)\times N(d,e).
\end{equation}
Following \cref{ssec:profunctor_matrix}, this is analogous to matrix multiplication:
$(AB)_{i,k}=\sum_jA_{i,j}B_{j,k}$. Equivalently, $(M\odot N)(c,e)$ is the coequalizer of the
diagram
\begin{equation}
    \label{eqn:coeqComp}
  \begin{tikzcd}
    \displaystyle\coprod_{d_1,d_2\in\cat{D}} M(c,d_1)\times\cat{D}(d_1,d_2)\times N(d_2,e)
    \ar[r,shift left] \ar[r,shift right]
    & \displaystyle\coprod_{d\in\cat{D}} M(c,d)\times N(d,e)
  \end{tikzcd}
\end{equation}
where the two maps are given by the right action of $\cat{D}$ on $M$ and by the left action of
$\cat{D}$ on $N$. In the notation of \cref{ssec:profunctor_bimodule}, we can write elements of
$(M\odot N)(c,e)$ as tensors $m\otimes n$, where $m\in M(c,d)$ and $n\in N(d,e)$ for some
$d\in\cat{D}$. The coequalizer then implies that $(m\cdot f)\otimes n=m\otimes(f\cdot n)$ whenever
the equation makes sense. Notice the similarity to the tensor product of bimodules over rings.

Alternatively, we can define the tensor product by the composition
\begin{equation*}
  M\odot N = \op{\cat{C}}\xrightarrow{\;M\;}\Set^{\cat{D}}\xrightarrow{\;\Lambda_N\;}\Set^{\cat{E}},
\end{equation*}
or by the composition $\Lambda'_N\circ M\colon\cat{C}\to\Set^{\op{\cat{E}}}$. This is
clearly equivalent to \cref{eqn:coendComp}, using \cref{eqn:LambdaPointwise}.

For any category $\cat{C}$, there is a profunctor
$\Hom_{\cat{C}}\colon\op{\cat{C}}\times\cat{C}\to\Set$, which we will often write as
$\cat{C}=\Hom_{\cat{C}}$ when unambiguous. For any functors $F\colon\cat{C}\to\Set$ and
$G\colon\op{\cat{C}}\to\Set$, there are natural isomorphisms
\begin{equation}
    \label{eq:coYoneda}
  \int^{c\in\cat{C}}F(c)\times\cat{C}(c,c')\iso F(c')
  \qquad
  \int^{c\in\cat{C}}\cat{C}(c',c)\times G(c) \iso G(c'),
\end{equation}
a result sometimes referred to as the coYoneda lemma \cite[(3.71)]{Kelly:1982a}. Continuing with the
analogy from \cref{ssec:profunctor_matrix}, $\Hom_{\cat{C}}$ acts like an identity matrix:
$\sum_i\delta_{i,j}v_i=v_j$. That is, these hom profunctors act as units for the tensor product,
since \cref{eq:coYoneda} shows that $\Hom_{\cat{C}}\odot M \iso M \iso M\odot\Hom_{\cat{D}}$.
Following \cref{ssec:profunctor_bimodule}, one can think of $\Hom_{\cat{C}}$ as the regular
$(\cat{C},\cat{C})$-bimodule, i.e.\ as $\cat{C}$ acting on itself on both sides
\cite{Matsumura:1989a}.

\subsection{Profunctor morphisms}
  \label{prof_morphisms}

A morphism $\phi\colon M\Rightarrow N$ between two profunctors
\[
  \begin{tikzcd}
    \cat{C} \ar[r,tick,shift left,"M"] \ar[r,tick,shift right,"N"'] & \cat{D},
  \end{tikzcd}
\]
is defined to be a natural transformation between the set-valued functors. In other words, for each
$c\in\cat{C}$ and $d\in\cat{D}$ there is a component function $\phi_{c,d}\colon M(c,d)\to N(c,d)$
such that the equation $\phi(f\cdot m \cdot g)=f\cdot\phi(m)\cdot g$ holds whenever it makes sense.

Categories, profunctors, and profunctor morphisms form a bicategory $\Prof$. To explain how functors
fit in, we need to discuss proarrow equipments.

\subsection{Proarrow equipments}
  \label{sec:equipments}

Before going into more properties of profunctors, it will be useful to put them in a more general
and abstract framework. A \emph{double category} is a 2-category-like structure involving two types
of 1-cell \dash horizontal and vertical \dash as well as 2-cells. A \emph{proarrow equipment} (which we
typically abbreviate to just \emph{equipment}) is a double category satisfying a certain fibrancy
condition. An excellent reference is the paper
\cite{Shulman:2008a}, where they are called \emph{framed bicategories}.

We will see in \cref{ex:prof_double_cat} that there is an equipment $\dProf$ whose objects are
categories, whose vertical 1-cells are functors, and whose horizontal 1-cells are profunctors. This is
the motivating example to keep in mind for equipments. In \cref{ThedoublecategoryData} we will
define $\Data$, the other main proarrow equipment of the paper, whose objects are database schemas.

\begin{definition}
    \label{def:double_cat}
  A \emph{double category} $\dcat{D}$ consists of the following data:
  \begin{itemize}[nosep]
  \item A category $\dcat{D}_0$, which we refer to as the \emph{vertical category} of $\dcat{D}$.
    For any two objects $A,B\in\dcat{D}_0$, we will write $\dcat{D}_0(A,B)$ for the set of vertical
    arrows from $A$ to $B$. We refer to objects of $\dcat{D}_0$ as objects of $\dcat{D}$.
  \item A category $\dcat{D}_1$, equipped with two functors
    $\lframe,\rframe\colon\dcat{D}_1\to\dcat{D}_0$, called the \emph{left frame} and \emph{right
    frame} functors. Given an object $M\in\Ob\dcat{D}_1$ with $A=\lframe(M)$ and $B=\rframe(M)$, we
    say that $M$ is a \emph{proarrow} (or \emph{horizontal arrow}) \emph{from $A$ to $B$} and write
    $M\colon A\tickar B$. A morphism $\phi\colon M\to N$ in $\dcat{D}_1$ is called a 2-cell, and is
    drawn as follows, where $f=\lframe(\phi)$ and $g=\rframe(\phi)$:
    \begin{equation} \begin{tikzcd}
        \label{eqn:2cell}
      A \ar[r,tick,"M" domA] \ar[d,"f"']
      & B\ar[d,"{g}"]
      \\
      C \ar[r,tick,"N"' codA]
      & D
      \twocellA{\phi}
    \end{tikzcd} \end{equation}
  \item A \emph{unit} functor $\unit\colon\dcat{D}_0\to\dcat{D}_1$, which is a section of both
    $L$ and $R$, i.e.\ $\lframe\circ\unit=\id_{\dcat{D}_0}=\rframe\circ \unit$. We will often write
    $\unit_A$ or even $A$ for the unit proarrow, $\unit(A)\colon A\tickar A$, and similarly
    $\unit_f$ of just $f$ for $\unit(f)$.
  \item A functor $\odot\colon\dcat{D}_1\times_{\dcat{D}_0}\dcat{D}_1\to\dcat{D}_1$, called
    \emph{horizontal composition}, that is weakly associative and weakly unital in the sense that there are
    coherent unitor and associator isomorphisms. See \cite{Shulman:2008a} for details.
  \end{itemize}
  Given a double category $\dcat{D}$, we will sometimes write $\Ver(\dcat{D})$ for the vertical
  category $\dcat{D}_0$. There is also a \emph{horizontal bicategory}, denoted $\HHor(\dcat{D})$,
  whose objects and 1-cells are the objects and horizontal 1-cells of $\dcat{D}$, and whose 2-cells
  are the 2-cells of $\dcat{D}$ of the form \cref{eqn:2cell} such that $f=\id_A$ and $g=\id_{B}$.

  Given $f,g,M,N$ as in \cref{eqn:2cell}, we write $\subsub{\dcat{D}}{f}{g}(M,N)$ for the set of 2-cells from
  $M$ to $N$ with frames $f$ and $g$, and write $\HHor(\dcat{D})(M,N)$ for the case where $f$ and
  $g$ are identity morphisms. If $A$ and $B$ are objects, then $\dcat{D}(A,B)$ will always mean the
  set of \emph{vertical} arrows from $A$ to $B$, where $\HHor(\dcat{D})(A,B)$ is used when we want
  the category of proarrows.
\end{definition}

We follow the convention of writing horizontal composition serially, i.e.\ the horizontal composite
of proarrows $M\colon A\tickar B$ and $N\colon B\tickar C$, is $M\odot N\colon A\tickar C$.

\begin{definition}
    \label{closedequipment}
  A double category $\dcat{D}$ is \emph{right closed} [resp.\ \emph{left closed}] when its
  horizontal bicategory is, i.e.\ when composing a proarrow $N$ [resp.\ $M$] with an arbitrary
  proarrow, $(-\odot N)$, [resp.\ $(M\odot -)$] has a left adjoint. Following \cite{Shulman:2008a},
  we denote this left adjoint by $(N\rhd -)$ [resp.\ by $(-\lhd M)$]; hence there are bijections
  \begin{align*}
    \HHor(\dcat{D})(X\odot N,P) &\cong \HHor(\dcat{D})(X,N\rhd P) \\
    \HHor(\dcat{D})(M\odot X,P) &\cong \HHor(\dcat{D})(X,P\lhd M)
  \end{align*}
  natural in $X$ and $P$. $\dcat{D}$ is \emph{biclosed} when both adjoints exist.
\end{definition}

Recall from
\cite{Borceux:1994a-2} the definitions of cartesian morphisms and fibrations of categories.

\begin{definition}
    \label{def:equipment}
  A \emph{proarrow equipment} (or just \emph{equipment}) is a double category $\dcat{D}$ in which the frame functor
  \[
    (\lframe,\rframe)\colon\dcat{D}_1\to\dcat{D}_0\times\dcat{D}_0
  \]
  is a fibration. If $f\colon A\to C$ and $g\colon B\to D$ are vertical morphisms and $N\colon
  C\tickar D$ is a proarrow, a cartesian morphism $M\to N$ in $\dcat{D}_1$ over $(f,g)$ is a 2-cell
  \[ \begin{tikzcd}
      A \ar[r,tick, "M" domA] \ar[d,"f"']
      & B\ar[d,"{g}"] \\
      C \ar[r,tick,"N"' codA]
      & D
      \twocellA{\tn{cart}}
  \end{tikzcd} \]
  which we call a \emph{cartesian 2-cell}. We refer to $M$ as the \emph{restriction of $N$ along $f$
  and $g$}, written $M=N(f,g)$.

  Equivalently, an equipment is a double category in which every vertical arrow $f\colon A\to B$ has a
  \emph{companion} $\comp{f}\colon A\tickar B$ and a \emph{conjoint} $\conj{f}\colon B\tickar A$,
  together with 2-cells satisfying certain equations (see \cite{Shulman:2008a}). In this view, the
  canonical cartesian lifting of some proarrow $N$ along $(f,g)$ is given by $N(f,g)\cong\comp{f}\odot
  N\odot\conj{g}$.
\end{definition}

\subsection{Adjunction between representable proarrows}
  \label{sec:adjunction_reps}

Any vertical morphism in an equipment $\dcat{D}$ induces an adjunction $\comp{f}\dashv\conj{f}$ in
the horizontal bicategory $\HHor(\dcat{D})$, with unit denoted $\eta_f$ and counit denoted
$\epsilon_f$. Moreover, the following bijective correspondences hold for any vertical morphisms
$f\colon A\to B$, $g\colon C\to D$, and proarrows $M\colon A\tickar B$, $N\colon C\tickar D$:
\begin{equation}
    \label{bijcorrs}
  \begin{aligned}
    \subsub{\dcat{D}}{f}{g}(M,N) &\iso \HHor(\dcat{D})(M, \comp{f}\odot N\odot\conj{g}) \\
                        &\iso \HHor(\dcat{D})(M\odot\comp{g}, \comp{f}\odot N) \\
                        &\iso \HHor(\dcat{D})(\conj{f}\odot M, N\odot\conj{g}) \\
                        &\iso \HHor(\dcat{D})(\conj{f}\odot M\odot\comp{g}, N).
  \end{aligned}
\end{equation}
The last bijection shows that in an equipment, the frame functor
$(\lframe,\rframe)\colon\dcat{D}_1\to\dcat{D}_0\times\dcat{D}_0$ turns out to also be an
opfibration.

We record some notation for \cref{bijcorrs}. Given a $2$-cell $\phi\in\subsub{\dcat{D}}{f}{g}(M,N)$, we write
$\comp{\phi}\in\HHor(\dcat{D})(M\odot\comp{g}, \comp{f}\odot N)$ and
$\conj{\phi}\in\HHor(\dcat{D})(\conj{f}\odot M, N\odot\conj{g})$ for its image under the above
bijections,
\begin{equation*}
  \begin{tikzcd}
    A \ar[r,tick,"M" domA] \ar[d,"f"']
    & B\ar[d,"{g}"]
    \\
    C \ar[r,tick,"N"' codA]
    & D
    \twocellA{\phi}
  \end{tikzcd}
  \qquad\qquad
  \begin{tikzcd}
    A \ar[r,tick,"M"] \ar[d,equal]
    & |[alias=domA]| B\ar[r,tick,"\comp{g}"]
    & D\ar[d,equal]
    \\
    A\ar[r,tick,"\comp{f}"']
    & |[alias=codA]| C \ar[r,tick,"N"']
    & D
    \twocellA{\comp{\phi}}
  \end{tikzcd}
  \qquad\qquad
  \begin{tikzcd}
    C \ar[r,tick,"\conj{f}"] \ar[d,equal]
    & |[alias=domA]| A\ar[r,"M"]
    & B\ar[d,equal]
    \\
    C\ar[r,"N"']
    & |[alias=codA]| D \ar[r,tick,"\conj{g}"']
    & B
    \twocellA{\conj{\phi}}
  \end{tikzcd}
\end{equation*}

\begin{example}
    \label{ex:prof_double_cat}
  There is a double category $\dProf$ defined as follows. The vertical category is $\dProf_0=\Cat$
  the category of small categories and functors. Given objects $\cat{C},\cat{D}\in\dProf$, a
  horizontal arrow between them is a profunctor $M\colon\cat{C}\tickar\cat{D}$, as described in
  \cref{ssec:profunctors}. A 2-cell $\phi\in\subsub{\dProf}{F}{G}(M,N)$, as to the left of
  \cref{eqn:Prof2cells}, denotes a natural transformation, as to the right of \cref{eqn:Prof2cells},
  with components $\phi_{c,d}\colon M(c,d)\to N(Fc,Gd)$:
  \begin{equation}
      \label{eqn:Prof2cells}
    \begin{tikzcd}
      \cat{C} \ar[r,tick,"M" domA] \ar[d,"F"']
      & \cat{D}\ar[d,"{G}"] \\
      \cat{E} \ar[r,tick,"N"' codA]
      & \cat{F}
      \twocellA{\phi}
    \end{tikzcd}
    \hspace{.6in}
    \begin{tikzcd}[column sep=.8em, row sep=5ex]
      \op{\cat{C}}\times\cat{D} \ar[dr,"M"'] \ar[rr,"\op{F}\times G"]
      & \ar[d,phantom,"\overset{\phi}{\Rightarrow}" near start]
      & \op{\cat{E}}\times\cat{F}\ar[dl,"N"] \\
      & \Set
    \end{tikzcd}
  \end{equation}
  The horizontal composite of profunctors $M\odot N$ is defined by the coend \cref{eqn:coendComp},
  or equivalently by the coequalizer \cref{eqn:coeqComp}, and the horizontal unit is
  $\unit_\cat{C}=\Hom_{\cat{C}}\colon\cat{C}\tickar\cat{C}$. This gives $\dProf$ the structure of a
  double category, such that $\HHor(\dProf)$ is the bicategory $\Prof$ defined in
  \cref{prof_morphisms}.

  Moreover, the double category $\dProf$ is biclosed (see \cref{closedequipment}): given proarrows
  $M\colon\cat{C}\tickar\cat{D}$, $N\colon\cat{D}\tickar\cat{E}$, and $P\colon\cat{C}\tickar\cat{E}$,
  one defines left and right exponentiation using ends
  \begin{align*}
    (N\rhd P)(c,d)= & \int_{e\in\cat{E}}P(c,e)^{N(d,e)}=\Fun{\cat{E},\Set}(N(d,-),P(c,-)) \\
    (P\lhd M)(d,e)= & \int_{c\in\cat{C}}P(c,e)^{M(c,d)}=\Fun{\op{\cat{C}},\Set}(M(-,d),P(-,e))
  \end{align*}
  which evidently inherit left and right actions from the respective categories when viewed as
  bimodules.

  Finally, $\dProf$ is an equipment because for any $F,G,N$ as in \cref{eqn:Prof2cells}, there is a
  cartesian 2-cell whose domain is precisely the profunctor $N(F,G)\coloneqq N\circ(\op{F}\times G)$ obtained by
  composition. The companion and conjoint of any functor $F\colon\cat{C}\to\cat{D}$ are the
  representable profunctors \cref{repprofs}
  \[
  \comp{F}=\cat{D}(F,-)\quad\tn{and}\quad\conj{F}=\cat{D}(-,F).
  \]
  Thus we can also represent the cartesian lifting
  as $N(F,G)=\comp{F}\odot N\odot\conj{G}$.
\end{example}

\begin{definition}
    \label{def:local_colimits}
  Let $\cat{I}$ be a small category. We say that a double category $\dcat{D}$ \emph{has local
  colimits of shape $\cat{I}$} if, for each pair of objects $A,B\in\dcat{D}$, the hom-category
  $\HHor(\dcat{D})(A,B)$ has colimits of shape $\cat{I}$ and these are preserved by horizontal
  composition on both sides,
  \begin{align*}
  L\odot\left(\colim_{i\in\cat{I}}M_i\right)\cong & \colim_{i\in\cat{I}}(L\odot M_i) \\
  \left(\colim_{i\in\cat{I}}M_i\right)\odot N\cong & \colim_{i\in\cat{I}}(M_i\odot N).
  \end{align*}
  We say that $\dcat{D}$ \emph{has local colimits} if it has local colimits of shape $\cat{I}$ for
  all small $\cat{I}$.
\end{definition}

\begin{example}
    \label{ex:Prof_local_colims}
  The equipment $\dProf$ has local colimits. Indeed, each horizontal bicategory is a category of
  set-valued functors. Colimits exist, and they are preserved by horizontal composition because
  composition is defined by coends, which are themselves colimits.
\end{example}

\subsection{Collage of a proarrow}
  \label{ssec:collage}
In some equipments $\dcat{D}$, a proarrow can be represented in a certain sense by an object in $\dcat{D}$,
called its collage. For example, it is well known that a profunctor can be represented by a category, as we
review in \cref{ex:prof_collages}. In this section we collect some useful properties of the collage construction,
in an arbitrary equipment.

We note briefly that the collage construction was also studied in \cite{Wood:1985a}, in a slightly
different setting. The definition we give below of an equipment with extensive collages is somewhat
more general than the set of axioms considered in \cite{Wood:1985a}, as we don't require the
existence of Kleisli objects for (horizontal) monads.

\begin{definition}\label{def:collage}
Let $M\colon A\tickar B$ be a proarrow in an equipment $\dcat{D}$. Its \emph{collage} is an object
$\scol{M}$ equipped with vertical arrows $\incnodes{A}\colon A\to\scol{M}\from B\cocolon\incnodes{B}$, called the
\emph{collage inclusions}, together with a 2-cell
\begin{equation}
    \label{eqn:mu}
  \begin{tikzcd}
    A \ar[r,tick,"M" domA] \ar[d,"\incnodes{A}"']
    & B \ar[d,"\incnodes{B}"] \\
    \scol{M} \ar[r,tick,"\scol{M}"' codA] & \scol{M},
    \twocellA{\mu}
  \end{tikzcd}
\end{equation}
that is \emph{universal} in the sense that any diagram as to the left below (a cocone under $M$)
factors uniquely as to the right:
\begin{equation}
    \label{univpropcollages}
  \begin{tikzcd}
    A \ar[r,tick,"M" domA] \ar[d,"f_A"']
    & B \ar[d,"f_B"]
    \\
    X \ar[r,tick,"X"']
    & X
    \twocellA{f}
  \end{tikzcd}
  \quad = \quad
  \begin{tikzcd}
    A \ar[r,tick,"M" domA] \ar[d,"\incnodes{A}"']
    & B \ar[d,"\incnodes{B}"]
    \\
    \scol{M} \ar[r,tick,"\scol{M}"' {domB,codA}] \ar[d,"\bar{f}"']
    & \scol{M} \ar[d,"\bar{f}"]
    \\
    X \ar[r,tick,"X"' codB]
    & X
    \twocellA{\mu}
    \twocellB[pos=.6]{\bar{f}}
  \end{tikzcd}
\end{equation}
\end{definition}

\begin{remark}
    \label{rem:collage_adjoint}
  The existence of a 2-cell $\mu$ with the above universal property amounts to the existence of a
  left adjoint $\scol{(-)}\colon\dcat{D}_1\to\dcat{D}_0$ to the unit functor $\unit$ from
  \cref{def:double_cat}, since it establishes a bijection
  $\dcat{D}_0(\scol{M},X)\cong\dcat{D}_1(M,\unit_X)$. From this perspective, the universal 2-cell
  $\mu:M\Rightarrow\unit_{\scol{M}}$, as in \cref{eqn:mu}, is the unit of the adjunction.
\end{remark}

\begin{definition}
    \label{normalcollages}
  An equipment $\dcat{D}$ is said to \emph{have collages} if every proarrow in $\dcat{D}$ has a
  collage as in \cref{univpropcollages}. By \cref{rem:collage_adjoint}, $\dcat{D}$ has collages if
  and only if there exists a left adjoint $\scol{(-)}\colon\dcat{D}_1\to\dcat{D}_0$ to the unit
  functor $\unit$.

  We say $\dcat{D}$ \emph{has normal collages} if additionally the unit of the adjunction $\mu$ is
  cartesian.
\end{definition}

\begin{example}
    \label{ex:prof_collages}
  The proarrow equipment $\dProf$ has normal collages. The collage $\scol{M}$ of a profunctor
  $M\colon\op{\cat{C}}\times\cat{D}\to\Set$ is a category where
  $\Ob(\scol{M})\coloneqq\Ob(\cat{C})\sqcup\Ob(\cat{D})$, and
  \begin{equation}
      \label{eq:prof_collage}
    \scol{M}(x,y) =
    \begin{cases}
      \cat{C}(x,y) & \text{if $x\in\cat{C}$ and $y\in\cat{C}$} \\
      M(x,y) & \text{if $x\in\cat{C}$ and $y\in\cat{D}$} \\
      \emptyset & \text{if $x\in\cat{D}$ and $y\in\cat{C}$}\\
      \cat{D}(x,y) & \text{if $x\in\cat{D}$ and $y\in\cat{D}$}
    \end{cases}
  \end{equation}
  Composition in $\scol{M}$ is defined using composition in $\cat{C}$ and $\cat{D}$ and the
  functoriality of $M$. There are evident functors $\incnodes{\cat{C}}\colon\cat{C}\to\scol{M}$ and
  $\incnodes{\cat{D}}\colon\cat{D}\to\scol{M}$, and the 2-cell $\mu\colon M\Rightarrow\unit_{\scol{M}}$ sends an element $m\in M(c,d)$ to
  $m\in\scol{M}(\incnodes{\cat{C}}(c),\incnodes{\cat{D}}(d))=M(c,d)$. It is easy to see that $\mu$ is cartesian, 
  so $\dProf$ has normal collages.

  This construction satisfies the universal property \cref{univpropcollages}. Suppose we are given
  $f_{\cat{C}}\colon\cat{C}\to\cat{X}$, $f_{\cat{D}}\colon\cat{D}\to\cat{X}$, and a 2-cell $f$ as in
  \cref{univpropcollages}. It is easy to see that the unique $\bar{f}\colon\scol{M}\to\cat{X}$ (and
  so $\unit_{\bar{f}}\colon\unit_{\scol{M}}\Rightarrow\unit_{\cat{X}}$) that works is defined by
  cases, using $f_{\cat{C}}$ on objects and morphisms in $\cat{C}$, using $f_{\cat{D}}$ on objects
  and morphisms in $\cat{D}$, and using $f$ on morphisms with domain in $\cat{C}$ and codomain in
  $\cat{D}$.

  Note also that for any profunctor $M$ as above, there is an induced functor $\scol{M}\to\ncat{2}$, where
  $\ncat{2}=\{0\to1\}$, sometimes called the \emph{free arrow category}, is the collage of the terminal profunctor
  $\singleton\tickar\singleton$. In fact, if $\Cat/\ncat{2}$ denotes the slice category, it is not hard to
  check that the collage construction provides an equivalence of categories 
  \begin{equation}\label{equiv_prof_cat2}
  \dProf_1\simeq\Cat/\ncat{2}
  \end{equation}
  In particular, from a functor $F\colon\cat{A}\to\ncat{2}$ we obtain a profunctor
  between the pullbacks of $F$ along $0,1\colon\{*\}\to\ncat{2}$ respectively.
\end{example}

\subsection{Proarrows between collages; simplices}
We now want to consider general proarrows $\scol{M}\tickar\scol{N}$ between collages in $\dcat{D}$,
by defining a category of simplices. Although we will only need this in the case $\dcat{D}=\dProf$,
we found the proofs simpler in the general case.

For intuition, consider two profunctors $M\colon\cat{C}_0\tickar\cat{C}_1$ and
$N\colon\cat{D}_0\tickar\cat{D}_1$. A profunctor $X\colon\scol{M}\tickar\scol{N}$ must assign a set
$X(c,d)$ in four different cases: $c$ is an object in either $\cat{C}_0$ or $\cat{C}_1$, and
likewise for $d$. We could try splitting $X$ into four profunctors
$X_{i,j}\colon\cat{C}_i\tickar\cat{D}_j$, but this would not encode all of the functorial actions
needed to recover $X$. For instance, given objects $c\in\cat{C}_0$, $c'\in\cat{C}_1$, and
$d\in\cat{D}_0$, and given an element $x\in X_{1,0}(c',d)$ and a morphism $m\colon c\to c'$ in
$\scol{M}$ (i.e.\ an element $m\in M(c,c')$), there is an element $m\cdot x\in X_{0,0}(c,d)$. The
idea behind the following construction is to encode all of the data of a profunctor $X$ between
collage objects by four profunctors, together with four 2-cells which capture all of those
functorial actions.

\begin{definition}
    \label{def:simplices}
  Let $M\colon A_0\tickar A_1$ and $N\colon B_0\tickar B_1$ be proarrows in $\dcat{D}$. We define an
  $(M,N)$-\emph{simplex} $X$ to be a collection of proarrows $\{X_{0,0},X_{0,1},X_{1,0},X_{1,1}\}$
  \[
    \begin{tikzcd}[row sep=2.6em,column sep=3.6em]
      A_1 \ar[r,tick,"X_{1,0}"] \ar[dr,"X_{1,1}"{pos=.8,inner sep=0ex}]
      & B_0 \ar[d,tick,"N"] \\
      A_0 \ar[ur,tick,crossing over,"X_{0,0}"{pos=.2,inner sep=0ex}]
      \ar[r,tick,"X_{0,1}"'] \ar[u,tick,"M"]
      & B_1
    \end{tikzcd}
  \]
  together with four 2-cells $X_{0,*},X_{1,*},X_{*,0},X_{*,1}$ as in
  \begin{equation*}
    \begin{tikzcd}[row sep={between origins,2em},column sep={between origins,5em}]
      A_1 \ar[dr,tick,"X_{1,k}"] & \\
      & |[alias=codA]| B_k \\
      A_0 \ar[uu,tick,"M" domA] \ar[ur,tick,"X_{0,k}"'] &
      \twocellA[pos=.35]{X_{*,k}}
    \end{tikzcd}
    \qquad
    \begin{tikzcd}[row sep={between origins,2em},column sep={between origins,5em}]
      & B_0 \ar[dd,tick,"N" codA] \\
      |[alias=domA]| A_j \ar[ur,tick,"X_{j,0}"] \ar[dr,tick,"X_{j,1}"'] & \\
      & B_1
      \twocellA[pos=.65]{X_{j,*}}
    \end{tikzcd}
  \end{equation*}
  such that the following equation holds:
  \[
    \begin{tikzcd}[row sep=2.6em,column sep=3.6em]
      |[alias=domA]| A_1 \ar[r,tick,"X_{1,0}"]
      & B_0 \ar[d,tick,"N"] \\
      A_0 \ar[ur,tick,"X_{0,0}"{pos=.25,inner sep=0ex,codA,domB}]
      \ar[r,tick,"X_{0,1}"'] \ar[u,tick,"M"]
      & |[alias=codB]| B_1
      \twocellA[pos=.65]{X_{*,0}}
      \twocellB[pos=.35]{X_{0,*}}
    \end{tikzcd}
    \quad=\quad
    \begin{tikzcd}[row sep=2.6em,column sep=3.6em]
      A_1 \ar[r,tick,"X_{1,0}"] \ar[dr,tick,"X_{1,1}"{pos=.75,inner sep=0ex,codA,domB}]
      & |[alias=domA]| B_0 \ar[d,tick,"N"] \\
      |[alias=codB]| A_0 \ar[r,tick,"X_{0,1}"'] \ar[u,tick,"M"]
      & B_1
      \twocellA[pos=.65]{X_{1,*}}
      \twocellB[pos=.35]{X_{*,1}}
    \end{tikzcd}
  \]
  \end{definition}

A morphism $\alpha\colon X\to Y$ between two $(M,N)$-simplices consists of component 2-cells
$\alpha=(\alpha_{0,0},\alpha_{0,1},\alpha_{1,0},\alpha_{1,1})$, where $\alpha_{j,k}\colon
X_{j,k}\to Y_{j,k}$ satisfy four evident equations. We have thus defined the \emph{category of
$(M,N)$-simplices}, denoted $\Simp{M}{N}$.

Suppose that the equipment $\dcat{D}$ has local initial objects; see \cref{def:local_colimits}.
Then for any proarrow $M\colon A_0\tickar A_1$, there is an $(M,M)$-simplex given by the proarrows
\begin{equation} \label{eq:unit_simplex}
  \begin{tikzcd}[row sep=2.6em,column sep=3.6em]
    A_1 \ar[r,tick,"0"] \ar[dr,"A_1"{pos=.75,inner sep=0ex}]
    & A_0 \ar[d,tick,"M"] \\
    A_0 \ar[ur,tick,crossing over,"A_0"{pos=.25,inner sep=0ex}]
    \ar[r,tick,"M"'] \ar[u,tick,"M"]
    & A_1
  \end{tikzcd}
\end{equation}
together with the evident 2-cells; we call this the \emph{unit simplex} on $M$ and denote it by
$1_M\in\Simp{M}{M}$.

\subsection{The functor \texorpdfstring{$\res{M}{N}$}{S}} \label{res_functor}

There is a functor $\res{M}{N}\colon\HHor(\dcat{D})(\scol{M},\scol{N})\to\Simp{M}{N}$ defined as
follows. On some $P\colon\scol{M}\tickar\scol{N}$, the four proarrows are given by the
restrictions along the collage inclusions $\incnodes{A_j}\colon A_j\to\scol{M}$ and $\incnodes{B_k}\colon
B_k\to\scol{N}$, namely $X_{j,k}=\comp{i}_{A_j}\odot P\odot\conj{i}_{B_k}$, and the 2-cells are
given by horizontal composition with the universal $\mu_M$, $\mu_N$.

The following proposition follows directly from definitions.

\begin{proposition}
    \label{prop:nice_units}
  Suppose that $\dcat{D}$ has local initial objects and collages. The four 2-cells
    \begin{equation}
        \label{nice_unit_condition1}
      \begin{tikzcd}
        A \ar[r,tick,"A" domA] \ar[d,"\incnodes{A}"']
        & A \ar[d,"\incnodes{A}"] \\
        \scol{M} \ar[r,tick,"\scol{M}"' codA]
        & \scol{M}
        \twocellA{\incnodes{A}}
      \end{tikzcd}
      \quad
      \begin{tikzcd}
        A \ar[r,tick,"M" domA] \ar[d,"\incnodes{A}"']
        & B \ar[d,"\incnodes{B}"] \\
        \scol{M} \ar[r,tick,"\scol{M}"' codA]
        & \scol{M}
        \twocellA{\mu}
      \end{tikzcd}
      \quad
      \begin{tikzcd}
        B \ar[r,tick,"0" domA] \ar[d,"\incnodes{B}"']
        & A \ar[d,"\incnodes{A}"] \\
        \scol{M} \ar[r,tick,"\scol{M}"' codA]
        & \scol{M}
        \twocellA{!}
      \end{tikzcd}
      \quad
    \begin{tikzcd}
        B \ar[r,tick,"B" domA] \ar[d,"\incnodes{B}"']
        & B \ar[d,"\incnodes{B}"] \\
        \scol{M} \ar[r,tick,"\scol{M}"' codA]
        & \scol{M}
        \twocellA{\incnodes{B}}
        \end{tikzcd}
    \end{equation}
  induce a morphism $u_M\colon 1_M\to\res{M}{M}(\unit_{\scol{M}})$ in $\Simp{M}{M}$ by unique
  factorization through cartesian 2-cells. The following are equivalent
  \begin{enumerate}
    \item $u_M$ is an isomorphism in $\Simp{M}{M}$.
    \item each of the four squares in \cref{nice_unit_condition1} is cartesian.
    \item the four induced 2-cells are isomorphisms:
    \begin{equation}
        \label{nice_unit_condition}
      \eta_{\incnodes{A}}\colon\unit_A\simrightarrow\comp{i}_A\odot\conj{i}_A,\quad
      \mu\colon M\simrightarrow\comp{i}_A\odot\conj{i}_B,\quad
      !\colon 0\simrightarrow\comp{i}_B\odot\conj{i}_A,\quad
        \eta_{\incnodes{B}}\colon\unit_B\simrightarrow\comp{i}_B\odot\conj{i}_B.
    \end{equation}
  \end{enumerate}
\end{proposition}

Note that if $\dcat{D}$ satisfies the equivalent conditions in \cref{prop:nice_units} then, in
particular, it has normal collages.

\begin{definition}
    \label{def:extensive}
  Let $\dcat{D}$ be an equipment. We will say that $\dcat{D}$ has \emph{extensive collages} if it
  satisfies the following conditions:
  \begin{enumerate}
  \item\labelold[condition]{item:extensive_def_1} $\dcat{D}$ has collages and local initial objects,
  \item\labelold[condition]{item:extensive_def_2} any of the equivalent conditions from
    \cref{prop:nice_units} are satisfied,
  \item\labelold[condition]{item:extensive_def_3} for every pair of proarrows $M$ and $N$, the
    functor $\res{M}{N}\colon\HHor(\dcat{D})(\scol{M},\scol{N})\to\Simp{M}{N}$ is an equivalence of
    categories.
  \end{enumerate}
\end{definition}

Extensive collages are best behaved in the presence of local finite colimits. The following
proposition provides a condition which is equivalent to \cref{item:extensive_def_3} above in this
case, but which is often easier to verify. The proof provides an explicit construction of the
inverse of $\res{M}{N}$ using colimits in the horizontal bicategories.

\begin{proposition}
    \label{prop:inverse_equiv}
  Suppose that $\dcat{D}$ is an equipment with collages, that it satisfies \cref{item:extensive_def_2}
  in \cref{def:extensive}, and that $\dcat{D}$ has local finite colimits (so it also satisfies
  \cref{item:extensive_def_1}). Then \cref{item:extensive_def_3} is equivalent to the following
  condition:
  \begin{enumerate}[label=\arabic*'.,ref=\arabic*',start=3,labelindent=\parindent,leftmargin=*]
  \item for any proarrow $M\colon A\tickar B$, the following square is a pushout in
  $\HHor(\dcat{D})(\scol{M},\scol{M})$:
    \begin{equation}\label{eq:collage_pushout}
      \begin{tikzcd}[column sep=5.5em,row sep=2.2em]
       \conj{i}_A\odot\comp{i}_A\odot\conj{i}_B\odot\comp{i}_B \ar[r,"\epsilon_{\incnodes{A}}\odot\conj{i}_B\odot\comp{i}_B"]
        \ar[d,"\conj{i}_A\odot\comp{i}_A\odot\epsilon_{\incnodes{B}}"']
        & \conj{i}_B\odot\comp{i}_B \ar[d,"\epsilon_{\incnodes{B}}"] \\
        \conj{i}_A\odot\comp{i}_A \ar[r,"\epsilon_{\incnodes{A}}"']
        & \unit_{\scol{M}}\ar[ul,phantom,very near start,"\ulcorner"]
      \end{tikzcd}
    \end{equation}
    \labelold[condition]{item:extensive_def_3_alt}
  \end{enumerate}
\end{proposition}
\begin{proof}
  Suppose $\dcat{D}$ has local finite colimits and satisfies \cref{item:extensive_def_2}. First
  assuming \cref{item:extensive_def_3} we will show that \cref{eq:collage_pushout} is a pushout. It
  suffices that its image under the equivalence $\res{M}{N}$ (\cref{res_functor}) is a pushout, i.e.\
  each of the four restriction functors,
  \[\comp{i}_A\odot\textrm{--}\odot\conj{i}_A\colon\HHor(\dcat{D})(\scol{M},\scol{M})\to\HHor(\dcat{D})(A,A),\]
  as well as $\comp{i}_A\odot\textrm{--}\odot\conj{i}_B$, $\comp{i}_B\odot\textrm{--}\odot\conj{i}_A$,
  and $\comp{i}_B\odot\textrm{--}\odot\conj{i}_B$, take the diagram
  \cref{eq:collage_pushout} to a pushout square. This follows easily from
  \cref{item:extensive_def_2}, in particular the four isomorphisms of \cref{nice_unit_condition}.

  Conversely, assuming \cref{item:extensive_def_3_alt}, we will show that $\res{M}{N}$ is an
  equivalence of categories for any pair of proarrows $M\colon A_0\tickar A_1$, $N\colon B_0\tickar
  B_1$. To define the inverse functor, let $X\in\Simp{M}{N}$ be a simplex, and consider the diagram
  \[
    \begin{tikzcd}[row sep=1em,column sep=2em]
      &A_1 \ar[rr,tick,"X_{1,0}"] \ar[ddrr,"X_{1,1}"{pos=.8,inner sep=0ex}]
      && B_0 \ar[dd,tick,"N"]\ar[rd,tick,"\comp{i}_{B_0}"] \\
      \scol{M}\ar[ru,tick,"\conj{i}_{A_1}"]\ar[rd,tick,"\conj{i}_{A_0}"']&&&&\scol{N}\\
      &A_0 \ar[uurr,tick,crossing over,"X_{0,0}"{pos=.2,inner sep=0ex}]
      \ar[rr,tick,"X_{0,1}"'] \ar[uu,tick,"M"]
      && B_1\ar[ru,tick,"\comp{i}_{B_1}"']
    \end{tikzcd}
  \]
  which also contains six 2-cells:
  \begin{gather*}
  X_{*,k}\colon M\odot X_{1,k}\to X_{0,k},\quad
  X_{j,*}\colon X_{j,0}\odot N\to X_{j,1},\\
  \conj{\mu}_M\colon\conj{i}_{A_0}\odot M\to\conj{i}_{A_1},\quad
  \comp{\mu}_N\colon N\odot\comp{i}_{B_1}\to\comp{i}_{B_0}
  \end{gather*}
  where the $\mu$'s are universal 2-cells and $\comp{\mu}$ and $\conj{\mu}$ are as in
  \cref{sec:adjunction_reps}.

  The inverse to $\res{M}{N}$, which we denote
  $(X\mapsto\bcol{X})\colon\Simp{M}{N}\to\HHor(\dcat{D})(\scol{M},\scol{N})$, is given by sending the
  simplex $X$ to the colimit in $\HHor(\dcat{D})(\scol{M},\scol{N})$
  of the $3\times 3$ square: 
  \footnote{We suppress the $\odot$ symbol in the objects to reduce the required space.}
  \begin{equation}
      \label{eq:simplex_collage}
    \begin{tikzcd}
      \conj{i}_{A_0}X_{0,0}\comp{i}_{B_0}
      & \conj{i}_{A_0}MX_{1,0}\comp{i}_{B_0} \ar[l,"X_{*,0}"'] \ar[r,"\conj{\mu}_M"]
      & \conj{i}_{A_1}X_{1,0}\comp{i}_{B_0}
      & P\conj{i}_{B_0}\comp{i}_{B_0} \\
      \conj{i}_{A_0}X_{0,0}N\comp{i}_{B_1} \ar[u,"\comp{\mu}_N"] \ar[d,"X_{0,*}"']
      & \conj{i}_{A_0}MX_{1,0}N\comp{i}_{B_1} \ar[u,"\comp{\mu}_N"']
      \ar[d,"X_{1,*}"] \ar[l,"X_{*,0}"'] \ar[r,"\conj{\mu}_M"]
      & \conj{i}_{A_1}X_{1,0}N\comp{i}_{B_1} \ar[u,"\comp{\mu}_N"'] \ar[d,"X_{1,*}"]
      & P\conj{i}_{B_0}N\comp{i}_{B_1} \ar[u,"\comp{\mu}_N"']
      \ar[d,"\conj{\mu}_N"] \\
      \conj{i}_{A_0}X_{0,1}\comp{i}_{B_1}
      & \conj{i}_{A_0}MX_{1,1}\comp{i}_{B_1} \ar[l,"X_{*,1}"] \ar[r,"\conj{\mu}_M"']
      & \conj{i}_{A_1}X_{1,1}\comp{i}_{B_1}
      & P\conj{i}_{B_1}\comp{i}_{B_1} \\
      \conj{i}_{A_0}\comp{i}_{A_0}P
      & \conj{i}_{A_0}M\comp{i}_{A_1}P \ar[l,"\comp{\mu}_M"]
      \ar[r,"\conj{\mu}_M"']
      & \conj{i}_{A_1}\comp{i}_{A_1}P
      & P
    \end{tikzcd}
  \end{equation}
  Note that this colimit can be formed by first taking the pushout of each row, and then taking the
  pushout of the resulting span, or by taking column-wise pushouts first. For the time being, ignore
  the separated right-hand column and bottom row of \cref{eq:simplex_collage}.

  We now show that $\res{M}{N}$ and $X\mapsto\bcol{X}$ are inverse equivalences. Suppose
  $P\colon\scol{M}\tickar\scol{N}$ is a proarrow and let $X=\res{M}{N}(P)$; we want to show that
  there is a natural isomorphism $P\cong\bcol{X}$. Performing the substitution
  $X_{j,k}=\comp{i}_{A_j}P\conj{i}_{B_k}$ and using the isomorphisms from
  \cref{nice_unit_condition}, e.g.\ $M\iso\comp{i}_{A_0}\conj{i}_{A_1}$, each row (resp.\ each
  column) can be seen as a composition of some proarrow \dash namely the one in the right-hand column
  (resp.\ bottom row) \dash with the diagram \cref{eq:collage_pushout}. Since local colimits commute
  with proarrow composition, the right-hand column (resp.\ bottom row) proarrows are indeed the pushouts. In
  the same way, one checks that $P$ is the colimit of both the right-hand column and the bottom row.

  In the other direction, if $X\in\Simp{M}{N}$ is any simplex and $\bcol{X}$ is the colimit of the
  square in \cref{eq:simplex_collage}, we want to show that $\res{M}{N}(\bcol{X})\iso X$. It is
  straightforward to check that $\comp{i}_{A_j}\odot\bcol{X}\odot\conj{i}_{B_k}\iso X_{j,k}$ by composing the
  square with $\comp{i}_{A_j}$ on the left and $\conj{i}_{B_k}$ on the right and applying the
  equations of \cref{nice_unit_condition}. It is moreover easy to see that these isomorphisms form
  the components of an isomorphism of simplices $\res{M}{N}(\bcol{X})\iso X$. Thus $\res{M}{N}$ is
  an equivalence of categories.
\end{proof}

\begin{remark}
  It is likely possible to characterize equipments with extensive collages (assuming local finite
  colimts) in terms of an adjunction of double categories. We won't pursue this further here, but
  for the interested reader we provide a rough sketch as a starting point for further investigation.

  If $\dcat{D}$ is an equipment with local finite colimts, one can define an equipment
  $\ndcat{S}{imp}(\dcat{D})$ whose vertical category is $\dcat{D}_1$ and whose horizontal 1-cells
  are simplices. The composition in $\ndcat{S}{imp}(\dcat{D})$ is given by
  \cref{eq:simplex_composition}. There is a double functor
  $\unit\colon\dcat{D}\to\ndcat{S}{imp}(\dcat{D})$ sending each object $A\in\dcat{D}$ to the unit
  proarrow $\unit_A$ and each proarrow $M\colon A\tickar B$ to the unit simplex $1_M$ defined in
  \cref{eq:unit_simplex}.

  If $\dcat{D}$ has extensive collages, then $\unit$ has a left adjoint $\nfun{Col}$ sending each
  proarrow $M\in\ndcat{S}{imp}(\dcat{D})$ to its collage $\Col{M}$ and acting on simplices by the
  pushout \cref{eq:simplex_collage}. Looking at the definition \cref{def:extensive}, it seems that
  \cref{item:extensive_def_1} is related to the existence of a left adjoint to $\unit$,
  \cref{item:extensive_def_2} is related to the property that the 2-cell components of the unit of
  this double-adjunction are cartesian, and \cref{item:extensive_def_3_alt} is related to the property
  that the right adjoint $\nfun{Col}$ is normal (preserves unit proarrows). Perhaps this observation
  can be worked into an equivalent charaterization of equipments with extensive collages, but we
  leave it to the motivated reader to investigate further.
\end{remark}

\begin{example}
    \label{Profextensivecollages}
  The equipment $\dProf$ has extensive collages. Indeed, $\dProf$ has local colimits by
  \cref{ex:Prof_local_colims} and normal collages by \cref{ex:prof_collages}.
  Moreover, we will verify that $\dProf$ satisfies \cref{item:extensive_def_3_alt} of
  \cref{prop:inverse_equiv}.

  If $M\colon\cat{C}\tickar\cat{D}$ is a profunctor, then we need to show that
  \cref{eq:collage_pushout} is a pushout in the category $\Fun{\op{\scol{M}}\times\scol{M},\Set}$.
  It suffices to show that it is a pointwise pushout. For any objects $x,y\in\scol{M}$, it is not
  hard to see that \cref{eq:collage_pushout} becomes one of the following pushout squares in $\Set$:
  \begin{equation*}
    \begin{tikzcd}[row sep=small,column sep=small]
      & {} \ar[r,phantom,"y\in\cat{C}"] & {}
      & {} \ar[r,phantom,"y\in\cat{D}"] & {} \\
      {} \ar[d,phantom,"x\in\cat{C}"] &
      0 \ar[r] \ar[d] \ar[dr,phantom,"\ulcorner"pos=1] & 0 \ar[d]
      & M(x,y) \ar[r] \ar[d] \ar[dr,phantom,"\ulcorner"pos=1] & M(x,y) \ar[d] \\
      {} &
      \cat{C}(x,y) \ar[r] & \cat{C}(x,y)
      & M(x,y) \ar[r] & M(x,y) \\
      {} \ar[d,phantom,"x\in\cat{D}"] &
      0 \ar[r] \ar[d] \ar[dr,phantom,"\ulcorner"pos=.86] & 0 \ar[d]
      & 0 \ar[r] \ar[d] \ar[dr,phantom,"\ulcorner"pos=1] & \cat{D}(x,y) \ar[d] \\
      {} &
      0 \ar[r] & 0
      & 0 \ar[r] & \cat{D}(x,y)
    \end{tikzcd}
  \end{equation*}
\end{example}

\subsection{Collages as lax (co)limits}

When an equipment $\dcat{D}$ has extensive collages and local finite colimits
(like $\dProf$), there is another
universal property involving collages, which can be
expressed entirely in terms of the horizontal bicategory $\HHor(\dcat{D})$.

\begin{definition}
  Let $\ccat{B}$ be a bicategory, let $F\colon A\to B$ be a 1-cell in $\ccat{B}$, and let $X$ be an
  object in $\ccat{B}$. Define a category of \emph{lax cocones from $F$ to $X$}, written
  $\Cocone{F}{X}$, as follows: an object of $\Cocone{F}{X}$ is a diagram
  \[
  \begin{tikzcd}[row sep=tiny]
    B \ar[dr,"P_B"] & \\
                    & |[alias=codA]| X \\
    A \ar[uu,"F"domA] \ar[ur,"P_A"'] &
    \twocellA[pos=.35]{\pi}
  \end{tikzcd}
  \]
  and a morphism $\alpha\colon(P_A,P_B,\pi)\to(Q_A,Q_B,\chi)$ is a pair of 2-cells $\alpha_A\colon P_A\to
  Q_A$ and $\alpha_B\colon P_B\to Q_B$ making an evident diagram commute.

  Any cocone $(P_A,P_B,\pi)\in\Cocone{F}{X}$ induces a functor $\ccat{B}(X,Y)\to\Cocone{F}{Y}$ by
  composition. If this functor is an equivalence of categories, then we say that $X$ is a \emph{lax colimit} of
  the arrow $F$ (see for example \cite{Kelly:1989a}). Dually,
  there is a category $\Cone{X}{F}$ of \emph{lax cones from $X$ to $F$}, employed in the definition of \emph{lax limits}
  of arrows.
\end{definition}

\begin{proposition}
    \label{prop:collage_lax_limit}
  Let $\dcat{D}$ be an equipment with extensive collages and local finite colimits, and let
  $M\colon A\tickar B$ be a proarrow with collage $\incnodes{A}\colon A\to\scol{M}\from B\cocolon \incnodes{B}$.
  The triangle on the left exhibits $\scol{M}$ as a lax colimit of the 1-cell
  $M$ in $\HHor(\dcat{D})$, and the triangle on the right exhibits $\scol{M}$ as a lax limit of $M$.
  \[
  \begin{tikzcd}[row sep=tiny]
    B \ar[dr,tick,"\comp{i}_B"] & \\
    & |[alias=codA]| \scol{M} \\
    A \ar[uu,tick,"M"domA] \ar[ur,tick,"\comp{i}_A"'] &
    \twocellA[pos=.35]{\comp{\mu}}
  \end{tikzcd}
  \qquad
  \begin{tikzcd}[row sep=tiny]
    & A \ar[dd,tick,"M"codA] \\
    |[alias=domA]| \scol{M} \ar[ur,tick,"\conj{i}_A"] \ar[dr,tick,"\conj{i}_B"'] & \\
    & B
    \twocellA[pos=.65]{\conj{\mu}}
  \end{tikzcd}
  \]
\end{proposition}
\begin{proof}
  The 2-cells $\comp{\mu}$, $\conj{\mu}$ correspond to the cartesian $\mu$ as in
  \cref{sec:adjunction_reps}. We will show that the triangle on the left is a lax colimit cocone,
  i.e.\ that composing with $\comp{\mu}$ induces an equivalence of categories
  $\HHor(\dcat{D})(\scol{M},Y)\to\Cocone{M}{Y}$ for any $Y$.
  We define the inverse functor to send a cocone $(P_A,P_B,\pi)$ to the proarrow
  $P\colon\scol{M}\tickar Y$ defined by a pushout in $\HHor(\dcat{D})(\scol{M},Y)$:
  \begin{equation}
    \label{eq:cocone_pushout}
    \begin{tikzcd}
      \conj{i}_A\odot M\odot P_B \ar[r,"\conj{\mu}\odot P_B"] \ar[d,"\conj{i}_A\odot\pi"']
      \ar[rd,phantom,"\ulcorner" very near end]
      & \conj{i}_B\odot P_B \ar[d] \\
      \conj{i}_A\odot P_A \ar[r] & P
    \end{tikzcd}
  \end{equation}
  Suppose we start with an arbitrary proarrow $Q\colon\scol{M}\tickar Y$, and compose with
  $\comp{\mu}$ to get the cocone $\pi=\comp{\mu}\odot Q\colon M\odot\comp{i}_B\odot Q\to\comp{i}_A\odot Q$. We can
  see that the pushout \cref{eq:cocone_pushout} is just \cref{eq:collage_pushout} composed by $Q$ on
  the right, showing $P\iso Q$. On the other hand, if we start with an arbitrary cocone $\pi$,
  take the pushout $P$ as in \cref{eq:cocone_pushout}, then compose on the left with 
  $\comp{\mu}\colon M\odot\comp{i}_B\to\comp{i}_A$, it is easy to check that we get $\pi$ back.

  Thus the pushout \cref{eq:cocone_pushout} does define an inverse
  functor $\Cocone{M}{Y}\to\HHor(\dcat{D})(\scol{M},Y)$, showing that the triangle on the left is a
  lax colimit cocone.  The lax limit cone follows by a dual argument.
\end{proof}

\begin{remark}
  A converse to \cref{prop:collage_lax_limit} holds: if $\dcat{D}$ has local finite
  colimits such that the conclusion to \cref{prop:collage_lax_limit} holds for all
  proarrows $M\colon A\tickar B$ in $\dcat{D}$, then $\dcat{D}$ has extensive collages. We won't
  need this converse, and so do not prove it. The proof is straightforward, regarding a simplex as a
  ``lax cocone of lax cones'' (or visa-versa).
\end{remark}

\begin{remark}
    \label{proarrowsinoutcollage}
  For convenience, we will break down the universal property of $\mcol{M}$
  as the lax limit of $M$. Suppose $\dcat{D}$ has extensive collages.

  Given any $P_A\colon X\tickar A$, $P_B\colon X\tickar B$, and 2-cell
  $\pi\colon P_A\odot M\to P_B$, there is a proarrow $P\colon X\tickar\mcol{M}$ (which is unique
  up to isomorphism by the
  2-dimensional part of the universal property of \cref{prop:collage_lax_limit}) such that
  $\pi\cong P\odot\conj{\mu}$. Namely cartesian 2-cells exist, by $P_A\cong P\odot\conj{i}_A$, $P_B\cong
  P\odot\conj{i}_B$, satisfying the equation (where $\mu$ is also cartesian)
  \begin{equation}\label{1-dim_univ_prop}
    \begin{tikzcd}
      X \ar[r,tick,"P_A" domA] \ar[d,equal]
      & A \ar[r,tick,"M" domB] \ar[d,"\incnodes{A}"]
      & B \ar[d,"\incnodes{B}"] \\
      X \ar[r,tick,"P"' codA]
      & \scol{M} \ar[r,tick,"\scol{M}"' codB]
      & \scol{M}
      \twocellalt{A}{\mathrm{cart}}
      \twocellB{\mu}
    \end{tikzcd}
    \quad = \quad
    \begin{tikzcd}[row sep={}]
      & |[alias=domA]| A \ar[dr,bend left=15,tick,"M"] & \\[between origins,1.5em]
      X \ar[ur,bend left=15,tick,"P_A"]
      \ar[rr,bend right=15,tick,"P_B"' {codA,domB}]
      \ar[d,equal]
      && B \ar[d,"\incnodes{B}"] \\[1.8em]
      X \ar[rr,tick,"P"' codB]
      && \scol{M}
      \twocellA{\pi}
      \twocellalt{B}{\mathrm{cart}}
    \end{tikzcd}
  \end{equation}
  The 2-dimensional part of the universal property says that, given $\alpha_A\colon p_A\to q_A$ and
  $\alpha_B\colon p_B\to q_B$ such that $\alpha_B\circ p=q\circ\alpha_A$, there is a unique
  $\alpha\colon P\to Q$ making the evident diagrams commute.
  
  The universal property for the lax colimit is dual.
\end{remark}

\section{Algebraic theories}
    \label{sec:algebraic_theories}

In this section, we recall some basic aspects of the well-known work on algebraic theories and their
algebras \cite{Adamek.Rosicky.Vitale:2011a} relevant to our purposes.
In particular, algebraic theories are often used to define data types within various
programming languages \cite{Mitchell:1996a}, and as stated in the introduction, our main goal is to 
connect databases and programming languages.

\begin{definition}
    \label{def:alg_theory}
  A (\emph{multisorted}) \emph{algebraic theory} is a cartesian strict monoidal category $\T[1]$
  together with a set $S_{\T}$, elements of which are called \emph{base sorts}, such that the monoid
  of objects of $\T[1]$ is free on $S_{\T}$. The terminal object in $\T[1]$ is denoted $\One$.

  The category $\ATh$ has algebraic theories as objects, and morphisms $\T\to\T[0]'$ are product
  preserving functors $F$ which send base sorts to base sorts: for any $s\in S_{\T}$, $F(s)\in
  S_{\T[0]'}$.
\end{definition}


\begin{remark}
  Throughout this paper we will discuss algebraic theories---categories with finite products and
  functors that preserve them---which are closely related to the notion of finite product
  sketches; see \cite{Barr.Wells:1985a}. However, aside from issues of syntax and computation,
  everything we say in this paper would also hold if algebraic theories were replaced by
  \emph{essentially algebraic theories}---categories with finite limits and functors that
  preserve them---which are analogous to finite limit sketches.
\end{remark}

\begin{definition}
    \label{def:algebra}
  Let $\T[1]$ be an algebraic theory. An \emph{algebra} (sometimes called a \emph{model}) of $\T[1]$
  is a finite product-preserving functor $\T\to\Set$. The category $\T\alg$ of $\T$-algebras is the
  full subcategory of $\Fun{\T,\Set}$ spanned by the finite product-preserving functors.
\end{definition}

\begin{example}
    \label{ex:representable_algebras}
  If $\T[1]$ is an algebraic theory, and $t\in\T[1]$ is an object, then the representable
  functor $\T[2](t,-)$ preserves finite products. Thus the Yoneda embedding
  $\yoneda\colon\op{\T[0]} \to\Fun{\T,\Set}$ factors through $\T\alg$.

  In particular, $\yoneda(\One)=\T[2](\One,-)$ is the initial $\T$-algebra for any algebraic theory,
  called the \emph{algebra of constants} and denoted by $\kappa\coloneqq\yoneda(\One)$.
\end{example}

We state the following theorem for future reference; proofs can be found in
\cite{Adamek.Rosicky:1994a}.

\begin{theorem}
    \label{thm:colimits_alg_theory}
  Let $\T$ be any algebraic theory.
  \begin{itemize}
    \item The Yoneda embedding $\yoneda\colon\op{\T[0]}\to\T\alg$ is dense. (By definition, $\T\alg$
      is a full subcategory of $\Fun{\T,\Set}$.)
    \item $\T\alg$ is closed in $\Fun{\T,\Set}$ under sifted colimits.
      (\cite[Prop.~2.5]{Adamek.Rosicky:1994a}.)
    \item $\T\alg$ has all colimits. (\cite[Thm.~4.5]{Adamek.Rosicky:1994a}.)
  \end{itemize}
\end{theorem}

\begin{warning}
    \label{wrn:TAlg_colims}
  Note that the forgetful functor $\T\alg\to\Fun{\T,\Set}$ in general does \emph{not} preserve
  colimits; i.e.\ colimits in $\T\alg$ are not taken pointwise. However, see
  \cref{rmk:SInst_colims}.
\end{warning}

\begin{remark}
    \label{rmk:density}
  For convenience, we will recall the notion of a dense functor, though we only use it in the case
  of the inclusion of a full subcategory.  A functor $F\colon\cat{A}\to\cat{C}$ is \emph{dense}
  if one of the following equivalent conditions holds:
  \begin{itemize}[nosep]
    \item for any object $C\in\cat{C}$, the canonical cocone from the canonical diagram
      $(F\downarrow C)\to\cat{C}$ to $C$ is a colimit cocone,
    \item the identity functor $\id_{\cat{C}}$ is the pointwise left Kan extension of $F$ along itself,
    \item the representable functor $\cat{C}(F,-)\colon\cat{C}\to[\op{\cat{A}},\Set]$ is fully faithful,
    \item (assuming $\cat{C}$ is cocomplete) for any object $C\in\cat{C}$, the canonical morphism
      $\int^{A\in\cat{A}}\cat{C}(F(A),C)\cdot F(A) \to C$ is an isomorphism.
  \end{itemize}
\end{remark}

\subsection{Algebraic profunctors}
In the previous section, we recalled the basic elements of the theory of profunctors
(see
\cref{ssec:profunctors,ssec:profunctor_matrix,ssec:profunctor_bimodule,sec:representable_profunctors,ssec:profunctor_tensor,prof_morphisms}).
At this point, we wish to characterize
those profunctors between a category and
an algebraic theory $M\colon\cat{C}\tickar\T$,
which interact nicely with the products in $\T$.

The following equivalences are easy to establish, by
translating a product-preserving condition for
$M\colon\op{\cat{C}}\times\T\to\Set$ under
$\left(-\times\cat{A}\right)\dashv(-)^\cat{A}$, and by \cref{eq:prof_collage}
for the collage construction in $\dProf$.

\begin{lemma}
    \label{lem:profunctor_products}
  Let $\cat{C}$ be a category and $\T$ an algebraic theory.
  For any profunctor $M\colon\cat{C}\tickar\T$, the following are equivalent:
  \begin{itemize}
  \item for each $c\in\cat{C}$, the functor $M(c,\text{--})\colon\T\to\Set$ preserves finite
  products,
  \item $M\colon \T\to\Set^{\op{\cat{C}}}$ preserves finite products,
  \item $M\colon \op{\cat{C}}\to\Set^{\T}$ factors through the full subcategory $\T\alg$,
  \item the inclusion $i_{\T}\colon\T\to\scol{M}$ into the collage of $M$ preserves finite
  products.
  \end{itemize}
\end{lemma}

\begin{definition}
    \label{def:profunctor_products}
  We refer to a profunctor $M$ satisfying any of the equivalent conditions of
  \cref{lem:profunctor_products} as an \emph{algebraic profunctor}, or we say that it
  \emph{preserves products on the right}. We denote a profunctor $M\colon\cat{C}\tickar\T$
  which is algebraic, using a differently-decorated arrow
  \[M\colon\cat{C}\tickxar\T.\]
  We define the category $\ProfTimes$ to be the full subcategory of the pullback
  \[
    \begin{tikzcd}
      \ProfTimes\ar[r,hook]&\cdot \ar[r] \ar[d] \ar[dr,phantom,"\lrcorner" pos=.1]
      & \dProf_1 \ar[d,"{(\lframe,\rframe)}"] \\
      &\Cat\times\ATh \ar[r]
      & \Cat\times\Cat
    \end{tikzcd}
  \]
  spanned by the algebraic profunctors. Here, $\lframe$ and $\rframe$ are the frame functors (\cref{def:double_cat}).
\end{definition}

Suppose given a pair of composable profunctors $\cat{C}\xtickar{M}\cat{D}\xtickxar{N}\T$ in which
the latter is algebraic. We want to compose them in such a way that the composition is also algebraic. It
is not hard to see that ordinary profunctor composition $M\odot N$ does not generally satisfy this property;
however, we can define a composition which does. In \cref{def:otimes} we will formalize
this as a left action $\otimes$ of $\dProf$ on $\ProfTimes$:
\begin{equation}\label{eqn:ProfAction}
  \begin{tikzcd}
    \Cat
    & \ProfTimes \ar[l,"\lframe"'] \ar[r,"\rframe"]
    & \ATh \\
    \dProf_1 \ar[u,"\rframe"] \ar[d,"\lframe"']
    & \cdot \ar[dr,dashed,"\otimes "] \ar[l] \ar[u] \ar[ul,phantom,"\ulcorner" pos=.1] & \\
    \Cat
    && \ProfTimes \ar[ll,"\lframe"] \ar[uu,"\rframe"']
  \end{tikzcd}
\end{equation}
We thus aim to define a functor $\otimes$ (dotted line) from the category
of composable profunctor pairs where the second is algebraic, such that
the above diagram commutes.

Let $\cat{D}$ be a category, $\T$ an algebraic theory, and
$N\colon\cat{D}\tickxar\T$ an algebraic profunctor. By
\cref{lem:profunctor_products}, we can consider $N$ to be a functor
$N\colon\op{\cat{D}}\to\T\alg$. Define the functor
$\LambdaTimes_N\colon\Set^{\cat{D}}\to\T\alg$ by the coend formula
\[
\LambdaTimes_N(J) =\int^{d\in\cat{D}}J(d)\cdot N(d)
\]
taken in the category $\T\alg$. This
coend exists because $\T\alg$ is cocomplete, and the formula coincides with
\cref{eqn:LambdaNonPointwise}, except there the coend is taken in $\Set^{\T}$, thus is pointwise.

\begin{definition}
    \label{def:otimes}
  Let $M\in\dProf_1(\cat{C},\cat{D})$ be a profunctor, and let $N\in\ProfTimes(\cat{D},\T)$ be
  an algebraic profunctor. The \emph{left tensor of $M$ on $N$}, denoted
  $M\otimes N\in\ProfTimes(\cat{C},\T)$ is defined by the composition $\LambdaTimes_N\circ
  M\colon\op{\cat{C}}\to\T\alg$.
\end{definition}

This left tensor can evidently be extended to a functor $\otimes $ as in \cref{eqn:ProfAction}. It
is also simple to check that it defines a left action of $\dProf$ on $\ProfTimes$, in the sense
that $\otimes$ respects units and composition in $\dProf$.

\section{Presentations and syntax}\label{sec:presentations_syntax}

In this section we will introduce syntax for algebraic theories, as well as for categories and
(co)presheaves. In general, a presentation of a given mathematical object consists of
generators and relations in a specified form. The object itself is then obtained by recursively
generating \emph{terms} according to a syntax, and then quotienting by the relations.

The material in this section is relatively standard (see, e.g.\ \cite{Jacobs:1999a} or \cite{Mitchell:1996a}). We go through
it carefully in order to fix the notation we will use in examples.

\subsection{Presentations of algebraic theories}

The presentation of an algebraic theory, as defined in \cref{def:alg_theory}, does not explicitly
mention products. Instead, it relies on multi-arity function symbols on the base sorts.
A signature simply lays out these sorts and function symbols.

\begin{definition}
    \label{def:signature}
  An \emph{algebraic signature} is a pair $\Sigma=(S_{\Sigma},\Phi_{\Sigma})$, where $S_{\Sigma}$ is
  a set of base sorts and $\Phi_{\Sigma}$ is a set of \emph{function symbols}. Each function symbol
  $f\in\Phi$ is assigned a (possibly empty, ordered) list of sorts $\dom(f)$ and a single sort
  $\cod(f)$. We use the notation $f\colon (s_1,\dots,s_n)\to s'$ to mean that
  $\dom(f)=(s_1,\dots,s_n)$ and $\cod(f)=s'$. We call $n$ the \emph{arity} of $f$; if $n=0$, we say
  it is $0$-ary and write it $f\colon()\to s'$.
\end{definition}

\begin{definition}
    \label{def:ASig}
  Let $\ASig$ denote the category of algebraic signatures. A morphism $F\colon\Sigma\to\Sigma'$
  between signatures is a pair of functions $F_S\colon S_{\Sigma}\to S_{\Sigma'}$ and
  $F_{\Phi}\colon\Phi_{\Sigma}\to\Phi_{\Sigma'}$, such that for any function symbol
  $f\in\Phi_{\Sigma}$ with $f\colon (s_1,\dots,s_n)\to s'$,
  $\dom(F_{\Phi}f)=(F_S(s_1),\dots,F_S(s_n))$ and $\cod(F_{\Phi}f)=F_S(s')$.
\end{definition}

\begin{example}
    \label{ex:monoid_action_sig}
  Consider the signature $\Sigma$ for the algebraic theory of monoid actions on a set. There are two
  sorts, $S=\{m,s\}$, and three function symbols, $\eta\colon()\to m$ for the unit, $\mu\colon
  (m,m)\to m$ for the multiplication, and $\alpha\colon (m,s)\to s$ for the action. If $\Sigma'$ is
  the signature for the theory of monoids, there is an evident inclusion morphism $\Sigma'\to\Sigma$.
\end{example}

\begin{example}
    \label{ex:underlying_signature}
  Every algebraic theory $\T$ has an underlying algebraic signature $\Sigma_{\T}$, whose
  base sorts are those of $\T$, and whose function symbols $f\colon (s_1,\dots,s_n)\to s'$ are
  the morphisms $f\in\T\mkern1mu(s_1\times\cdots\times s_n,s')$. This defines a functor
  $U\colon\ATh\to\ASig$.
\end{example}

We will see in \cref{rmk:free_alg_theory} that $U$ has a left adjoint, giving the free algebraic
theory generated by a signature. We construct this left adjoint \emph{syntactically}, and we will
make use of this syntax throughout the paper.

\begin{definition}
    \label{def:context}
  Fix an algebraic signature $\Sigma$. A \emph{context} $\Gamma$ \emph{over} $\Sigma$ is formally
  a set $\Gamma_v$ together with a function $\Gamma_s\colon \Gamma_v\to S_{\Sigma}$. In other words,
  a context is an object of the slice category $\Set_{/S_{\Sigma}}$, or equivalently the functor
  category $\Set^{S_{\Sigma}}$, regarding ${S_{\Sigma}}$ as a discrete category. When the set
  $\Gamma_v$ is finite, we will encode both $\Gamma_v$ and $\Gamma_s$ as a list $\Gamma=
  (x_1:s_1,\dots,x_n:s_n)$, and refer to $\Gamma$ as a \emph{finite} context.

  If $\Gamma=(x_1:s_1,\dots,x_n:s_n)$ and $\Gamma'=(x'_1:s'_1,\dots,x'_m:s'_m)$ are two contexts, we
  will write $\Gamma,\Gamma'=(x_1:s_1,\dots,x_n:s_n,\,x'_1:s'_1,\dots,x'_m:s'_m)$ for their
  concatenation, equivalently given by the induced function $\Gamma_v\sqcup\Gamma'_v\to S_{\Sigma}$.
  In practice, when concatenating contexts, we implicitly assume that variables are renamed as
  necessary to avoid name clashes. We denote the empty context by $\emptyContext$.
\end{definition}

\begin{remark}
  Intuitively, a context $(x_1:s_1,\dots,x_n:s_n)$ represents the declaration that symbol $x_i$
  ``belongs to the sort'' $s_i$. We treat the parentheses around a context as optional, and use them
  only as an aid to readability.
\end{remark}

The primary role of contexts is to explicitly list the ``free variables'' which are allowed to be
used inside an expression. Thus a context $(x:\Int,y:A)$ roughly corresponds to the English ``let
$x$ be an integer and let $y$ be an element of $A$''. The next definition makes this intuition
precise.

\begin{definition}
    \label{def:term}
  Fix an algebraic signature $\Sigma$ and a context $\Gamma$. A \emph{term} in context $\Gamma$ is
  an expression built out of the variables in $\Gamma$ and function symbols in $\Sigma$. Every term
  has an associated sort. We use the notation $\Gamma\vdash t:s$ to denote that $t$ is a term in
  context $\Gamma$ and that $t$ has sort $s$.

  Terms in context $\Gamma$ are defined recursively as follows:
  \begin{itemize}[nosep]
    \item if $(x:s)\in\Gamma$, then $\Gamma\vdash x:s$,
    \item if $f\colon (s_1,\dots,s_n)\to s'$ is a function symbol in $\Sigma$ and $\Gamma\vdash
      t_i:s_i$ for each $1\leq i\leq n$, then $\Gamma\vdash f(t_1,\dots,t_n):s'$.
  \end{itemize}
\end{definition}

We will sometimes refer to terms $\emptyContext\vdash t$ in the empty context as \emph{ground
terms}. A ground term $t$ must not contain any variables, and so must be constructed entirely out of
function symbols in $\Sigma$ (which includes 0-ary function symbols). Note that there can be terms
in non-empty contexts which contain no variables, but we will not call these ground terms.

\begin{example}
    \label{ex:monoid_action_term}
  In \cref{ex:monoid_action_sig} we gave the signature $\Sigma$ for monoid actions. An example term
  is $x_1:m,\, x_2:m,\, p:s\vdash \alpha(\mu(x_1,x_2),p):s$. An example ground term is
  $\emptyContext\vdash \mu(\eta,\mu(\eta,\eta)):m.$
\end{example}

One can think of a variable $x$ which appears in a term $t$ as a placeholder which can be replaced
by other expressions. For instance, in $x^3-2x$, the variable $x$ can be replaced by any number, or
even another polynomial. To make this precise, the operation of substitution is defined recursively.

\begin{definition}
    \label{def:substitution}
  Let $\Theta$, $\Gamma$, and $\Psi$ be contexts. If $(\Theta,\,x:s,\,\Psi)\vdash t:s'$ and
  $\Gamma\vdash u:s$ are terms, then $\Theta,\Gamma,\Psi\vdash t[x\coloneq u]:s'$ denotes the term
  obtained by replacing all occurrences of $x$ in $t$ with $u$. This substitution operation is
  defined formally by recursion:
  \begin{itemize}[nosep]
    \item $x[x\coloneq u] = u$,
    \item $x'[x\coloneq u] = x'$ if $x'\neq x$,
    \item $f(t_1,\dots,t_n)[x\coloneq u] = f(t_1[x\coloneq u],\dots,t_n[x\coloneq u])$.
  \end{itemize}

  If $(\Theta,\,x_1:s_1,\dots,x_n:s_n,\,\Psi)\vdash t:s'$ and $\Gamma\vdash u_i:s_i$ for all $1\leq
  i\leq n$, let $\Theta,\Gamma,\Psi\vdash t[x_1\coloneq u_1,\dots,x_n\coloneq u_n]:s'$ denote the
  term obtained by simultaneous substitution, also written $t[x_i\coloneq u_i]$ or
  $t[\vec{x}\coloneq\vec{u}]$ for compactness when this is clear.
\end{definition}

\begin{definition}
    \label{def:context_morphism}
  Let $\Gamma$ and $\Theta$ be contexts over an algebraic signature $\Sigma$, where
  $\Theta=(x_1:s_1,\dots,x_n:s_n)$ is finite. A \emph{context morphism} $\Gamma\to\Theta$ is a tuple
  of terms $\Gamma\vdash t_i:s_i$ for $1\leq i\leq n$, written $[x_1\coloneq t_1,\dots,x_n\coloneq
  t_n]\colon\Gamma\to\Theta$, or $[x_i\coloneq t_i]$ or $[\vec{x}\coloneq\vec{t}]$ for compactness.

  If $\Psi=(y_1:s'_1,\dots,y_m:s'_m)$ is another finite context, and
  $[\vec{y}\coloneq\vec{u}]\colon\Theta\to\Psi$ a context morphism, the composition
  $[\vec{y}\coloneq\vec{u}]\circ[\vec{x}\coloneq\vec{t}]$ is defined to be $[y_i\coloneq
  u_i[\vec{x}\coloneq\vec{t}]]\colon\Gamma\to\Psi$.
\end{definition}

\begin{example}
    \label{ex:monoid_action_morphism}
  Continuing with \cref{ex:monoid_action_sig,ex:monoid_action_term}, consider contexts
  $\Gamma=(x_1:m,\, x_2:m,\, p:s)$ and $\Theta=(y:m,\, q:s)$. There is a context morphism
  $\Gamma\to\Theta$ given by \[\left[y\coloneqq x_1,
  q\coloneqq\alpha\big(\mu(x_1,x_2),p\big)\right].\]
\end{example}

\begin{definition}
    \label{def:cat_of_contexts}
  Let $\Sigma$ be an algebraic signature. Define the \emph{category of contexts over $\Sigma$},
  denoted $\Cxt{\Sigma}$, to be the category of finite contexts over $\Sigma$ and context morphisms.
  We define the \emph{category of possibly infinite contexts over $\Sigma$}, denoted
  $\piCxt{\Sigma}$, to be the obvious extension.
\end{definition}

\begin{remark}
    \label{rmk:free_alg_theory}
  It is not hard to see that $\Cxt{\Sigma}$ has finite products, given by concatenation of contexts,
  and that the objects of $\Cxt{\Sigma}$ are freely generated under products by the base sorts
  (i.e.\ the singleton contexts). Thus this construction defines a functor $\Cxt{}\colon\ASig\to\ATh$.
  In fact, the functor $\Cxt{}$ is left adjoint to the underlying signature functor
  $U\colon\ATh\to\ASig$ from \cref{ex:underlying_signature}. Hence we will also refer to
  $\Cxt{\Sigma}$ as the \emph{free algebraic theory on the signature $\Sigma$}.
\end{remark}

The category $\ASig$ is for many purposes too rigid: a morphism in $\ASig$ is required to send
function symbols to function symbols, whereas one often wants to send function symbols to a more
complex expression. We now define this more flexible category of signatures.

\begin{definition}
    \label{def:ASig*}
  Define $\ASig^*$ to be the Kleisli category of the monad induced by the adjunction $\Cxt{}\dashv
  U$ of \cref{rmk:free_alg_theory} on the category $\ASig$.%
  \footnote{
    We use the symbol $\vdash$ between contexts and terms; we use the symbol $\dashv$ for
    adjunctions. Both are standard notation and no confusion should arise.
  }
  Concretely, $\ASig^*$ is defined just like $\ASig$ in \cref{def:ASig}, but where a morphism
  $F\colon\Sigma\to\Sigma'$ between signatures is allowed to send a function symbol $f\colon
  (s_1,\dots,s_n)\to s'$ in $\Phi_{\Sigma}$ to an arbitrary term
  $\big(x_1:F_S(s_1),\dots,x_n:F_S(s_n)\big)\vdash F_{\Phi}(f):F_S(s')$ over $\Sigma'$. Composition of these
  signature morphisms is defined by substitution.
\end{definition}

We are now ready to discuss presentations of algebraic theories. We begin with a careful
consideration of equations.

\begin{definition}
    \label{def:alg_equations}
  Let $\Sigma$ be an algebraic signature. An \emph{equation over} $\Sigma$ is a pair of terms
  $(t,t')$, where $t$ and $t'$ are in the same finite context $\Gamma$ and have the same sort $s$.
  We denote such a pair by the equation $\Gamma\vdash (t=t'):s$, or simply by $\Gamma\vdash t=t'$ if
  no confusion should arise.

  Let $E$ be a set of equations over $\Sigma$. Define $\approx_E$ to be the smallest equivalence
  relation on terms over $\Sigma$ such that
  \begin{enumerate}[nosep]
    \item if $\Gamma\vdash t=t'$ is an equation of $E$, then $\Gamma\vdash t\approx_E t'$,
    \item if $f\colon (s_1,\dots,s_n)\to s'$ is a function symbol and $\Gamma\vdash (t_i\approx_E
      t'_i):s_i$ for all $1\leq i\leq n$, then $\Gamma\vdash \bigl(f(t_1,\dots,t_n)\approx
      f(t'_1,\dots,t'_n)\bigr):s'$,
    \item \labelold[condition]{equations_substitution} if $\Theta\vdash (t\approx_E t'):s$, and
      $[\vec{x}\coloneq\vec{u}]\colon\Gamma\to\Theta$ is a context morphism, then
      $\Gamma\vdash\bigl(t[\vec{x}\coloneq\vec{u}]\approx_E t'[\vec{x}\coloneq\vec{u}]\bigr):s$.
  \end{enumerate}
\end{definition}

\begin{remark}
  \Cref{equations_substitution} of \cref{def:alg_equations} is equivalent to the following two
  conditions:
  \begin{enumerate}[label=\labelcref{equations_substitution}\alph*.,
                    ref=\labelcref{equations_substitution}\alph*,
                    labelindent=\parindent,
                    leftmargin=*]
    \item if $(\Theta,\Psi)\vdash (t\approx_E t'):s'$, then $(\Theta,x:s,\Psi)\vdash (t\approx_E
      t'):s'$ for any sort $s'$,
    \item if $(\Theta,x:s,\Psi)\vdash (t\approx_E t'):s'$, and $\Gamma\vdash u:s$ is a term, then we
      have $(\Theta,\Gamma,\Psi)\vdash \big(t[x\coloneqq u]\approx_E t'[x\coloneqq u]\big):s'$.
  \end{enumerate}
\end{remark}

\begin{definition}
    \label{def:alg_theory_presentation}
  Let $\Sigma$ be an algebraic signature, and $E$ a set of equations over $\Sigma$. The algebraic
  theory $\Cxt{\Sigma}/E$ is the quotient of $\Cxt{\Sigma}$ by the equivalence relation $\approx_E$.
  In other words, the objects of $\Cxt{\Sigma}/E$ are finite contexts over $\Sigma$, and the morphisms
  are $\approx_E$-equivalence classes of context morphisms. This quotient is well defined because
  $\approx_E$ is by definition preserved under substitution.

  We call the pair $(\Sigma,E)$ a \emph{presentation} of the algebraic theory $\T$ if there is
  an isomorphism $\T\iso\Cxt{\Sigma}/E$. We call it a \emph{finite presentation} if both
  $\Sigma$ and $E$ are finite.
\end{definition}

We now conclude our running example of monoid actions.

\begin{example}
  In \cref{ex:monoid_action_sig}, we gave the signature for monoid actions on sets, with sorts $m,s$
  and function symbols $\eta,\mu,\alpha$. To present the algebraic theory of monoid actions on sets,
  we add the following four equations:
  \begin{align*}
    x:m&\vdash \mu(x,\eta)=x
    &
    x,y,z:m&\vdash\mu\big(x,\mu(y,z)\big)=\mu\big(\mu(x,y),z\big)
    \\
    x:m&\vdash\mu(\eta,x)=x
    &
    x:m,\,y:m,\,p:s&\vdash\alpha\big(x,\alpha(y,p)\big)=\alpha\big(\mu(x,y),p\big)
  \end{align*}
\end{example}

\begin{definition}
    \label{def:APr}
  Define the category of \emph{algebraic presentations} $\APr$ as follows: the objects of $\APr$ are
  pairs $(\Sigma,E)$, where $\Sigma$ is an algebraic signature and $E$ is a set of equations over
  $\Sigma$. A morphism $F\colon(\Sigma,E)\to(\Sigma',E')$ is a morphism $F\colon\Sigma\to\Sigma'$ in
  the Kleisli category $\ASig^*$ such that $F(t)\approx_{E'}F(t')$ for each equation $t=t'$ of $E$.

  Let $\Cxt{}$ also denote the functor $\APr\to\ATh$ sending a pair $(\Sigma,E)$ to
  $\Cxt{\Sigma}/E$.
\end{definition}

\begin{remark}
    \label{rmk:canonical_presentation}
  Any algebraic theory $\T$ has a \emph{canonical presentation} $(\Sigma_{\T},E_{\T})$,
  where $\Sigma_{\T}$ is the underlying signature from \cref{ex:underlying_signature}, and
  $E_{\T}$ is defined such that an equation $x_1:s_1,\ldots,x_n:s_n\vdash (t=t'):s$ is in $E_{\T}$ if and only if the
  morphisms corresponding to $t$ and $t'$ in the hom-set $\T\left(s_1\times\cdots\times s_n, s\right)$ are equal.

  It is not hard to see that $\Cxt{\Sigma_{\T}}/E_{\T}\iso\T$ for any algebraic
  theory $\T$. It is also straightforward to check that $\Cxt{}\colon\APr\to\ATh$ is fully
  faithful, and hence an equivalence of categories.
\end{remark}

The following easy proposition establishes the fundamental connection between a presentation for an
algebraic theory $\T$ and algebras on $\T$.

\begin{proposition}
    \label{prop:algs_on_presented_theory}
  Let $\Sigma$ be an algebraic signature and $E$ be a set of equations, and consider an assignment of a set
  $F_s$ to each sort $s\in S_{\Sigma}$ and a function $F_f\colon F_{s_1}\times\cdots\times
  F_{s_n}\to F_{s'}$ to each function symbol $f\colon (s_1,\dots,s_n)\to s'$ in $\Phi_{\Sigma}$.
  This assignment uniquely extends to a $\Cxt{\Sigma}$-algebra $F$. In particular, given any term
  $(x_1:s_1,\dots,x_n:s_n)\vdash t:s'$, there is a function $F_t\colon F_{s_1}\times\cdots\times
  F_{s_n}\to F_{s'}$.

  The assignment uniquely extends to a $\Cxt{\Sigma}/E$-algebra if and only if it satisfies the
  equations $E$, i.e.\ for each equation $\Gamma\vdash t_1=t_2$ of $E$, the functions $F_{t_1}$ and
  $F_{t_2}$ are equal.
\end{proposition}

\begin{example}
    \label{CommRingsex}
  Consider the presentation $(\Sigma,E)$, where $S_{\Sigma}=\{\Int\}$ is the only sort,
  $\Phi_{\Sigma}$ consists of the five function symbols $0,1\colon()\to\Int$,
  $(-)\colon(\Int)\to\Int$, and ${+},{\times}\colon(\Int,\Int)\to\Int$, and $E$ is the set of
  equations shown in \cref{eqsrings}.
  \begin{figure}
    \centering
    \small
    \begin{tabular}{@{}r@{${}\vdash{}$}l@{}}
      \toprule
      $x,y,z:\Int$ & $(x+y)+z = x+(y+z)$ \\
      $x:\Int$     & $x+0=x$ \\
      $x,y:\Int$   & $x+y=y+x$ \\
      $x:\Int$     & $x+(-x)=0$ \\
      $x,y,z:\Int$ & $(x\times y)\times z = x\times(y\times z)$ \\
      $x:\Int$     & $x\times 1=x$ \\
      $x,y:\Int$   & $x\times y=y\times x$ \\
      $x,y,z:\Int$ & $x\times(y+z)=(x\times y)+(x\times z)$ \\
      \bottomrule
    \end{tabular}
    \caption{Equations for the algebraic theory of commutative rings.}
    \labelold{eqsrings}
  \end{figure}
  The algebraic theory $\T=\Cxt{\Sigma}/E$ generated by this presentation is a category with objects
  the contexts over $\Sigma$, such as $(a:\Int)$ or $(x,y,z:\Int)$. Some example context morphisms
  $(x,y,z:\Int)\to(a:\Int)$ are $[a\coloneq (x+y)\times(x+z)]$ and $[a\coloneq ((x\times x)+(x\times
  z))+((y\times x)+(y\times z))]$. These context morphisms are equivalent under $\approx_E$, so
  determine the same morphism in $\T$.

  It is possible to show that the category $\T\alg$ is equivalent to the category $\ncat{CRing}$ of
  commutative rings. In particular, a product preserving functor $F\colon\T\to\Set$ must send the
  context $(a:\Int)$ to some set $R$, the context morphism $[a\coloneq
  x+y]\colon(x,y:\Int)\to(a:\Int)$ to a function $R\times R\to R$, etc.

  We name the single sort `$\Int$' to fit with the practice in type theory, in which the ``elements''
  of a type $\tau$ are considered to be the ground terms of type $\tau$. That the ground terms of
  sort `$\Int$' are precisely the integers is equivalent to the fact that $\ZZ$ is the initial
  commutative ring.
\end{example}

We will use the following algebraic theory throughout the paper in our database examples.

\begin{example}
    \label{Typedef}
  Consider the multi-sorted algebraic theory $\mathbf{Type}$ generated by the finite theory
  presentation with base sorts, function symbols and equations as defined in \cref{tab:Type},
  page~\pageref{tab:Type}. It may be helpful to recall that implication can be written as
  $(a\Rightarrow b)=\neg a\vee b$. We use an axiomatization of Boolean algebras which is proven
  complete in \cite{Huntington:1904a}.

  \begin{figure}[htp]
    \centering
    \small
    \begin{tabular}{@{}r@{}l@{}}
      \toprule
      $S_\Sigma:\quad$ & $\Int$, $\Bool$, $\Str$ \\
      $\Phi_\Sigma:\quad$ &
      \begin{tabular}[t]{@{}lll@{}}
        \begin{tabular}[t]{@{}r@{${}:{}$}l@{}}
          $0,1$ & $()\to\Int$ \\
          $({-})$ & $(\Int)\to\Int$ \\
          $+,\times$ & $(\Int,\Int)\to\Int$ \\
          $\leq$ & $(\Int,\Int)\to\Bool$
        \end{tabular} &
        \begin{tabular}[t]{@{}r@{${}:{}$}l@{}}
          $\top,\bot$ & $()\to\Bool$ \\
          $\neg$ & $(\Bool)\to\Bool$ \\
          $\wedge$ & $(\Bool,\Bool)\to\Bool$ \\
          $\vee$ & $(\Bool,\Bool)\to\Bool$
        \end{tabular} &
        \begin{tabular}[t]{@{}r@{${}:{}$}l@{}}
          $\varepsilon,`\mathrm{a}',\dots,`\mathrm{Z}'$ & $()\to\Str$ \\
          $(.)$ & $(\Str,\Str)\to\Str$ \\
          $\textrm{eq}$ & $(\Str,\Str)\to\Bool$
        \end{tabular}
      \end{tabular} \\
      $E:\quad$ & \textit{\hspace{-.5em}boolean algebra:} \\
      &
      \begin{tabular}[t]{@{}r@{${}\vdash{}$}lr@{${}\vdash{}$}l@{}}
        $\alpha:\Bool$ & $\alpha\vee\bot=\alpha$ & $\alpha:\Bool$ & $\alpha\wedge\top=\alpha$ \\
        $\alpha,\beta:\Bool$ & $\alpha\vee\beta=\beta\vee\alpha$ &
        $\alpha,\beta:\Bool$ & $\alpha\wedge\beta=\beta\wedge\alpha$ \\
        $\alpha:\Bool$ & $\alpha\vee\neg\alpha=\top$ &
        $\alpha:\Bool$ & $\alpha\wedge\neg\alpha=\bot$ \\
        $\alpha,\beta,\gamma:\Bool$ & $\alpha\vee(\beta\wedge\gamma) = (\alpha\vee\beta)\wedge(\alpha\vee\gamma)$ &
        $\alpha,\beta,\gamma:\Bool$ & $\alpha\wedge(\beta\vee\gamma) = (\alpha\wedge\beta)\vee(\alpha\wedge\gamma)$
      \end{tabular} \\
      & \textit{\hspace{-.5em}commutative ring: all equations from \cref{eqsrings}} \\
      & \textit{\hspace{-.5em}totally pre-ordered ring:} \\
      &
      \begin{tabular}[t]{@{}r@{${}\vdash{}$}l@{}}
        $x,y,z:\Int$ & $\neg((x\leq y)\wedge(y\leq z))\vee(x\leq z)=\top$ \\
        $x,y:\Int$ & $(x\leq y)\vee(y\leq x)=\top$ \\
        $x,y,z,w:\Int$ & $\neg((x\leq y)\wedge(z\leq w))\vee(x+z\leq y+w)=\top$ \\
        $x,y,z:\Int$ & $\neg((x\leq y)\wedge(0\leq z))\vee(x\times z\leq y\times z)=\top$ \\
        $x,y,z:\Int$ & $\neg((x\times z\leq y\times z)\wedge(0\leq z))\vee(x\leq y)=\top$ \\
        $\emptyContext$ & $(1\leq0)=\neg\top$
      \end{tabular} \\
      & \textit{\hspace{-.5em}monoid:} \\
      &
      \begin{tabular}[t]{@{}r@{${}\vdash{}$}lr@{}l@{}}
        $s:\Str$ & $s.\varepsilon=s$ &
        $s,t,u:\Str$ & ${}\vdash(s.t).u=s.(t.u)$ \\
        $s:\Str$ & $\varepsilon.s=s$ &
      \end{tabular} \\
      & \textit{\hspace{-.5em}congruence:} \\
      &
      \begin{tabular}[t]{@{}r@{${}\vdash{}$}l@{}}
        $s:\Str$ & $(s\;\textrm{eq}\;s)=\top$ \\
        $s,t:\Str$ & $(s\;\textrm{eq}\;t)=(t\;\textrm{eq}\;s)$ \\
        $s,t,u:\Str$ &
          $\neg((s\;\textrm{eq}\;t)\wedge(t\;\textrm{eq}\;u))\vee(s\;\textrm{eq}\;u)=\top$ \\
          $s,t,u,v:\Str$ & $\neg((s\;\textrm{eq}\;t)\wedge(u\;\textrm{eq}\;v))
            \vee(s.u\;\textrm{eq}\;t.v)=\top$ 
      \end{tabular} \\
      & \textit{\hspace{-.5em}decidable equality:} \\
      &
      \begin{tabular}[t]{@{}r@{${}\vdash{}$}lcr@{}l@{}}
        $s,t,u:\Str$ & $(s.u\;\textrm{eq}\;t.u)=(s\;\textrm{eq}\;t)$ && \\
        $s,t,u:\Str$ & $(s.t\;\textrm{eq}\;s.u)=(t\;\textrm{eq}\;u)$ && \\
        $s,t:\Str$ & $(s.`\textrm{a}'\;\textrm{eq}\;t.`\textrm{b}')=\neg\top$
          & $\dots$ & $s,t:\Str$ & ${}\vdash(s.`\textrm{y}'\;\textrm{eq}\;t.`\textrm{z}')=\neg\top$ \\
        $s,t:\Str$ & $(`\textrm{a}'.s\;\textrm{eq}\;`\textrm{b}'.t)=\neg\top$
          & $\dots$ & $s,t:\Str$ & ${}\vdash(`\textrm{y}'.s\;\textrm{eq}\;`\textrm{z}'.t)=\neg\top$ \\
        $s:\Str$ & $(s.`\textrm{a}'\;\textrm{eq}\;\epsilon)=\neg\top$
          & $\dots$ & $s:\Str$ & ${}\vdash(s.`\textrm{z}'\;\textrm{eq}\;\epsilon)=\neg\top$ \\
        $s:\Str$ & $(`\textrm{a}'.s\;\textrm{eq}\;\epsilon)=\neg\top$
          & $\dots$ & $s:\Str$ & ${}\vdash(`\textrm{z}'.s\;\textrm{eq}\;\epsilon)=\neg\top$ \\
      \end{tabular} \\
      \bottomrule
    \end{tabular}
    \caption{Presentation of $\Type$, our running example of an algebraic theory.}
    \labelold{tab:Type}
  \end{figure}

  Clearly this algebraic theory includes the one from \cref{CommRingsex} as a sub-theory. Similarly
  to viewing ground terms of type `$\Int$' as the integers, those of type `$\Str$' are strings of
  letters, presented as the free monoid on $52$ generators (upper and lower case letters). For
  example, when we later write $`\mathrm{Admin}':\Str$ we actually mean the term
  $\emptyContext\vdash\mathrm{`A'.`d'.`m'.`i'.`n'}:\Str$.%
  \footnote{Similarly, we may use the shorthand $x-y$ to denote what is really $x+(-y)$.}
  It can be shown that the ground terms of type `$\Bool$' are $\{\textrm{True},\textrm{False}\}$.

\end{example}

\subsection{Presentations of algebras}
  \label{sec:presentation_algebra}
We now turn to presentations of algebras. Fix a presentation $(\Sigma,E)$, and let
$\T=\Cxt{\Sigma}/E$ be the presented algebraic theory. Recall by \cref{def:context} how we can
think of objects in $\Set_{/S_{\Sigma}}$, the category of $S_{\Sigma}$-indexed sets, as (possibly
infinite) contexts over $\Sigma$. There is an evident forgetful functor
$U\colon\T\alg\to\Set_{/S_{\Sigma}}$, which sends an algebra $A\colon\T\to\Set$ to the indexed
set $\{(UA)_s\}_{s\in S_{\Sigma}}$ where $(UA)_s=A(s)$.

\begin{definition}
    \label{def:free_alg}
  Let $\Gamma\in\Set_{S_{\Sigma}}$ be an $S_{\Sigma}$-indexed set, thought of as a context,
  and let $\T=\Cxt{\Sigma}/E$ be a presented algebraic theory. Define the
  \emph{free $\T$-algebra on $\Gamma$}, denoted $\FrAlg{\Gamma}\colon\T\to\Set$, to be
  the algebra for which $\FrAlg{\Gamma}(\Theta)$ is the set
  of $\approx_E$-equivalence classes of context morphisms $\Gamma\to\Theta$, with functoriality
  given by composition of context morphisms.
\end{definition}

\begin{remark}
    \label{rmk:term_alg}
  With notation as in \cref{def:free_alg}, the elements of $\FrAlg{\Gamma}$ of sort $s$ are just the
  $\approx_E$-equivalence classes of terms $\Gamma\vdash t:s$, and for any function symbol $f\colon
  (s_1,\dots,s_n)\to s'$ the induced function sends a tuple of terms $\Gamma\vdash t_i:s_i$ to the
  term $\Gamma\vdash f(t_1,\dots,t_n):s'$. By \cref{prop:algs_on_presented_theory}, this completely
  defines the algebra $\FrAlg{\Gamma}$, whose standard name is the \emph{term algebra} over
  $\Gamma$. There is an adjunction
  \begin{equation}\label{eqn:term_algebra}
  \FrAlg{\textrm{--}}\colon\Set_{S_{\Sigma}}\leftrightarrows(\Cxt{\Sigma}/E)\alg\cocolon U
  \end{equation}
  For this reason, we may refer to $\FrAlg{\Gamma}$ as the \emph{free algebra} on the generating
  context $\Gamma$.
\end{remark}

\begin{example}
    \label{ex:alg_constants}
  Let $\T$ be any algebraic theory. The algebra generated by the empty context is the
  algebra of constants $\FrAlg{\emptyset}=\kappa=\yoneda(\One)$; see \cref{ex:representable_algebras}.
  Note that any term in $\kappa$ is necessarily a ground term (\cref{def:term}). If $X$ is any other
  $\T$-algebra, we refer to terms in the image
  of the unique map $\kappa\to X$ as \emph{constants} in $X$.
\end{example}

\begin{example}
    \label{ex:polynomial_ring}
  Let $\T$ be the algebraic theory from \cref{CommRingsex}. The elements of $\FrAlg{x,y:\Int}$
  of the unique base sort `$\Int$' are $\approx_E$-equivalence classes of terms $x,y:\Int\vdash
  t:\Int$, such as $x,y:\Int\vdash (x+y)\times x$. But these are just polynomials in the variables
  $x$ and $y$, hence the commutative ring $\FrAlg{x,y:\Int}\in\T\alg$ is the polynomial ring
  $\ZZ[x,y]$, the free commutative ring on the set $\{x,y\}$.
\end{example}

\begin{remark}
    \label{rmk:free_alg}
  Note that the Kleisli category for the adjunction \cref{eqn:term_algebra} is precisely the
  opposite of the category $\piCxt{\Sigma}/E$ of possibly infinite contexts
  (\cref{def:cat_of_contexts}), and the restriction of this Kleisli category to those objects $X\to
  S_{\Sigma}$ of $\Set_{/S_{\Sigma}}$ for which $X$ is finite is the category
  $\op{(\Cxt{\Sigma}/E)}$. Another way to say this is that the algebraic theory $\Cxt{\Sigma}/E$ is
  isomorphic to the opposite of the category of finitely generated free algebras over
  $\Cxt{\Sigma}/E$, a fact which is true for any algebraic theory; see
  \cite[\S~8]{Adamek.Rosicky.Vitale:2011a}.
\end{remark}

\begin{definition}
    \label{def:eqs_in_context}
  Let $\Sigma$ be an algebraic signature, $\Gamma$ a context over $\Sigma$, and $e$ an equation over
  $\Sigma$. Say that $e$ is an \emph{equation in $\Gamma$} if it is between terms in context
  $\Gamma$. A set $E'$ of equations over $\Sigma$ is said to be \emph{in $\Gamma$} if each element
  $e\in E'$ is.
\end{definition}

\begin{definition}
    \label{def:presented_algebra}
  Let $(\Sigma,E)$ be a presentation for an algebraic theory $\T$. A \emph{$\T$-algebra
  presentation} is a pair $(\Gamma,E')$, where $\Gamma$ is a context over $\Sigma$, and $E'$ is a
  set of equations in $\Gamma$. Define $\FrAlg{\Gamma}/E'$ to be the quotient of the free $\T$-algebra
  $\FrAlg{\Gamma}$ (\cref{def:free_alg}) by the equations $E'$. Concretely, $(\FrAlg{\Gamma}/E')(\Theta)$ is the set of
  $\approx_{E\cup E'}$-equivalence classes of context morphisms $\Gamma\to\Theta$.

  A \emph{morphism of $\T$-algebra presentations} (cf.\ \cref{def:APr}) $(\Gamma',E')\to(\Gamma'',E'')$
  is simply a context morphism $[\vec{x}\coloneq\vec{t}]\colon\Gamma''\to\Gamma'$ (note the
  direction!) such that for each equation $\Gamma'\vdash u=v$ in $E'$, it follows that
  $\Gamma''\vdash u[\vec{x}\coloneq\vec{t}]\approx_{E\cup E''} v[\vec{x}\coloneq\vec{t}]$.
\end{definition}

\begin{example}
  Let $(\Sigma,E)$ be the theory of commutative rings as in \cref{CommRingsex}, and let
  $\Gamma=(x,y:\Int)$. Then $\FrAlg{\Gamma}$ is the polynomial ring $\ZZ[x,y]$ (\cref{ex:polynomial_ring}).
  If $e$ is the equation $x^3=y^2$ then $(\FrAlg{\Gamma}/\{e\})$ is the ring $\ZZ[x,y]/(x^3-y^2)$.
\end{example}

\begin{remark}
    \label{rmk:canonical_presentation_alg}
  Recall that by \cref{rmk:canonical_presentation}, every algebraic theory has a canonical presentation
  and the functor $\APr\to\ATh$ from presentations to theories is an equivalence. For algebras the
  same turns out to be true. First, every $\T$-algebra $A\in(\Cxt{\Sigma}/E)\alg$ has a
  canonical presentation $(\Gamma,E')$, where $\Gamma$ is the underlying $S_{\Sigma}$-indexed set $UA$, and $E'$
  is the set of equations $\Gamma\vdash t=t'$ for which $t$ and $t'$ are equated under the counit
  $\FrAlg{\Gamma}\to A$ of the adjunction from \cref{eqn:term_algebra}. 
  Second, the category of such presentations (whose objects and morphisms are given in
  \cref{def:presented_algebra}) is equivalent to $\T\alg$.
\end{remark}


\subsection{Presentations of categories}
  \label{ssec:category_presentations}

It is well known that categories are algebraic over directed graphs, i.e.\ that a category can be
presented by giving a graph together with a set of equations (see
e.g.~\cite[\S~II.8]{MacLane:1998a}). In the interest of completeness and consistency, we will show
here how to consider presentations for categories as a special case of presentations for algebraic
theories (see \cref{def:category_presentation}).

Formally, a directed graph $G$ consists of a set $G_0$ of nodes and a set $G_1$ of edges, together
with functions $\dom,\cod\colon G_1\to G_0$. Note that a directed graph $G$ can be seen as
an algebraic signature (\cref{def:signature}) in which all function symbols are unary. The set of
sorts of the unary signature is simply the set $G_0$ of nodes of $G$, and the set $G_1$ of edges is
taken as the set of function symbols.

\begin{definition}
    \label{def:unary_equations_signatures}
  Say that an algebraic signature $\Sigma$ is \emph{unary} when all of its function symbols are
  unary. As usual, we will write $f\colon A\to B$ as shorthand for $\dom(f)=A$ and $\cod(f)=B$. From now
  on, we will identify a graph $G$ with its corresponding unary algebraic signature.
\end{definition}

\begin{remark}
    \label{rmk:terms_and_paths}
  Let $G$ be a graph, and let $A,B\in G_0$ be nodes. Terms (\cref{def:term}) of type $B$ in a
  singleton context $(x:A)$ over (the unary signature associated to) $G$ can be identified with
  paths from $A$ to $B$ in $G$. As $G$ has only unary function symbols, all such terms must be of
  the form $x:A\vdash f_n(\dots f_2(f_1(x))):B$ for some $n\geq 0$.
\end{remark}

\begin{proposition}
  Let $G$ be a graph and $\Fr(G)$ the free category generated by $G$. Then $\Fr(G)$ is isomorphic to the full
  subcategory of $\Cxt{G}$ spanned by the singleton contexts.
\end{proposition}

\begin{notation}
  Let $\Sigma$ be an algebraic signature, and let $\Gamma\vdash t$ be a term in some context. In
  order to reduce parentheses, we will use the notation $\Gamma\vdash t.f_1.f_2\dots f_n$ to denote
  $\Gamma\vdash f_n(\dots f_2(f_1(t)))$, assuming that this is a well-formed term and that each
  $f_i$ is unary.
\end{notation}

\begin{definition}
  Let $\Sigma$ be a (not necessarily unary) signature. An equation $\Gamma\vdash t=t'$ over $\Sigma$
  is \emph{unary} if the context $\Gamma$ is a singleton. Say that a set $E$ of equations is unary
  if it consists only of unary equations.
\end{definition}

\begin{definition}
    \label{def:category_presentation}
  A \emph{category presentation} is a pair $(G,E)$, where $G$ is a graph and $E$ is a set of unary
  equations over $G$. Define the category \emph{presented} by $(G,E)$, denoted $\Fr(G)/E$, to be the
  full subcategory of $\Cxt{G}/E$ spanned by the singleton contexts.
\end{definition}

\begin{proposition}
    \label{prop:free_fpcat_on_cat}
  Let $(G,E)$ be a category presentation. The category $\Cxt{G}/E$ is the free
  category-with-finite-products on the category $\Fr(G)/E$. In particular, there is an equivalence
  of categories $(\Cxt{G}/E)\alg \equiv \Fun{\Fr(G)/E,\Set}$.
\end{proposition}

\subsection{Presentations of set-valued functors}
  \label{ssec:copresheaf_presentations}

If $\cat{C}$ is a category given by a presentation $(G,E)$, then \cref{prop:free_fpcat_on_cat}
provides a way of giving presentations for functors $\cat{C}\to\Set$. Let $\Gamma$ be a context over
the unary algebraic signature $G$. Then we can form the free algebra
$\FrAlg{\Gamma}\in(\Cxt{G}/E)\alg$ as in \cref{def:free_alg}. Under the equivalence $(\Cxt{G}/E)\alg
\equiv \Fun{\cat{C},\Set}$, this corresponds to a functor $\cat{C}\to\Set$, namely the restriction
of $\FrAlg{\Gamma}\colon\Cxt{G}/E\to\Set$ to its full subcategory of singleton contexts $\cat{C}$.
We will denote this restriction $\FrCpsh{\Gamma}$.

It is straightforward to check that the adjunction from \cref{rmk:term_alg} restricts to an
adjunction $\FrCpsh{\textrm{--}}\colon\Set_{G_0}\leftrightarrows\Fun{\cat{C},\Set}\cocolon U$.
Hence $\FrCpsh{\Gamma}$ is the \emph{free copresheaf} on $\cat{C}$ generated by $\Gamma$.

Similarly, if $E'$ is a set of equations in context $\Gamma$, as in \cref{def:eqs_in_context}, then we
denote by $\FrCpsh{\Gamma}/E'$ the restriction of $\FrAlg{\Gamma}/E'\colon\Cxt{G}/E\to\Set$ to
$\cat{C}$, and refer to this as the copresheaf \emph{presented} by $(\Gamma,E')$.

\subsection{Presentations of algebraic profunctors}
  \label{ssec:profunctor_presentations}

In \cref{sec:schemas} we will be interested in algebraic profunctors $M\colon\cat{C}\tickxar\T$
where $\cat{C}$ is a category and $\T$ is an algebraic theory; see \cref{def:profunctor_products}.
Our approach to presenting an algebraic profunctor $M$ between $\cat{C}=\Fr(G)/E_G$ and $\T=\Cxt{\Sigma}/E_{\Sigma}$
will be in terms of its collage $\bcol{M}$, as in \cref{ex:prof_collages}.

\begin{definition}
    \label{def:profunctor_presentation}
  Let $G=(G_0,G_1)$ be a graph (unary signature) and $\Sigma=(S_{\Sigma},\Phi_{\Sigma})$ be an algebraic signature.
  A \emph{profunctor signature} $\Upsilon$ from $G$ to $\Sigma$ is a set of unary function symbols, where each function
  symbol $\Mor{att}\in\Upsilon$ is assigned a sort $a\coloneqq\dom(\Mor{att})\in G_0$ and a sort $\tau\coloneqq\cod(\Mor{att})\in S_{\Sigma}$. We will sometimes refer to
  the function symbol $\Mor{att}\in\Upsilon$ as an \emph{attribute}, and denote it $\Mor{att}\colon a\to\tau$.

  A profunctor signature $\Upsilon$ has an associated algebraic signature
  $\scol{\Upsilon}=(S_{\scol{\Upsilon}}, \Phi_{\scol{\Upsilon}})$, with sorts $S_{\scol{\Upsilon}} =
  G_0\sqcup S_{\Sigma}$, and function symbols $\Phi_{\scol{\Upsilon}} =
  G_1\sqcup\Upsilon\sqcup\Phi_{\Sigma}$.

  Say that a set $E_{\Upsilon}$ of equations over $\scol{\Upsilon}$ is a set of \emph{profunctor
  equations} if for each equation $\Gamma\vdash (t_1=t_1):s'$ in $E_{\Upsilon}$, the context is a
  singleton $\Gamma=(x:s)$ with $s\in G_0$ and $s'\in S_{\Sigma}$.
\end{definition}

\begin{definition}
    \label{def:presented_alg_prof}
  Let $(G,E_G)$ be a category presentation, $(\Sigma,E_{\Sigma})$ an algebraic theory presentation,
  $\Upsilon$ a profunctor signature from $G$ to $\Sigma$, and  $E_{\Upsilon}$ a set of profunctor
  equations. Let $E_{\scol{\Upsilon}}=E_G\cup E_{\Sigma}\cup E_{\Upsilon}$.
  Define the algebraic profunctor \emph{presented} by this data,
  denoted $\FrAlg{\Upsilon}/E_{\Upsilon}\colon\Fr(G)/E_G\tickxar\Cxt{\Sigma}/E_{\Sigma}$, as follows:
  \begin{itemize}[nosep]
    \item for any node $a\in G_0$ and context $\Gamma\in\Cxt{\Sigma}$, the set
      $(\FrAlg{\Upsilon}/E_{\Upsilon})(a,\Gamma)$ is the hom set
      $(\Cxt{\scol{\Upsilon}}/E_{\scol{\Upsilon}})((x:a),\Gamma)$, i.e.\ the set of $\approx_{E_G\cup
      E_{\Sigma}\cup E_{\Upsilon}}$-equivalence classes of context morphisms $(x:a)\to\Gamma$ over
      $\scol{\Upsilon}$,
    \item the functorial actions are given by substitution.
  \end{itemize}
\end{definition}

It is clear from the definition that the collage of the profunctor $\FrAlg{\Upsilon}/E_{\Upsilon}$
is a full subcategory of $\Cxt{\scol{\Upsilon}}/E_{\scol{\Upsilon}}$. In fact, it is not much harder
to see the following proposition; cf.\ \cref{prop:free_fpcat_on_cat}.

\begin{proposition}
    \label{prop:collage_free_product_completion}
  Let $\cat{C}$ be a category with presentation $(G,E_G)$, let $\T$ be an algebraic theory with
  presentation $(\Sigma,E_{\Sigma})$, and let $P\colon\cat{C}\tickxar\T$ be an algebraic profunctor
  with presentation $(\Upsilon,E_{\Upsilon})$. The category
  $\Cxt{\scol{\Upsilon}}/E_{\scol{\Upsilon}}$ is the free completion of the collage $\scol{P}$ under
  finite products for which existing products in $\T$ are preserved. In particular, the category
  $(\Cxt{\scol{\Upsilon}}/E_{\scol{\Upsilon}})\alg$ of functors $(\Cxt{\scol{\Upsilon}}/E_{\scol{\Upsilon}})\to\Set$ which preserve all finite
  products is equivalent to the category of functors $\scol{P}\to\Set$ whose restriction to
  $\T$ preserves finite products.
\end{proposition}

\begin{example}
  Let $\cat{C}$ be the category presented by the terminal graph $G_0=\{X\}$, $G_1=\{f\}$, with
  equation $x:X\vdash x.f=x.f.f$. Let $\T$ be the algebraic theory of commutative rings, as in
  \cref{CommRingsex}. Consider the algebraic profunctor $M\colon\cat{C}\tickxar\T$ presented by
  a single attribute $\Upsilon=\{p\colon X\to\Int\}$ and a single equation $E=\{x:X\vdash
  x.f.p=x.p\times x.p\}$. One can check that this presents the following profunctor
  $\cat{C}\tickxar\T$:
    \begin{align*}
    \FrAlg{\Upsilon}/E_{\Upsilon}
    &\cong
    \mathbb{Z}[x.f^n.p]/\big(\;x.f^{n+1}.p=(x.f^n.p)^2\;,\;\;x.f^{n+2}.p=x.f^{n+1}.p \;\big)
    \\&\cong
    \mathbb{Z}[y_0,y_1]/\big(\;y_1=y_0^2\;,\;\;y_1=y_1^2\;\big)
    \\&\cong
    \mathbb{Z}[y]/\big(\;y^2=y^4\;\big)
  \end{align*}
  where, in the first line, $n$ ranges over all natural numbers. The edge $f\in G_1$ induces the
  ring endomorphism $f(y)\mapsto y^2$.
\end{example}

\section{Algebraic database schemas}\label{sec:schemas}

In this section we move beyond background and into our construction of databases. What we call
(algebraic) databases straddle what are traditionally known as relational databases and the more modern
graph databases. Importantly, algebraic databases also integrate a programming language $\Type$,
by which to operate on attribute values. 

We take our terminology from the relational database world. That is, a database consists of a conceptual 
layout, called a \emph{schema}, as well as some conforming data, called an \emph{instance} 
(because it represents our knowledge in the current instant of time). In this section we discuss the category of
schemas; in \cref{sec:instances} we discuss instances on them.

\subsection{Schemas}
For the rest of the paper, $\Type$ will be an arbitrary multi-sorted finitely presented algebraic theory,
as defined in \cref{def:alg_theory_presentation}. However, in all examples, we will fix $\Type$ to be
the algebraic theory described in \cref{Typedef}. Recall from \cref{def:profunctor_products} the notion of algebraic
profunctors $M\in\ProfTimes(\cat{C},\Type)$, denoted $M\colon\cat{C}\tickxar\Type$.

\begin{definition}
    \label{def:schema}
  A \emph{database schema} $\schema{S}$ over $\Type$ is a pair $(\snodes{S},\satts{S})$, where
  \begin{itemize}
  \item $\snodes{S}$ is a category, and
  \item $\satts{S}\colon\snodes{S}\tickxar\Type$ is an algebraic profunctor;
  i.e.\ $\satts{S}\in\ProfTimes(\snodes{S},\Type)$.
  \end{itemize}
  We refer to $\snodes{S}$ as the \emph{entity category} of $\schema{S}$ and to $\satts{S}$ as the
  \emph{observables} profunctor. We will also write
  $\satts{S}\colon\op{\snodes{S}}\to\TypeAlg$ for the exponential transpose of
  $\satts{S}\colon\op{\snodes{S}}\times\Type\to\Set$; see \cref{lem:profunctor_products}.
\end{definition}

\begin{remark}
    \label{collageschema}
  It is often convenient to work with schemas in terms of their collages. If $\schema{S}$ is the
  schema $\satts{S}\colon\snodes{S}\tickxar\Type$, we write $\scol{S}$ for the collage of the
  profunctor $\satts{S}$; see \cref{ex:prof_collages}. By \cref{equiv_prof_cat2}, it comes equipped
  with a map $\scol{S}\to\ncat{2}$ and we refer to the two pullbacks below respectively as the
  \emph{entity side} and the \emph{type side} of the collage:
  \begin{equation*}
    \begin{tikzcd}
      \snodes{S} \ar[r,"\incnodes{S}"] \ar[d,"!"'] \ar[dr,phantom,very near start, "\lrcorner"]
      & \scol{S} \ar[d]
      & \Type \ar[l,"\inctypes"'] \ar[d,"!"] \ar[dl,phantom,very near start, "\llcorner"] \\
      \{*\} \ar[r,"0"']
      & \ncat{2}
      & \{*\} \ar[l,"1"]
    \end{tikzcd}
  \end{equation*}
\end{remark}

\begin{example}
    \label{ex:schema_constants}
  Any $\Type$-algebra $X\colon\Type\to\Set$ can be regarded as a schema $(\singleton,X)$,
  where the entity category is terminal.
  In particular, the initial $\Type$-algebra $\kappa$, described in \cref{ex:representable_algebras}, can be viewed as a
  schema $\schema{U}=(\singleton,\kappa)$ called the \emph{unit schema}.%
  \footnote{$\schema{U}$ is the unit of a certain symmetric monoidal structure on $\Schema$
  (\cref{Schema}), whose restriction to entities is the cartesian monoidal structure on $\Cat$;
  however, we do not pursue that here.}
\end{example}

\subsection{Presentations of schemas}
  \label{Schemapresentation}

A presentation for a schema $\satts{S}\colon\snodes{S}\tickxar\Type$ is simply a presentation for the
category $\snodes{S}$ (see \cref{def:category_presentation}) together with a presentation for the
algebraic profunctor $\satts{S}$ (see \cref{def:presented_alg_prof}). We spell this out in
\cref{def:schema_presentation}.

\begin{definition}
    \label{def:schema_presentation}
  A \emph{schema signature} $\Xi=(G_{\Xi},\Upsilon_{\Xi})$ consists of a graph $G_{\Xi}$
  together with a profunctor signature $\Upsilon_{\Xi}$ from
  $G_{\Xi}$ to the signature of $\Type$.

  A \emph{schema presentation} $(\Xi,E_\Xi)$ consists of a schema signature $\Xi$, together with equations
  $E_\Xi=(\eqnodes{E},\eqatts{E})$, where $\eqnodes{E}$ is a set of unary equations over $G_{\Xi}$, and
  $\eqatts{E}$ is a set of profunctor equations over
  $\Upsilon_{\Xi}$. Note that $(G_{\Xi},\eqnodes{E})$ is a presentation for a category, which will be the
  entity category $\snodes{S}$, and $(\Upsilon_{\Xi},\eqatts{E})$ is a presentation for an algebraic profunctor
  $\snodes{S}\tickxar\Type$. We denote the presented schema by $\Fr(\Xi)/E_{\Xi}$.

  We will write $\scol{\Xi}$ to mean the associated algebraic signature $\scol{\Upsilon}_{\Xi}$ as in
  \cref{def:profunctor_presentation}, with sorts $(G_{\Xi})_0\sqcup S_{\Sigma}$ and
  function symbols $(G_{\Xi})_1\sqcup\Upsilon_{\Xi}\sqcup\Phi_{\Sigma}$,
  where $\Type\cong\Cxt{\Sigma}/E_{\Sigma}$.
\end{definition}

In what follows, we refer to function symbols in $\Upsilon_{\Xi}$ as \emph{attributes}, and refer to
a general term $(x:A)\vdash t:\tau$, where $A\in(G_{\Xi})_0$ and $\tau\in\Type$, as an \emph{observable
on $A$ of type $\tau$}. In other words, for a schema $\satts{S}\colon\snodes{S}\tickxar\Type$ and
objects $A\in \snodes{S}$ and $\tau\in\Type$, an observable on $A$ of type $\tau$ is an element
$t\in\satts{S}(A,\tau)$.

\begin{example}
The unit $\schema{U}=(\singleton,\kappa)$ of \cref{ex:schema_constants} is presented by the graph
with one vertex and no edges, the empty profunctor signature, and no equations.
That is, $\schema{U}$ has no attributes, so each of its
observables is a ground term $\emptyContext\vdash c:\tau$, i.e.\ a constant $c\in\kappa(\tau)=\Type(\One,\tau)$.
\end{example}

\begin{example}
    \label{Sschema}
  Let $\Type$ be as in \cref{Typedef}. Consider the presentation $(\Xi,\eqnodes{E},\eqatts{E})$ for
  a schema $\schema{S}$ as displayed in \cref{Spresentation}, which will serve as a motivating
  example throughout the paper. In this presentation, we use the labels "$\Ents$" for $(G_{\Xi})_0$,
  "$\Edges$" for $(G_{\Xi})_1$, "$\Atts$" for $\Upsilon_{\Xi}$, "$\PathEqs$" for $\eqnodes{E}$, and
  "$\ObsEqs$" for $\eqatts{E}$.
  \begin{figure}
    \centering
    \small
    \begin{tabular}{@{}l@{\hspace{1.5em}}r@{}} \toprule
      \begin{tabular}[t]{@{}ll@{}}
        \Ents:
        & $\EntOb{Emp}$,\ $\EntOb{Dept}$ \\
        \Edges:
        & \begin{tabular}[t]{@{}l@{${}:{}$}l@{${}\to{}$}l@{}}
            $\Mor{mgr}$ & $\EntOb{Emp}$&$\EntOb{Emp}$ \\
            $\Mor{wrk}$ & $\EntOb{Emp}$&$\EntOb{Dept}$ \\
            $\Mor{sec}$ & $\EntOb{Dept}$&$\EntOb{Emp}$ \\
          \end{tabular} \\
        \PathEqs:
        & \begin{tabular}[t]{@{}l@{${}={}$}l@{}}
            $e:\EntOb{Emp}\phantom{u}\vdash e\Mor{.mgr.mgr}$ & $e\Mor{.mgr}$ \\
            $e:\EntOb{Emp}\phantom{u}\vdash e\Mor{.mgr.wrk}$ & $e\Mor{.wrk}$ \\
            $d:\EntOb{Dept}\vdash d\Mor{.sec.wrk}$ & $d$ \\
          \end{tabular} \\
      \end{tabular}
      &
      \begin{tabular}[t]{@{}ll@{}}
        \Atts:
        &
          \begin{tabular}[t]{@{}l@{${}:{}$}l@{${}\to{}$}l@{}}
            $\Mor{last}$ & $\EntOb{Emp}$&$\Str$ \\
            $\Mor{name}$ & $\EntOb{Dept}$&$\Str$ \\
            $\Mor{sal}$ & $\EntOb{Emp}$&$\Int$ \\
          \end{tabular} \\
        \ObsEqs:
        & \begin{tabular}[t]{@{}l@{${}={}$}l@{}}
            \multicolumn{2}{@{}l@{}}{$e:\EntOb{Emp}\vdash$} \\
            $\quad(e.\Mor{sal}\leq e.\Mor{mgr.sal})$ & $\top$ \\
          \end{tabular} \\
      \end{tabular} \\ \bottomrule
    \end{tabular}
    \caption{Presentation of $\schema{S}$, our running example of a schema.}
    \labelold{Spresentation}
  \end{figure}

  Below \cref{Spicture} is a graphical display of this presentation; its two grey dots are the entities,
  its six arrows are the edges and attributes, and its four equations are the path and observable equations.
  \begin{equation}
      \label{Spicture}
    \begin{tikzpicture}[schema]
      \node[entity]                            (Dept) {Dept};
      \node[entity,above=of Dept]              (Emp)  {Emp};
      \node[type,right=1.2 of Emp]             (Int)  {Int};
      \node[type,below=of Int]                 (Str)  {Str};
      \node[type,below right=.5 and .5 of Int] (Bool) {Bool};
      \path (Emp)  edge["wrk"'name=wrk,bend right] (Dept)
                   edge["mgr"name=mgr,loop above]  (Emp)
                   edge["sal"]                     (Int)
                   edge["last",inner sep=2pt]      (Str)
            (Dept) edge["sec"',bend right,
                         inner sep=3pt,pos=.4]     (Emp)
                   edge["name"']                   (Str);
      \path (Bool.east |- Str.south) ++(9pt,0) node[anchor=south west,schema eqs] (Eqs) {
        $\Mor{mgr.mgr} = \Mor{mgr}$ \\
        $\Mor{mgr.wrk} = \Mor{wrk}$ \\
        $\Mor{sec.wrk} = \id$       \\[1ex]
        $(\Mor{sal}\leq\Mor{mgr}.\Mor{sal}) = \top$
      };
      \node[draw,fit=(Eqs) (wrk) (mgr)] (Sbox) {};
      \node[schema name,below left=of Sbox.north east] {$\schema{S}$};
    \end{tikzpicture}
  \end{equation}

  The presented schema $\schema{S}$ is built according to \cref{def:category_presentation,def:presented_alg_prof},
  as we now describe explicitly. The entity category $\snodes{S}$ is the free
  category on the subgraph of grey objects and arrows between them, modulo the top three equations.
  An example (context) morphism $\EntOb{Emp}\to\EntOb{Dept}$ in $\snodes{S}$ is given by the path
  $\Mor{mgr.wrk.sec.mgr.wrk}$. From the equations, we can show that it is equivalent to $\Mor{wrk}$,
  \begin{displaymath}
    e:\EntOb{Emp}\vdash\left(e\Mor{.mgr.wrk.sec.mgr.wrk}\approx e\Mor{.wrk}\right):\EntOb{Dept}
  \end{displaymath}
  In other words, these two terms name the same morphism in $\snodes{S}$.

  The observables profunctor $\satts{S}\colon\snodes{S}\tickxar\Type$ is freely generated by the three arrows from an
  $\snodes{S}$-object \tikz{\node[circle, draw, fill=black!20, inner sep=1mm] {};} to a
  $\Type$-object \tikz{\node[circle, draw, inner sep=1mm] {};}, modulo the fourth equation. An
  example observable $\EntOb{Dept}\to\Bool$, i.e.\ an element of
  $\satts{S}(\EntOb{Dept},\Bool)$, is "whether a department $d$ is named Admin", given by the term
  $(d:\EntOb{Dept})\vdash\mathrm{eq}(d\Mor{.name},\mathrm{Admin})$. By the fourth equation, we can
  show it is equivalent to a more complex observable,
  \[
    d:\EntOb{Dept}\vdash
    \left((d\Mor{.sec.sal}\leq d\Mor{.sec.mgr.sal})\wedge
    \mathrm{eq}(d\Mor{.name},\mathrm{Admin})\right):\Bool.
  \]

  The schema $\schema{S}$ can accommodate database instances in some company setting, as we will see
  in \cref{Sinstance}. In such, there exist tables of employees and departments. In each there are
  columns (sometimes called \emph{foreign keys}) that reference other tables in order to state where
  an employee works, who is the departmental secretary, etc. There are also columns that state the
  last name and salary of each employee, etc. The equations express integrity constraints, e.g.\ the
  fact that the secretary of a department works therein, or that every employee is paid less than
  his or her manager.
\end{example}

\subsection{Schema mappings} We now discuss morphisms of schemas, also known as schema mappings \cite{Doan:2012a}.
These will eventually be the vertical morphisms in a proarrow equipment. Recall by \cref{def:profunctor_products}
that a morphism between two algebraic profunctors is just a 2-cell between profunctors as in \cref{eqn:Prof2cells}.

\begin{definition}
    \label{sch_mapping}
  A \emph{schema mapping} $\map{F}\colon\schema{S}\to\schema{T}$ is a pair $(\mnodes{F},\matts{F})$, where
  \begin{itemize}
  \item $\mnodes{F}\colon\snodes{S}\to\snodes{T}$ is a functor, and
  \item $\matts{F}$ is a 2-cell in $\dProf$
    \begin{equation*}
      \begin{tikzcd}
        \snodes{S} \ar[r,tickx,"\satts{S}" domA] \ar[d,"\mnodes{F}"']
        & \Type \ar[d,equal] \\
        \snodes{T} \ar[r,tickx,"\satts{T}"' codA]
        & \Type
        \twocellA{\matts{F}}
      \end{tikzcd}
    \end{equation*}
  \end{itemize}
  We will also write $\matts{F}$ for the corresponding natural transformation
  $\matts{F}\colon\satts{S}\Rightarrow\satts{T}\circ\op{\mnodes{F}}$ of functors $\op{\snodes{S}}\to\Type\alg$.
  \end{definition}
  \begin{definition}\label{Schema}
    Define the \emph{category of schemas}, denoted $\Schema$, to have database schemas as objects
    and schema mappings as morphisms.
\end{definition}

\begin{remark}
    \label{collagefunctor}
  From the universal property of collages (\cref{def:collage}) in $\dProf$, it follows easily that a schema mapping
  $\map{F}\colon\schema{S}\to\schema{T}$ is equivalently a functor $\mcol{F}$ over $\ncat{2}$ between their collages,
  as in the left-hand diagram
  \begin{equation}\label{eq:mcol_diagrams}
  \begin{tikzcd}[column sep=small, row sep=3.5ex]
      \scol{S} \ar[rr,"\mcol{F}"]\ar[dr,"p"']
    && \scol{T}\ar[dl,"p'"] \\
    & \ncat{2} &
  \end{tikzcd}
  \qquad\qquad
  \begin{tikzcd}[row sep=4.5ex]
    \snodes{S} \ar[r,"\mnodes{F}"] \ar[d,"\incnodes{S}"']
    & \snodes{T} \ar[d,"\incnodes{T}"] \\
    \scol{S} \ar[r,"\mcol{F}"']
    & \scol{T}
  \end{tikzcd}
  \qquad
  \begin{tikzcd}[row sep=4.5ex]
    \Type \ar[d,"\inctypes"']\ar[r,equal]& \Type \ar[d,"\inctypes"] \\
    \scol{S} \ar[r,"\mcol{F}"']
    & \scol{T}
  \end{tikzcd}
  \end{equation}
  such that the middle and right-hand diagrams are the pullbacks of the left-hand diagram along the two maps
  $\ncat{1}\to\ncat{2}$; see \cref{equiv_prof_cat2} and \cref{collageschema}. By definition, a schema mapping acts as identity on the $\Type$-side of the collages.
\end{remark}

\begin{example}
    \label{sch_map_G}
  Consider the schema presentation given by the following graph, attributes, and equations:
  \begin{equation}\label{Tpicture}
  \begin{tikzpicture}[schema]
    \node[entity]                              (Dept) {Dept};
    \node[entity,above left=.5 and .8 of Dept] (QR)   {QR};
    \node[entity,above=of Dept]                (Emp)  {Emp};
    \node[type,right=1.2 of Emp]               (Int)  {Int};
    \node[type,below=of Int]                   (Str)  {Str};
    \node[type,below right=.5 and .5 of Int]   (Bool) {Bool};
    \path (QR)   edge["f",inner sep=3pt]        (Emp)
                 edge["g"',inner sep=3pt]       (Dept)
          (Emp)  edge["wrk"',bend right]        (Dept)
                 edge["mgr"name=mgr,loop above] ()
                 edge["sal"]                    (Int)
                 edge["last",
                      inner sep=2pt]            (Str)
          (Dept) edge["sec"',bend right,
                      inner sep=3pt,
                      pos=.4]                   (Emp)
                 edge["name"']                  (Str);
    \path (Bool.east |- Str.south) ++(9pt,0) node[anchor=south west,schema eqs] (Eqs) {
      $\Mor{mgr.mgr} = \Mor{mgr}$                 \\
      $\Mor{mgr.wrk} = \Mor{wrk}$                 \\
      $\Mor{sec.wrk} = \id$                       \\[1ex]
      $(\Mor{sal} \leq \Mor{mgr.sal}) = \top$     \\
      $(\Mor{f.sal} \leq \Mor{g.sec.sal}) = \top$ \\
      $\Mor{f.wrk.name} = \mathrm{Admin}$
    };
    \node[draw,fit=(Eqs) (QR) (mgr)] (Tbox) {};
    \node[schema name,below left=of Tbox.north east] {$\schema{T}$};
  \end{tikzpicture}
  \end{equation}
  The schema $\schema{T}$ which it presents, includes $\schema{S}$ of \cref{Sschema}.
  In addition it has a new entity
  \EntOb{QR} \dash named for its eventual role as a query result table in \cref{Queries} \dash as well as
  two new edges $\Mor{f},\Mor{g}$, and two new observable equations
  \begin{equation}\label{eqs_for_S}
  q:\EntOb{QR}\vdash(q\Mor{.f.sal}\leq q\Mor{.g.sec.sal})=\top,\quad
  q:\EntOb{QR}\vdash q\Mor{.f.wrk.name}=\mathrm{Admin}.
  \end{equation}
  Thus we have a schema inclusion
  $\map{G}\colon\schema{S}\to\schema{T}$, which of course restricts to identity on the $\Type$-side by definition.
\end{example}

\begin{example}
    \label{sch_map_F}
  We will now describe another schema mapping, with codomain the above schema
  $\schema{T}$. We will again do so in terms of presentations:
  \begin{displaymath}
  \begin{tikzpicture}[schema]
    \node[entity]                            (A)    {A};
    \node[type,above right=.5 and .8 of A]   (Int)  {Int};
    \node[type,below=of Int]                 (Str)  {Str};
    \node[type,below right=.5 and .5 of Int] (Bool) {Bool};
    \path (A) edge["diff",inner sep=3pt]     (Int)
              edge["emp\_last",bend left,
                   inner sep=2pt]            (Str)
              edge["dept\_name"',bend right,
                   inner sep=2pt]            (Str);
    \node[entity,right=1.8 of Bool] (QR)    {QR};
    \node[entity,below right=.5 and .8 of QR]  (Dept)  {Dept};
    \node[entity,above=of Dept]                (Emp)   {Emp};
    \node[type,right=1.2 of Emp]               (Int')  {Int};
    \node[type,below=of Int']                  (Str')  {Str};
    \node[type,below right=.5 and .5 of Int']  (Bool') {Bool};
    \path (QR)   edge["f",inner sep=3pt]  (Emp)
                 edge["g"',inner sep=3pt] (Dept)
          (Emp)  edge["wrk"',bend right]  (Dept)
                 edge["mgr"name=mgr,loop above] ()
                 edge["sal"]              (Int')
                 edge["last",
                      inner sep=2pt]      (Str')
          (Dept) edge["name"']            (Str')
                 edge["sec"',bend right,
                      inner sep=3pt,
                      pos=.4]             (Emp);
    \path (Str'.south east) ++(9pt,0) node[anchor=south west,schema eqs] (Eqs) {
        plus eqs\\from \cref{Tpicture}
    };
    \node[fit=(QR) (mgr) (Bool') (Dept)]             (Tbox) {};
    \node[schema name,below left=of Tbox.north east] {$\schema{T}$};
    \node[fit=(A) (Str) (Bool) (mgr.north -| Int)] (Rbox) {};
    \node[schema name,below left=of Rbox.north east] {$\schema{R}$};
    \draw[->,very thick] (Rbox.east) to node[above] {$\map{F}$} (Rbox.east -| Tbox.north west);
  \end{tikzpicture}
  \end{displaymath}
  The schema $\schema{R}$ has a terminal entity category $\snodes{R}=\{\EntOb{A}\}$, along with three generating
  attributes \dash namely $\Mor{diff}$, $\Mor{emp\_last}$, and $\Mor{dept\_name}$ \dash from the unique object to the base
  $\Type$-sorts $\Str$, $\Int$. Schema $\schema{T}$ has six
  equations, whereas $\schema{R}$ has none.

  The schema mapping $\map{F}\colon \schema{R}\to\schema{T}$ viewed as a functor $\mcol{F}\colon \scol{R}\to\scol{T}$,
  is defined to map the unique object $\EntOb{A}\in\snodes{R}$ to $\EntOb{QR}\in\snodes{T}$ on the
  entity side, and to map the three attributes to the following observables in $\schema{T}$:
  \begin{align*}
  \Mor{diff}\mapsto\Mor{g.sec.sal}-\Mor{f.sal}\qquad\Mor{emp\_last}\mapsto \Mor{f.last}\qquad \Mor{dept\_name}\mapsto\Mor{g.name}
  \end{align*}
  Since on the type side it is identity and $\schema{R}$ has no equations, there is nothing more to check; we have
  defined a schema mapping. This choice will be justified by \cref{pideltaquery}.
\end{example}

\section{Algebraic database instances}\label{sec:instances}

\subsection{Instances and transforms}

Given a database schema $\schema{S}$, which is a conceptual layout of entities and their attributes
(see \cref{def:schema}), we are ready to assign each entity a table full of data laid out according
to the schema. Such an assignment is called an \emph{instance} on $\schema{S}$; it is a set-valued
functor (copresheaf) of a certain form. Morphisms between instances are often called
(attribute-preserving) database homomorphisms \cite{Abiteboul:1995a}, but we call them
\emph{transforms} because they are nothing more than natural transformations.

\begin{definition}
    \label{inst_collage}
  Let $\schema{S}$ be a database schema, $\scol{S}$ its collage, and
  $\inctypes\colon\Type\to\scol{S}$ the inclusion of the type side (see \cref{collageschema}). An
  \emph{$\schema{S}$-instance} $\inst{I}$ is a functor $I\colon\scol{S}\to\Set$ such that the
  restriction $\iterms{I}\coloneqq I\circ\inctypes$ preserves finite products, i.e.\
  $\iterms{I}\colon\Type\to\Set$ is a $\Type$-algebra.

  Define the \emph{category of $\schema{S}$-instances}, denoted $\sinst{S}$, to be the full
  subcategory of the functor category $\Fun{\scol{S},\Set}$ spanned by the $\schema{S}$-instances. A
  morphism $\trans{\alpha}\colon\inst{I}\to\inst{J}$ of instances is called a \emph{transform}.
\end{definition}

\begin{example}
    \label{Sinstance}
  Recall the schema $\schema{S}$ generated by the presentation of \cref{Sschema}, which had
  employees and departments as entities, edges and attributes such as manager and salary, and
  equations such as an employee's salary must be less than that of his or her manager. A summary of
  an $\schema{S}$-instance $\inst{J}$ is displayed in \cref{fig:Sinstance}, with one table for each
  entity in $\schema{S}$, and with a column for each edge and attribute.
  \begin{figure}
    \centering
    \small
    \begin{tabular}[c]{@{}>{\itshape}c|c>{\itshape}c>{\itshape}cc@{}}
      \toprule
      \upshape\EntOb{Emp} & \Mor{last} & \upshape\Mor{wrk} & \upshape\Mor{mgr} & \Mor{sal} \\
      \midrule
      e1 & Gauss    & d3 & e1 & $250$ \\
      e2 & Noether  & d2 & e4 & $200$ \\
      e3 & Einstein & d1 & e3 & $300$ \\
      e4 & Turing   & d2 & e4 & $400$ \\
      e5 & Newton   & d3 & e1 & $100$ \\
      e6 & Euclid   & d2 & e7 & $150$ \\
      e7 & Hypatia  & d2 & e7 & $x$ \\
      \bottomrule
    \end{tabular}
    \hspace{3em}
    \begin{tabular}[c]{@{}>{\itshape}c|c>{\itshape}c@{}}
      \toprule
      \upshape\EntOb{Dept} & \Mor{name} & \upshape\Mor{sec} \\
      \midrule
      d1 & HR    & e3 \\
      d2 & Admin & e6 \\
      d3 & IT    & e5 \\
      \bottomrule
    \end{tabular}
    \caption{Example of an $\schema{S}$-instance $\inst{J}$.}
    \labelold{fig:Sinstance}
  \end{figure}

  All data required to determine an instance is encapsulated in the above two tables (image of the
  entity side and the attributes) along with a choice of $\Type$-algebra, which is generally
  infinite. Here, the $\Type$-algebra must include not only constants, but also all terms using the
  indeterminate $x:\Int$, which expresses Hypatia's unknown salary. Moreover, the equation
  $e:\EntOb{Emp}\vdash (e.\Mor{sal}\leq e.\Mor{mgr.sal})=\top$ in the presentation of $\schema{S}$
  implies that the terms $150\leq x$ and $\top$ must be equal in $J(\Bool)$ (by letting
  $e=\mathit{e6}$).

  Explicitly, we can define the functor $J\colon\scol{S}\to\Set$ as follows: the restriction
  $\iterms{J}=J\circ\inctypes$ to $\Type$ is the presented type algebra
  \begin{equation*}
    \iterms{J}\cong\FrAlg{x:\Int}/(150\leq x = \top),
  \end{equation*}
  and $J$ is defined on entities by the following sets:
  \begin{equation*}
    J(\EntOb{Emp})=\{\mathit{e1,e2,e3,e4,e5,e6,e7}\}, \quad
    J(\EntOb{Dept})=\{\mathit{d1,d2,d3}\},
  \end{equation*}
  and on edges and attributes by functions as shown in the table, e.g.\
  \begin{alignat*}{3}
    J(\Mor{wrk})\colon &J(\EntOb{Emp})\to J(\EntOb{Dept})
      &&\qquad\text{by}\qquad&&
      \mathit{e1}\mapsto\mathit{d3},\ldots,\mathit{e7}\mapsto\mathit{d2}\\
    J(\Mor{name})\colon &J(\EntOb{Dept})\to J(\Str)
      &&\qquad\text{by}\qquad&&
      \mathit{d1}\mapsto\mathrm{HR},\ldots,\mathit{d3}\mapsto\mathrm{IT}
  \end{alignat*}
\end{example}

\begin{example}
    \label{J'instance}
  Let $\schema{S}$ be as in \cref{Sschema}. Consider another $\schema{S}$-instance $\bar{\inst{J}}$,
  which is the same as $\inst{J}$ except that $\textit{e7}$ is removed from $\bar{J}(\EntOb{Emp})$
  and the restriction $\bar{\iterms{J}}=\bar{J}\circ\inctypes$ to $\Type$ is just $\kappa$, the
  algebra of constants as in \cref{ex:representable_algebras}. We will have use for both $\inst{J}$
  and $\bar{\inst{J}}$ later.
\end{example}

\begin{definition}
  We refer to instances whose $\Type$-algebra is initial, i.e.\ $\iterms{I}=\kappa$, as \emph{ground
  instances}. So $\inst{J}$ from \cref{Sinstance} is not a ground instance but $\bar{\inst{J}}$ from
  \cref{J'instance} is. If a $\Type$-algebra is presented by generators and relations, generators
  (such as the indeterminate value `$x$' of \cref{Sinstance}) are often referred to as
  \emph{labelled nulls} or \emph{Skolem variables} \cite{Abiteboul:1995a}. 
\end{definition}

Even though it wasn't defined that way, the category of instances $\sinst{S}$ can be seen to be the
category of algebras for an algebraic theory. Proving this, as we do next, immediately gives us
several nice properties of the category $\sinst{S}$.

\begin{proposition}
    \label{prop:SInst_as_algebras}
  For any schema $\schema{S}$, the category of instances $\sinst{S}$ is equivalent to the category
  of algebras for an algebraic theory.
\end{proposition}
\begin{proof}
  Recall from \cref{def:algebra} the category of algebras for a theory.
  We can consider $\scol{S}$ as a finite-product sketch,
  whose designated product cones are all
  finite products in $\Type$. Then a model of this sketch is a functor $\scol{S}\to\Set$ which
  preserves finite products in $\Type$, i.e.\ an instance of $\schema{S}$. The category of models
  for any finite product sketch is equivalent to the category of algebras for an algebraic theory
  generated by the sketch; see e.g.\ \cite[\S 4.3]{Barr.Wells:1985a}.
\end{proof}

\begin{remark}
    \label{rmk:SInst_as_algebras}
  When given a presentation $(\Xi,E_\Xi)$ for a schema $\schema{S}$, as in
  \cref{def:schema_presentation}, we can make \cref{prop:SInst_as_algebras} much more concrete:
  combining \cref{prop:collage_free_product_completion,inst_collage}, it follows that there is an
  equivalence of categories $\sinst{S}\equiv(\Cxt{\scol{\Xi}}/E_{\scol{\Xi}})\alg$.
\end{remark}

\begin{corollary}
    \label{corr:SInst_cocomplete}
  For any schema $\schema{S}$, the category of instances $\sinst{S}$ has all small colimits.
\end{corollary}
\begin{proof}
  This follows from \ref{prop:SInst_as_algebras} and \ref{thm:colimits_alg_theory}.
\end{proof}

\begin{remark}
    \label{rmk:SInst_colims}
  As in \cref{wrn:TAlg_colims}, we note that colimits in $\sinst{S}$ do not always agree with the
  pointwise colimits in $\Fun{\scol{S},\Set}$, which can make them difficult to work with.
  However, the following simple observation is sometimes useful:

  Let $X\colon\cat{D}\to\sinst{S}$ be a diagram and let $U(X)$ be its composite with the inclusion
  $\sinst{S}\to\Fun{\scol{S},\Set}$. If the colimit $\colim(UX)$ in $\Fun{\scol{S},\Set}$ lands in
  the subcategory $\sinst{S}$ (i.e.\ preserves products in $\Type$), then the natural map
  $\colim(UX)\to U(\colim X)$ is an isomorphism. In other words, in this case, the colimit $\colim
  X$ can be taken pointwise. (Note that this observation only uses the fact that $\sinst{S}$ is a
  full subcategory of $\Fun{\scol{S},\Set}$).
\end{remark}

\begin{example}\label{initial_schema}
  Let $\emptyset\in\Schema$ be the initial schema, i.e.\ the unique schema having an empty entity
  category. Then $\scol{\emptyset}\cong\Type$, thus there is an isomorphism of categories
  \[
    \sinst{\emptyset}\iso\TypeAlg.
  \]
\end{example}

\begin{remark}
    \label{representableinstance}
  Notice that for any schema $\schema{S}$, the Yoneda embedding
  $\yoneda\colon\op{\scol{S}}\to\Fun{\scol{S},\Set}$ is product-preserving and hence factors through
  the forgetful functor $\sinst{S}\to\Fun{\scol{S},\Set}$. The left factor
  $\op{\scol{S}}\to\sinst{S}$, which we also denote $\yoneda$, is fully faithful. In particular, for
  any object $s\in\scol{S}$, the representable functor $\yoneda(s)\colon\scol{S}\to\Set$ given by
  $\yoneda(s)(x)=\scol{S}(s,x)$ is an instance, called the \emph{$\schema{S}$-instance represented
  by $s$}.

  Because the functor $\sinst{S}\to\Fun{\scol{S},\Set}$ is fully faithful, it follows that the
  embedding $\yoneda\colon\op{\scol{S}}\to\sinst{S}$ is dense (see \cref{rmk:density}). In
  particular, for any instance $I\in\sinst{S}$, there is a canonical isomorphism of $\schema{S}$-instances
  \begin{equation*}
    I \iso \int^{s\in\scol{S}}I(s)\cdot\yoneda(s)
  \end{equation*}
which also follows from \cref{rmk:SInst_colims}.
\end{remark}

\subsection{Presentations of instances}
  \label{instancespresentation}

Let $(\Xi,E_\Xi)$ be a presentation of a schema $S$ (\cref{def:schema_presentation}) and
$\scol{\Xi}$ its associated algebraic signature, whose generated theory is the free product
completion of its collage $\scol{S}$ (\cref{prop:collage_free_product_completion}). By
\cref{rmk:SInst_as_algebras}, we can use presentations of algebras for a theory
(\cref{def:presented_algebra}) to give presentations of $\schema{S}$-instances.

\begin{definition}
    \label{def:instance_presentation}
  Let $\Gamma$ be a context over the above algebraic signature $\scol{\Xi}$. The free
  $(\Cxt{\scol{\Xi}}/E_{\scol{\Xi}})$-algebra $\FrAlg{\Gamma}$ corresponds under the equivalence of
  \cref{prop:SInst_as_algebras} to an $\schema{S}$-instance, which we denote $\FrInst{\Gamma}$, and
  call the \emph{free $\schema{S}$-instance generated by $\Gamma$}.

  If $E_{\Gamma}$ is a set of equations in context $\Gamma$, then we similarly write
  $\FrInst{\Gamma}/E_{\Gamma}$ for the $\schema{S}$-instance corresponding to
  $\FrAlg{\Gamma}/E_{\Gamma}$, and call it the $\schema{S}$-instance \emph{presented by
  $(\Gamma,E_{\Gamma})$}. Concretely, $\left(\FrInst{\Gamma}/E_{\Gamma}\right)(s)= {\{\text{terms
  in context $\Gamma$ of type $s\in\scol{S}$}\}}/{\sim}$, by \cref{rmk:term_alg}.
\end{definition}

\begin{example}
    \label{ex:representable_instance}
  For any object $s\in\scol{S}$, the representable instance $\yoneda(s)$ as in
  \cref{representableinstance} is free, with one generator of type $s$, i.e.\
  $\yoneda(s)\iso\FrInst{(x:s)}$.
\end{example}

\begin{remark}
    \label{rmk:canonical_presentation_inst}
  Similar to \cref{rmk:canonical_presentation_alg}, any given instance $\inst{I}$ has a canonical
  presentation, where for each $s\in\scol{S}$ and $x\in I(s)$ there is a generator $x:s$, and for
  each arrow $f\colon s\to s'$ in $\scol{S}$ with $y=I(f)(x)$, there is an equation $x:s\vdash
  x.f=y$. In this way, the presentation essentially records every entry of every column in every
  table.

  For example, the canonical presentation of $\inst{J}$ from \cref{Sinstance} has context and
  equations
  \begin{equation}
      \label{Jcanonicalpresentation}
    \begin{aligned}
      \Gamma &= (\mathit{e1,\dots,e7}:\EntOb{Emp},\,
                 \mathit{d1,d2,d3}:\EntOb{Dept},\,
                 x:\Int) \\
      E &= \{\mathit{e1}\Mor{.last} = \mathrm{Gauss},\,
             \mathit{e1}\Mor{.wrk}  = \mathit{d3},\,\ldots,\,
             \mathit{d3}\Mor{.sec}  = \mathit{e5} \}
    \end{aligned}
  \end{equation}
\end{remark}

\begin{example}
    \label{frozeninstance}
  Let $\schema{S}$ be as in \cref{Sschema}. We will now describe an $\schema{S}$-instance $\inst{I}$
  that is fairly different-looking than that in \cref{Sinstance} or \ref{J'instance}, in that the
  values of most of its attributes are non-constants. Instances like $\inst{I}$ play a central role
  in database queries (see \cref{Queries}).

  We specify the instance $\inst{I}$ by means of a presentation
  $\inst{I}=\FrInst{\Gamma}/E_{\Gamma}$, where $\Gamma=(e:\EntOb{Emp},d:\EntOb{Dept})$, and where
  $E_{\Gamma}$ contains the two equations
  \begin{equation}
      \label{eqs_for_I}
    \begin{aligned}
      \Gamma&\vdash e\Mor{.wrk.name}=\textrm{Admin} \\
      \Gamma&\vdash (e\Mor{.sal}\leq d\Mor{.sec.sal})=\top.
    \end{aligned}
  \end{equation}
  Thus for any entity or type $s\in\scol{S}$, the elements of $I(s)$ are the equivalence classes of terms
  $\Gamma\vdash t:s$ built out of edges and attributes from $\schema{S}$ and function symbols from
  $\Type$, modulo the equations $E_{\Gamma}$ as well as those from $\schema{S}$.

  We can picture this instance in the tables shown in \cref{fig:frozeninstance}.
  \begin{figure}
    \centering
    \small
    \begin{tabular}[t]{@{}>{$}c<{$}|>{$}c<{$}>{$}c<{$}>{$}c<{$}>{$}c<{$}@{}}
      \toprule
      \EntOb{Emp}         & \Mor{last}           & \Mor{wrk}           & \Mor{mgr}       & \Mor{sal}           \\
      \midrule
      e                   & e.\Mor{last}         & e.\Mor{wrk}         & e.\Mor{mgr}     & e.\Mor{sal}         \\
      e.\Mor{mgr}         & e.\Mor{mgr.last}     & e.\Mor{wrk}         & e.\Mor{mgr}     & e.\Mor{mgr.sal}     \\
      d.\Mor{sec}         & d.\Mor{sec.last}     & d.\Mor{sec.wrk}=d   & d.\Mor{sec.mgr} & d.\Mor{sec.sal}     \\
      d.\Mor{sec.mgr}     & d.\Mor{sec.mgr.last} & d                   & d.\Mor{sec.mgr} & d.\Mor{sec.mgr.sal} \\
      e.\Mor{wrk.sec}     & e.\Mor{wrk.sec.last} & e.\Mor{wrk}         & \dots           & \dots               \\
      e.\Mor{wrk.sec.mgr} & \dots                & \dots               & \dots           & \dots               \\
      \bottomrule
    \end{tabular} \\
    \vspace{1ex}
    \begin{tabular}[t]{@{}>{$}c<{$}|>{$}c<{$}>{$}c<{$}@{}}
      \toprule
      \EntOb{Dept} & \Mor{name}     & \Mor{sec}       \\
      \midrule
      d            & d.\Mor{name}   & d.\Mor{sec}     \\
      e.\Mor{wrk}  & \mathrm{Admin} & e.\Mor{wrk.sec} \\
      \bottomrule
    \end{tabular}
    \caption{Example of a presented $\schema{S}$-instance $\inst{I}$.}
    \labelold{fig:frozeninstance}
  \end{figure}
  On types, $I$ contains many terms, as in
  \begin{align*}
    I(\Str) &= \{e\Mor{.last},e\Mor{.mgr.last},d\Mor{.sec.last},\dots,
      d\Mor{.name},\mathrm{Admin},\dots,\mathrm{aaBcZ},\dots\} \\
    I(\Int) &= \{e\Mor{.sal},e\Mor{.mgr.sal},\dots,-d\Mor{.sec.sal},e\Mor{.sal}+e\Mor{.sal}+1,\dots,
      28734,\dots\} \\
    I(\Bool) &= \{\mathrm{eq}(e\Mor{.last},d\Mor{.name}),\dots,
      e\Mor{.mgr.sal}\leq d\Mor{.sec.sal},\dots,\top,\neg\top\}
  \end{align*}
  Note, for example, that the value of $e\Mor{.wrk.name}$ in the table $\EntOb{Dept}$ has been
  replaced by `Admin' because of an equation of $\inst{I}$ and that the value of
  $e\Mor{.wrk.sec.wrk}$ has been replaced by $e\Mor{.wrk}$ because of a path equation of
  $\schema{S}$ (see \cref{Sschema}).
\end{example}

\begin{example}
    \label{frozeninstancetransform}
  Having defined two $\schema{S}$-instances $\inst{J}$ and $\inst{I}$ in the examples
  \ref{Sinstance} and \ref{frozeninstance} above, we will now explicitly describe the set
  $\sinst{S}(\inst{I},\inst{J})$ of instance transforms between them.

  By \cref{rmk:SInst_as_algebras,rmk:canonical_presentation_alg}, the set of transforms
  $\inst{I}\to\inst{J}$ is equivalent to the set of morphisms of presentations from the presentation
  $(\Gamma,E_{\Gamma})$ of $\inst{I}$ to the canonical presentation of $\inst{J}$. Such a morphism
  of presentations is simply an assignment of an element $\epsilon\in J(\EntOb{Emp})$ to the
  generator $e:\EntOb{Emp}$ and an element $\delta\in J(\EntOb{Dept})$ to the generator
  $d:\EntOb{Dept}$, such that the two equations $\epsilon.\Mor{wrk.name}=\mathrm{Admin}$ and
  $(\epsilon.\Mor{sal}\leq\delta.\Mor{sec.sal})=\top$ are true in $\inst{J}$.

  Without the equations, there would be 21 assignments
  $(\epsilon,\delta)\in J(\EntOb{Emp})\times J(\EntOb{Dept})$. It is easy to check that
  only three of those 21 satisfy the two equations: $(\mathit{e6,d1})$, $(\mathit{e2, d1})$, and
  $(\mathit{e6, d2})$. For instance, the equation $e\Mor{.wrk.name}=\mathrm{Admin}$ means we must
  have $\epsilon.\Mor{wrk}=\mathit{d2}$. Similarly, the equation $\big((e\Mor{.sal})\leq
  d\Mor{.sec.sal}\big)=\top$ rules out several choices. For example, the assignment
  $(\epsilon,\delta)\coloneq(\mathit{e7,d1})$ is invalid because we cannot deduce that $x\leq 300$
  from any equations of $\inst{J}$ (where we only know that $150\leq x$).
\end{example}

\begin{example}
    \label{ex:presented_transform}
  We will now consider transforms between two instances that are both presented. As usual, let
  $\schema{S}$ be the schema of \cref{Sschema}, and let $\inst{I}$ be the instance from above,
  \cref{frozeninstance}. We recall its presentation $(\Gamma_{\schema{I}},E_{\schema{I}})$, as well
  as present a new $\schema{S}$-instance $\inst{I'}$:
  \begin{align*}
    \Gamma_{\schema{I'}} &= \{e':\EntOb{Emp}\}
    & \Gamma_{\schema{I}} &= \{e:\EntOb{Emp},d:\EntOb{Dept}\} \\
    E_{\schema{I'}} &=
      \begin{aligned}[t]
        \{&e'\Mor{.wrk.name}=\tn{Admin} \\
          &(e'\Mor{.sal}\leq e'\Mor{.wrk.sec.sal})=\top\}
      \end{aligned}
    & E_{\schema{I}} &=
      \begin{aligned}[t]
        \{&e\Mor{.wrk.name}=\tn{Admin} \\
          &(e\Mor{.sal}\leq d\Mor{.sec.sal})=\top\}
      \end{aligned}
  \end{align*}
  As in \cref{def:presented_algebra}, to give an instance transform
  $\beta\colon\inst{I'}\Rightarrow\inst{I}$, it is equivalent to give a context morphism
  $\Gamma_I\to\Gamma_{I'}$ (see
  \cref{def:context_morphism}) in the opposite direction which respects the equations. In this case,
  there are only two which satisfy the equations: $[e'\coloneq e]$ and $[e'\coloneq
  e.\Mor{wrk.sec}]$.
\end{example}

\subsection{Decomposing instances}
  \label{decomposing_instances}

While \cref{inst_collage} is how we most often consider instances, it will sometimes be useful to
consider their entity and attribute parts separately. Recall the left action $\otimes$
of \cref{def:otimes} on algebraic profunctors.

An instance $\inst{I}$ on a schema $\schema{S}$ is equivalently defined to be a tuple
$(\inodes{I},\iterms{I},\iatts{I})$, where
\begin{itemize}[nosep]
  \item $\inodes{I}\colon\singleton\tickar\snodes{S}$ is a profunctor, called the \emph{entity side}
    of $\inst{I}$,
  \item $\iterms{I}\colon\singleton\tickar\Type$ is an algebraic profunctor, called the \emph{type
    side} of $\inst{I}$, and
  \item $\iatts{I}\colon\inodes{I}\otimes\satts{S}\to\iterms{I}$ is a profunctor morphism, called the \emph{values assignment} for $\inst{I}$:
    \begin{equation*}
      \begin{tikzcd}[row sep=0em]
        & |[alias=domA]| \snodes{S} \ar[dr,bend left=15,tickx,"\satts{S}"] & \\
        \singleton \ar[ur,bend left=15,tick,"\inodes{I}"] \ar[rr,bend right=20,tickx,"\iterms{I}"' codA]
        && \Type
        \twocellA{\iatts{I}}
      \end{tikzcd}
    \end{equation*}
\end{itemize}
The functor $I\colon\scol{S}\to\Set$ of \cref{inst_collage}, viewed as $I\colon\singleton\tickar\scol{S}$,
can then be uniquely recovered by the lax limit universal property of $\scol{S}$
spelled out in \cref{proarrowsinoutcollage}, for $X=\singleton$.

Note that the entity side $\inodes{I}\colon\snodes{S}\to\Set$ is just a copresheaf, the type side
$\iterms{I}\colon\Type\to\Set$ is just a $\Type$-algebra, and the values assignment $\iatts{I}$ is
equivalent to a morphism $\int^{s\in\snodes{S}}\inodes{I}(s)\cdot\satts{S}(s) \to \iterms{I}$ of
$\Type$-algebras (where the coend is in $\Type\alg$, see \cref{thm:colimits_alg_theory}).
We could also obtain $I\colon\scol{S}\to\Set$ from collages universal property \cref{univpropcollages}
in $\dProf$.

Similarly, a transform $\inst{I}\to\inst{J}$ between instances can equivalently be defined in terms
of separate entity and type components $(\tnodes{\alpha},\tterms{\alpha})$, where
$\tnodes{\alpha}\colon\inodes{I}\Rightarrow\inodes{J}$ and
$\tterms{\alpha}\colon\iterms{I}\Rightarrow\iterms{J}$ are profunctor morphisms, satisfying the
equation:
\begin{equation*}
  \begin{tikzcd}
    \{*\} \ar[r,bend left=55,tick,"\inodes{I}" domA]
    \ar[r,tick,"\inodes{J}"' codA]
    \ar[rr,bend right=45,tickx,"\iterms{J}"' codB]
    & |[alias=domB]| \snodes{S} \ar[r,tickx,"\satts{S}"]
    & \Type
    \twocellA{\tnodes{\alpha}}
    \twocellB{\iatts{J}}
  \end{tikzcd}
  \quad = \quad
  \begin{tikzcd}[row sep=.6em,baseline=(B.base)]
    & |[alias=domA]| \snodes{S} \ar[dr,bend left=20,tickx,"\satts{S}"] & \\
    |[alias=B]| \{*\} \ar[ur,bend left=20,tick,"\inodes{I}"]
    \ar[rr,tickx,"\iterms{I}"' {codA,domB}]
    \ar[rr,bend right=45,tickx,"\iterms{J}"' codB]
    && \Type
    \twocellA{\iatts{I}}
    \twocellB[pos=.6]{\tterms{\alpha}}
  \end{tikzcd}
\end{equation*}
Given $\alpha\colon I\Rightarrow J\colon \scol{S}\to\Set$, the entity and type components
$\tnodes{\alpha}$ and $\tterms{\alpha}$ are simply the restrictions of $\alpha$ along the collage
inclusions $\incnodes{S}\colon\snodes{S}\to\scol{S}\from\Type\cocolon\inctypes$. In the other
direction, given $\tnodes{\alpha}$ and $\tterms{\alpha}$, one recovers $\alpha$ by the 2-dimensional
part of the universal property of \cref{proarrowsinoutcollage}.

\section{The fundamental data migration functors}
  \label{ssec:DataMigrationFunctors}

In this section, we describe functors that transfer instances from one schema to another. More
specifically, we show how any schema mapping $\map{F}\colon\schema{S}\to\schema{T}$ induces a system
of three adjoint functors
\begin{equation*}
  \begin{tikzcd}
    \sinst{T}\ar[rr,"{\Delta_F}"description] &&
    \sinst{S}\ar[ll,bend left=21,"{\Pi_F}","{\bot}"']
    \ar[ll,bend right=21,"{\Sigma_F}"',"{\bot}"]
  \end{tikzcd}
\end{equation*}
which we call \emph{data migration functors}. They are related to the usual Kan extensions setting
between categories of presheaves. Recall from \cref{sch_mapping} that a schema mapping
$\map{F}\colon\schema{S}\to\schema{T}$ is a functor $\mnodes{F}\colon\snodes{S}\to\snodes{T}$ and a
2-cell $\matts{F}$:
\[
  \begin{tikzcd}
    \snodes{S} \ar[r,tickx,"\satts{S}" domA] \ar[d,"\mnodes{F}"']
    & \Type \ar[d,equal] \\
    \snodes{T} \ar[r,tickx,"\satts{T}"' codA]
    & \Type
    \twocellA{\matts{F}}
  \end{tikzcd}
\]

\begin{definition}
    \label{Deltadef}
  Let $\map{F}\colon\schema{S}\to\schema{T}$ be a schema mapping, and let
  $\scol{F}\colon\scol{S}\to\scol{T}$ be the induced map on collages (\cref{collagefunctor}). We
  define a functor $\Delta_F\colon\sinst{T}\to\sinst{S}$ as follows:
  \begin{itemize}
  \item For any instance $\inst{I}$ of $\schema{T}$, define $\Delta_F(\inst{I})\coloneqq
    I\circ\mcol{F}$. By \cref{eq:mcol_diagrams} and \cref{inst_collage}, the following diagram
    commutes:
    \begin{equation*}
      \begin{tikzcd}
        & \Type \ar[dl,"\inctypes"'] \ar[d,"\inctypes"] \ar[dr,"\iterms{I}"] & \\
        \scol{S} \ar[r,"\mcol{F}"']
        & \scol{T} \ar[r,"I"']
        & \Set
      \end{tikzcd}
    \end{equation*}
    Thus $\iterms{\Delta_F(\inst{I})}=I\circ\mcol{F}\circ\inctypes=\iterms{I}$ preserves
    products.
    \item For any $\trans{\alpha}\colon\inst{I}\to\inst{J}$ in $\sinst{T}$, define
    $\Delta_F(\alpha)=\alpha\circ\mcol{F}$.
  \end{itemize}
  We call $\Delta_F$ the \emph{pullback functor} (along $F$).
\end{definition}

\begin{example}
	\label{ex:underlying_typealg}
  For any schema $\schema{S}$, the unique map $!\colon\emptyset\to\schema{S}$ from the initial
  schema (\cref{initial_schema}) induces a functor $\Delta_!\colon\sinst{S}\to\TypeAlg$,
  denoted $\Delta_{!_S}$. For an instance $\inst{I}$ of
  $\schema{S}$, this functor returns the underlying $\Type$-algebra of the instance,
  $\Delta_{!_S}(\inst{I})\cong\iterms{I}$.
\end{example}

A schema mapping $\map{F}$ can be considered as a map of finite product sketches; see
\cref{prop:SInst_as_algebras}. In general one does not expect the pullback functor $\Delta_F$
between the corresponding categories of algebras to have a right adjoint; for example, there is no
`cofree monoid' on a set. However, because $\map{F}$ restricts to the identity on the $\Type$-side of
the collage by \cref{collagefunctor}, we find that $\Delta_F$ does have a right adjoint, denoted
$\Pi_F$, which we call the \emph{right pushforward functor}.

\begin{proposition}
    \label{Pidef}
  Let $\map{F}\colon\schema{S}\to\schema{T}$ be a schema mapping. The right Kan extension
  $\Ran_{\mcol{F}}\colon$ $\Fun{\scol{S},\Set}\to\Fun{\scol{T},\Set}$ takes $\schema{S}$-instances
  to $\schema{T}$-instances, defining a right adjoint to $\Delta_F$,
  \[
    \Delta_F\colon\sinst{T}\leftrightarrows\sinst{S}\cocolon\Pi_F
  \]
\end{proposition}
\begin{proof}
  Let $I\colon\scol{S}\to\Set$ be any functor. Then for an object $x\in\scol{T}$, the right Kan
  extension is given by
  \begin{equation}\label{formulaPi}
  \begin{aligned}
    \Pi_F(I)(x) \coloneqq (\Ran_{\mcol{F}} I)(x)
    &\cong \BigFun{\scol{T},\Set}\big(\scol{T}(x,-),\Ran_{\mcol{F}}(I)\big) \\
    &\cong \BigFun{\scol{S},\Set}\big(\scol{T}(x,\mcol{F}-),I\big) \\
    &\cong \int_{s\in\scol{S}}\;\;{I(s)^{\scol{T}(x,\mcol{F}s)}}.
  \end{aligned}
  \end{equation}

  We will show that this formula preserves the property of $\iterms{I}$ being product-preserving. In
  fact, it preserves the $\Type$-algebra exactly, i.e.\ the diagram on the left commutes (up to
  natural isomorphism)
  \begin{equation}
      \label{eqn:same_typeside}
    \begin{tikzcd}[column sep=0em]
      {\BigFun{\scol{S},\Set}} \ar[rr,"\Ran_{\mcol{F}}"] \ar[dr,"\textrm{--}\circ\inctypes"']
      && {\BigFun{\scol{T},\Set}} \ar[dl,"\textrm{--}\circ\inctypes"] \\
      & {\BigFun{\Type,\Set}} &
    \end{tikzcd}
    \qquad
    \begin{tikzcd}
      \Type \ar[r,"\id"] \ar[d,"\inctypes"'] & \Type \ar[d,"\inctypes"] \\
      \scol{S} \ar[r,"\mcol{F}"'] & \scol{T}.
    \end{tikzcd}
  \end{equation}
  or equivalently, the pullback square on the right satisfies the Beck-Chevalley condition for right
  Kan extensions. The latter follows formally because the inclusion
  $\inctypes\colon\Type\to\scol{T}$ is an opfibration, but we can easily check the commutativity of
  the left diagram directly: for any $\tau\in\Type$, \cref{formulaPi} gives
  \begin{equation*}
    (\Ran_{\mcol{F}} I)(\tau)\iso\left[\scol{S},\Set\right]\left(\scol{T}(\tau,\mcol{F}\textrm{--}),I\right)
    \iso\BigFun{\scol{S},\Set}\left(\scol{S}(\tau,\textrm{--}),I\right)\iso I(\tau),
  \end{equation*}
  completing the proof.
\end{proof}

We now define the \emph{left pushforward functor}, denoted $\Sigma_F$.

\begin{proposition}
    \label{Sigmadef}
  For any schema mapping $\map{F}\colon\schema{S}\to\schema{T}$, the functor $\Delta_F$ has a left
  adjoint
  \[
    \Sigma_F\colon\sinst{S}\leftrightarrows\sinst{T}\cocolon\Delta_F.
  \]
  If $\inst{I}\in\sinst{S}$ is an instance, then $\Sigma_F(\inst{I})$ is given by the following
  coend taken in $\sinst{T}$:
  \begin{equation}\label{formulaSigma}
    \Sigma_F(\inst{I}) \cong \int^{s\in\scol{S}} I(s)\cdot\yoneda(\mcol{F}s),
  \end{equation}
  where $\yoneda(\mcol{F}s)$ is the representable $\schema{T}$-instance $\scol{T}(\mcol{F}s,-)$,
  see \cref{representableinstance}.

  In other words, $\Sigma_F$ is the left Kan extension
  \begin{equation}
      \label{eqn:sigma_extension}
    \begin{tikzcd}[column sep=5em]
      \op{\scol{S}} \ar[r,"\op{\mcol{F}}"] \ar[d,"\yoneda"']
      & \op{\scol{T}} \ar[d,"\yoneda"] \\
      \sinst{S} \ar[r,"\Sigma_F=\Lan_{\yoneda}(\yoneda\circ\mcol{F})"']
      & \sinst{T}
    \end{tikzcd}
  \end{equation}
  and the above square in fact commutes.
\end{proposition}
\begin{proof}
  The coend exists because $\sinst{T}$ is cocomplete (\cref{corr:SInst_cocomplete}). It is simple to
  check that this defines a left adjoint to $\Delta_F$:
  \begin{align*}
    \sinst{T}\left(\int^{s\in\scol{S}} I(s)\cdot\yoneda(\op{\mcol{F}}(s)), J \right)
      & \iso \int_{s\in\scol{S}} \sinst{T}\bigl( I(s)\cdot\yoneda(\op{\mcol{F}}(s)), J \bigr) \\
      & \iso \int_{s\in\scol{S}} \Set\bigl( I(s), \sinst{T}(\yoneda(\op{\mcol{F}}(s)),J) \bigr) \\
      & \iso \int_{s\in\scol{S}} \Set\bigl(I(s),J(\mcol{F}(s))\bigr) \\
      & \iso \BigFun{\scol{S},\Set}(I,J\circ\mcol{F}) = \sinst{S}(I,\Delta_F(J)).
  \end{align*}
  The square commutes since $\Sigma_F\colon\sinst{S}\to\sinst{T}$ is a pointwise Kan extension along the fully faithful $\yoneda$.
\end{proof}

\begin{remark}
    \label{SigmaIpresent}
  The coend \cref{formulaSigma} is not typically a pointwise colimit, as pointed out in
  \cref{rmk:SInst_colims}. Hence, unlike $\Pi$ or
  $\Delta$, given an object $t\in\scol{T}$ there is in general no explicit formula for computing the
  set $(\Sigma_F(I))(t)$.

  However, obtaining the presentation of $\Sigma(I)$ from a presentation of $I$ is almost trivial:
  if $\inst{I}$ is presented by a context $\Gamma=(x_1:s_1,...,x_n:s_n)$ and some equations, then
  $\Sigma_F(\inst{I})$ is presented by the context
  $F(\Gamma)=(x_1:\mcol{F}(s_1),...,x_n:\mcol{F}(s_n))$ and respective equations by applying
  $\mcol{F}$ to edges, attributes of the term expressions.
\end{remark}

\begin{remark}\label{PiDeltapreservetypes}
  For any schema mapping $\map{F}\colon\schema{S}\to\schema{T}$, one can check using
  \cref{eqn:same_typeside,Deltadef} that the functors $\Pi_F$ and $\Delta_F$ preserve $\Type$-algebras,
  in the sense that $\iterms{(\Pi_F I)}\iso\iterms{I}$ and $\iterms{(\Delta_FJ)}\iso\iterms{J}$.
  This does not generally hold for $\Sigma$; in \cref{prop:sigma_preserves_when} we give a simple
  criterion for when it does.
\end{remark}

\begin{example}
    \label{Sigmaexample}
  We will give an example of the application of the left pushforward functor
  $\Sigma_H\colon\sinst{S}\to\sinst{L}$ on $\inst{J}$ from \cref{Sinstance}, for a schema mapping
  $\map{H}\colon\schema{S}\to\schema{L}$ as follows:
  \begin{displaymath}
  \begin{tikzpicture}[schema]
    \node[entity]                            (Dept) {Dept};
    \node[entity,above=of Dept]              (Emp)  {Emp};
    \node[type,right=1 of Emp]               (Int)  {Int};
    \node[type,below=of Int]                 (Str)  {Str};
    \node[type,below right=.5 and .3 of Int] (Bool) {Bool};
    \path (Emp)   edge["wrk"'name=wrk,bend right] (Dept)
                  edge["mgr"name=mgr,loop above]  (Emp)
                  edge["sal"]                     (Int)
                  edge["last",inner sep=2pt]      (Str)
          (Dept)  edge["name"']                   (Str)
                  edge["sec"',bend right,
                       inner sep=3pt,pos=.4]      (Emp);
    \node[entity,right=2 of Bool]               (Team)  {Team};
    \node[entity,below right=.5 and .8 of Team] (Dept') {Dept};
    \node[entity,above=of Dept']                (Emp')  {Emp};
    \node[type,right=1 of Emp']                 (Int')  {Int};
    \node[type,below=of Int']                   (Str')  {Str};
    \node[type,below right=.5 and .3 of Int']   (Bool') {Bool};
    \path (Emp')  edge["on"',inner sep=2pt]          (Team)
                  edge["wrk"',bend right,
                       inner sep=2pt,pos=.4]         (Dept')
                  edge["mgr"name=mgr',loop above]    (Emp')
                  edge["sal"]                        (Int')
                  edge["last",inner sep=2pt]         (Str')
          (Team)  edge["bel",inner sep=2pt,
                       pos=.4]                       (Dept')
                  edge["col"'name=col,bend right=50,
                       inner sep=2pt]                (Str')
          (Dept') edge["sec"',bend right,
                       inner sep=2pt,pos=.4]         (Emp')
                  edge["name"]                       (Str');
    \path (col.south -| Str'.east) ++(9pt,0) node[anchor=south west,schema eqs] (Eqs') {
      $\Mor{mgr.on}=\Mor{on}$ \\
      $\Mor{on.bel}=\Mor{wrk}$ \\[1ex]
      plus eqs\\from \cref{Spicture}
    };
    \path (Eqs'.south -| Str.east) ++(9pt,0) node[anchor=south west,schema eqs] (Eqs) {
      plus eqs\\from \cref{Spicture}
    };
    \node[fit=(Team) (mgr') (Eqs') (col)] (Lbox) {};
    \node[schema name,below left=of Lbox.north east] {$\schema{L}$};
    \node[fit=(Eqs) (wrk) (mgr) (col.south -| Str)] (Sbox) {};
    \node[schema name,below left=of Sbox.north east] {$\schema{S}$};
    \draw[->,very thick] (Sbox) to node[above] {$\map{H}$} (Lbox);
  \end{tikzpicture}
  \end{displaymath}

  Schema $\schema{S}$ is as in \cref{Sschema}, and schema $\schema{L}$ has a new entity
  `\EntOb{Team}', thought of as grouping employees into teams, which have a color-name and belong to
  some department. The two new equations ensure that an employee is on the same team as their
  manager and that their team belongs to their department.

  The functor $\mcol{H}\colon\scol{S}\to\scol{L}$ is an inclusion, preserving labels
  ($\mcol{H}(\EntOb{Emp})=\EntOb{Emp}$, etc.). Thus, by \cref{SigmaIpresent}, we find that the
  presentation of $\Sigma_H(\inst{J})$ is exactly that of $\inst{J}$, shown in
  \cref{Jcanonicalpresentation}, only now interpreted as a $\schema{L}$-instance presentation. To
  calculate the $\schema{L}$-instance it presents, one follows the explanation from
  \cref{instancespresentation} (as we explain briefly below) and finds that $\Sigma_H(\inst{J})$ is
  given by the tables shown in \cref{fig:Sigmaexample},
  \begin{figure}
    \centering
    \small
    \begin{tabular}[c]{@{}>{\itshape}c|>{$}c<{$}>{\itshape}c@{}}
      \toprule
      \upshape\EntOb{Team} & \Mor{col} & \upshape\Mor{bel} \\
      \midrule
      t1 & \mathit{t1}.\Mor{col} & d3 \\
      t2 & \mathit{t2}.\Mor{col} & d2 \\
      t3 & \mathit{t3}.\Mor{col} & d1 \\
      t4 & \mathit{t4}.\Mor{col} & d2 \\
      \bottomrule
    \end{tabular}\qquad
    \begin{tabular}[c]{@{}>{\itshape}c|c>{\itshape}c>{\itshape}cc>{\itshape}c@{}}
      \toprule
      \upshape\EntOb{Emp} & \Mor{last} & \upshape\Mor{wrk} & \upshape\Mor{mgr} & \Mor{sal} & \upshape{\Mor{on}} \\
      \midrule
      e1 & Gauss    & d3 & e1 & $250$ & t1\\
      e2 & Noether  & d2 & e4 & $200$ & t2\\
      e3 & Einstein & d1 & e3 & $300$ & t3\\
      e4 & Turing   & d2 & e4 & $400$ & t2\\
      e5 & Newton   & d3 & e1 & $100$ & t1\\
      e6 & Euclid   & d2 & e7 & $150$ & t4\\
      e7 & Hypatia  & d2 & e7 & $x$   & t4\\
      \bottomrule
    \end{tabular}\qquad
    \begin{tabular}[c]{@{}>{\itshape}c|c>{\itshape}c@{}}
      \toprule
      \upshape\EntOb{Dept} & \Mor{name} & \upshape\Mor{sec} \\
      \midrule
      d1 & HR    & e3 \\
      d2 & Admin & e6 \\
      d3 & IT    & e5 \\
      \bottomrule
    \end{tabular}
    \caption{The left pushforward instance, $\Sigma_H(\inst{J})\in\sinst{L}$.}
    \labelold{fig:Sigmaexample}
  \end{figure}
  where $\mathit{t1,t2,t3,t4}$ are freely generated terms. The $\Type$-algebra of this instance
  is larger than that of $\inst{I}$; it includes, for example, new terms
  $\mathit{t1}\Mor{.col},\ldots,\mathit{t4}\Mor{.col}:\Str$.

  To calculate the set of rows in the \EntOb{Team} table, following \cref{def:instance_presentation}
  one freely adds a new team for each employee to be on, but quotients by setting each employees
  team equal to that of his or her manager, due to the equation
  $(\Mor{mgr.on}=\Mor{on}):\EntOb{Team}$ in schema $\schema{L}$. Notice how we have one team
  belonging to HR and one team belonging to IT, but two teams belonging to Admin. This basically
  results from the freeness of the construction and the fact that there are two different managers,
  Turing and Hypatia, in Admin. The colors assigned to these teams are freely assigned as
  indeterminate string values (e.g.\ $\mathit{t1}.\Mor{col}$) in those cells.
\end{example}

\begin{example}
    \label{Piexample}
  Recall the schema mapping $\map{G}\colon\schema{S}\to\schema{T}$ from \cref{sch_map_G}, which is
  given by the inclusion of the $\schema{S}$-presentation (\ref{Spicture}) into the
  $\schema{T}$-presentation (\ref{Tpicture}). We are going to describe the effect of the induced
  right pushforward functor $\Pi_G\colon\sinst{S}\to\sinst{T}$ on the $\schema{S}$-instance
  $\inst{J}$ of \cref{Sinstance}.

  The $\schema{T}$-instance $\Pi_G(J)\colon\scol{T}\to\Set$ is given by an ordinary right Kan
  extension, as expressed by formula (\ref{formulaPi}). Its $\Type$-algebra coincides with that of
  $\inst{J}$, namely it is the presented algebra $\FrAlg{x:\Int}/{\sim}$. Because $\map{G}$ is of a
  particularly simple form, the only thing that remains to compute is $\Pi_G(J)(\EntOb{QR})$, which
  is the following subset of $J(\EntOb{Emp})\times J(\EntOb{Dept})$:
  \begin{center}
    \small
    \begin{tabular}{@{}>{\itshape}c|>{\itshape}c>{\itshape}c@{}}
      \toprule
      \upshape\EntOb{QR} & \upshape\Mor{f} & \upshape\Mor{g} \\
      \midrule
      qr1  & e2 & d1 \\
      qr2  & e6 & d1 \\
      qr3  & e6 & d2  \\
      \bottomrule
    \end{tabular}
  \end{center}

  These three elements of the product are the ones that satisfy the supplementary equations of the
  presentation of $\schema{T}$ (i.e.\ \Mor{f.sal}$\leq$\Mor{g.sec.sal}=$\top$ and
  \Mor{f.wrk.name}=$\mathrm{Admin}$).  Its columns $\Pi_G(J)(\Mor{f})$ and $\Pi_G(J)(\Mor{g})$
  represent the respective projections to $\Pi_G(J)(\EntOb{Emp})=J(\EntOb{Emp})$ and
  $\Pi_G(J)(\EntOb{Dept})=J(\EntOb{Dept})$. As usual, the names $\mathit{qr1}$, $\mathit{qr2}$, and
  $\mathit{qr3}$ are not canonical; perhaps more canonical names would be $(\mathit{e2,d1})$,
  $(\mathit{e6,d1})$, and $(\mathit{e6,d2})$.
\end{example}

\begin{remark}
  One may notice that there is an isomorphism between the set $\Pi_G(J)(\EntOb{QR})$ from
  \cref{Piexample} and the set $\sinst{S}(\inst{I},\inst{J})$ from \cref{frozeninstancetransform}.
  The reason is that there is in fact an isomorphism of $\inst{S}$-instances
  $\scol{T}(\EntOb{QR},\mcol{G}-)\cong I$, as is most evident by observing the similarity between
  the defining equations \cref{eqs_for_S,eqs_for_I}.
\end{remark}

\begin{example}
    \label{Deltaexample}
  Recall the schema mapping $\map{F}\colon\schema{R}\to\schema{T}$ described in \cref{sch_map_F}.
  Here we will discuss the pullback $\Delta_F(\inst{K})$, where $\inst{K}:=\Pi_G(\inst{J})$ is
  computed in \cref{Piexample}. Briefly, the table presentation of $\inst{K}$ consists of the tables
  $\EntOb{Emp},\EntOb{Dept}$ as in \cref{Sinstance} and $\EntOb{QR}$ as in \ref{Piexample}, and its
  $\Type$-algebra is $\FrAlg{x:\Int}/{\sim}$.

  The $\schema{R}$-instance $\Delta_F(K)\colon\scol{R}\to\Set$ is obtained by pre-composing with
  $\mcol{F}$, as in \cref{Deltadef}. It has the same $\Type$-algebra (\cref{PiDeltapreservetypes}),
  and its one entity table is
  \begin{center}
    \small
    \begin{tabular}{@{}>{\itshape}c|c>{}c>{}cr@{}}
      \toprule
      \upshape\EntOb{A} & \Mor{emp\_last} & \Mor{dept\_name} & \Mor{diff} \\
      \midrule
      qr1 & Noether & HR    & $100$  \\
      qr2 & Euclid  & HR    & $150$  \\
      qr3 &  Euclid  & Admin & $0$   \\
      \bottomrule
  \end{tabular}
  \end{center}
\end{example}

We conclude this section with some special cases for which $\Sigma$ is nicely behaved.

\subsection{A pointwise formula for \texorpdfstring{$\Sigma$}{Sigma}}
    \label{sec:pointwise_sigma}

  Given an arbitrary mapping $\map{F}\colon\schema{S}\to\schema{T}$ and $\schema{S}$-instance
  $\inst{I}$, the formula for the functor $\Sigma_F(I)\colon\scol{T}\to\Set$ cannot be given
  pointwise on objects $t\in\scol{T}$. However, there is a special kind of schema mapping $F$ for
  which we can write a pointwise formula for $\Sigma_F(\inst{I})$, namely those which induce a
  \emph{discrete opfibration} on collages $\scol{F}\colon\scol{S}\to\scol{T}$. This occurs if and
  only if $\scol{F}$ arises via the Grothendieck construction applied to a functor $\partial
  F\colon\scol{T}\to\Set$, for which the composite $\partial
  F\circ\inctypes\colon\Type\to\scol{T}\to\Set$ is terminal. Note that in this case, we have a
  bijection
  \[
    \Ob\scol{S}\cong
      \left\{\,(t,p) \mathrel{}\middle\vert\mathrel{} t\in\scol{T},p\in\partial F(t)\,\right\}\cong
      \coprod_{t\in\scol{T}}\partial F(t).
  \]
  One can show using ends, the adjunction $\Sigma\dashv\Delta$, and the fact that
  $\sinst{S}\subseteq[\scol{S},\Set]$ is fully faithful, that $\Sigma_F$ is then given by the
  following pointwise formula:
  \[
    \Sigma_F(I)(t) = \smashoperator{\coprod_{p\in\partial F(t)}} I(t,p).
  \]
  In particular, $\Sigma_F$ preserves $\Type$-algebras in this case, i.e.\ $\Sigma_F(I)
  (\tau)=I(\tau)$ for any $\tau\in\Type$.

  It is easy to show that if $\scol{F}$ is a discrete opfibration, then $\matts{F}$ is cartesian, so the preservation of $\Type$-algebras can also be seen as a special case of the
  following result.

\begin{proposition}
    \label{prop:sigma_preserves_when}
  The left pushforward $\Sigma_F$ along a schema mapping $\map{F}=(\mnodes{F},\matts{F})\colon
  \schema{S}\to\schema{T}$ preserves type-algebras if and only if $\matts{F}$ is cartesian.
\end{proposition}
\begin{proof}[sketch]
  Consider the commutative square in $\Schema$ shown here:
  \[
  \begin{tikzcd}[column sep=0pt]
    &\emptyset\ar[dl,"!_R"']\ar[dr,"!_S"]\\
    \schema{R}\ar[rr,"\map{F}"']&&\schema{S}
  \end{tikzcd}
  \]
  By \cref{ex:underlying_typealg} it suffices to show that $\matts{F}$ is cartesian if and only if
  the restriction of the unit map $\Delta_{!_R}\eta\colon \Delta_{!_R} \to
  \Delta_{!_R}\Delta_F\Sigma_F = \Delta_{!_S}\Sigma_F$ coming from $\Sigma_F\dashv\Delta_F$ is an
  isomorphism. Both sides preserve colimits, so since $\yoneda$ is dense, $\Delta_{!_R}\eta$ is an
  isomorphism if and only if the components $\Delta_{!_R}(\eta_{\yoneda(r)})$ are isomorphisms for
  any $r\in\scol{R}$. For $\tau\in\Type$, $\eta_{\yoneda(\tau)}$ is always an isomorphism. For
  $r\in\scol{R}$, we have $\Delta_{!_R}\big(\yoneda(r)\big) = \satts{R}(r)$ and $\Delta_{!_S}
  \Sigma_F\big(\yoneda(r)\big) = \satts{S}\big(\mnodes{F}(r)\big)$ by \cref{eqn:sigma_extension}. It
  is not difficult to verify that $\Delta_{!_R}(\eta_{\yoneda(r)}) \colon
  \Delta_{!_R}\big(\yoneda(r)\big) \to \Delta_{!_S} \Sigma_F\big(\yoneda(r)\big)$ and the component
  $\satts{R}(r) \to \satts{S}\big(\mnodes{F}(r)\big)$ of $\matts{F}$ at $r$ agree, completing the proof.
\end{proof}

\section{The double category \texorpdfstring{$\Data$}{S}}\label{ThedoublecategoryData}

In this section, we will introduce the notion of a \emph{bimodule} between two schemas. We will see
that bimodules generalize instances on a schema, as well as \emph{queries}, which are the subject of
\cref{Queries}. We will show that schemas, schema mappings, and bimodules together form an
equipment, which we denote $\Data$. For database-style examples of material from this section, see
\cref{Queries}.

\subsection{Relevant terminology and notation}

Recall that companions and conjoints in $\dProf$ are given by representable profunctors,
as explained in \cref{ex:prof_double_cat}. Also recall from \cref{def:profunctor_products}
that a profunctor $M$ whose codomain is an algebraic
theory $\cat{T}$ is called algebraic if it is product-preserving on the right; it is denoted
$M\colon\cat{C}\tickxar\cat{T}$. If $\schema{S}$ is a schema, then the functor
$\inctypes\colon\Type\to\scol{S}$ denotes the inclusion of $\Type$ into the collage (\cref{collageschema}).

\subsection{Bimodules between schemas}

Bimodules admit several equivalent definitions, and it is convenient to be able to switch between
these definitions as best suits the task at hand. We will begin with the one which we use most often.

\begin{definition}
    \label{def:bimodule}
  Let $\schema{R}$ and $\schema{S}$ be database schemas. A \emph{bimodule} $\bimod{M} \colon
  \schema{R} \tickar \schema{S}$ is a functor
  $M\colon \op{\scol{R}} \to \sinst{S}$ such that the
  following diagram commutes:
  \begin{equation}\label{eqn:bimodule_form2}
    \begin{tikzcd}
      \op{\Type} \ar[r,"\op{\inctypes}"] \ar[d,"\op{\inctypes}"']
      & \op{\scol{S}} \ar[d,"\yoneda"] \\
      \op{\scol{R}} \ar[r,"M"']
      & \sinst{S}
    \end{tikzcd}
  \end{equation}
  or succinctly, $M(\tau)=\yoneda(\tau)$ for any $\tau\in\Type$.

  A morphism of $(\schema{R},\schema{S})$-bimodules $\twoCell{\phi}\colon\bimod{M}\to\bimod{N}$ is a natural
  transformation $\phi\colon M\Rightarrow N$ that restricts to the identity on $\Type$.
  We denote by $\Bimod{\schema{R}}{\schema{S}}$ the category of $(\schema{R},\schema{S})$-bimodules.
\end{definition}

\begin{remark}
    \label{rem:symmetric_version_bimodule}
  It is possible to give \cref{def:bimodule} in a more symmetric form. A bimodule
  $\bimod{M}\colon\schema{R}\tickar \schema{S}$ is equivalently a profunctor
  $\bcol{M}\colon\scol{R}\tickar\scol{S}$ between collages such that:
  \begin{itemize}[nosep]
  \item the composite profunctor $\scol{R}\xtickar{\bcol{M}}\scol{S}\xtickar{\conj{\inctypes}}\Type$
    is algebraic, and
  \item the composite profunctor $\Type\xtickar{\comp{\inctypes}}\scol{R}\xtickar{\bcol{M}}\scol{S}$
    is isomorphic to the representable $\comp{\inctypes}\colon\Type\tickar\scol{S}.$
  \end{itemize}

  A morphism of $(\schema{R},\schema{S})$-bimodules $\twoCell{\phi}\colon\bimod{M}\to\bimod{N}$ is
  equivalently a profunctor transformation $\twocol{\phi}\colon\bcol{M}\Rightarrow\bcol{N}$ such
  that $\comp{\inctypes}\odot\twocol{\phi} = \id_{\comp{\inctypes}}$.

  While this formulation of bimodules may be useful for intuition, we will primarily use
  \cref{def:bimodule} in this paper.
\end{remark}

\subsection{Adjoints $\Lambda$ and $\Gamma$}
  \label{lambda_gamma}

Considering a bimodule $\bimod{M}\colon\schema{R}\tickar\schema{S}$ as a functor
$\op{\scol{R}}\to\sinst{S}$, we can apply the left Kan extension along the Yoneda embedding
$\op{\scol{R}}\to\sinst{R}$; see \cref{representableinstance}. The result is denoted
$\Lambda_M\coloneqq\Lan_Y(M)$,
\begin{equation}\label{lambda_yoneda}
\begin{tikzcd}
  \op{\scol{R}}\ar[r,"M"]\ar[d,"\yoneda"']&\sinst{S}\\
  \sinst{R}\ar[ru,"\Lambda_M"']
\end{tikzcd}
\end{equation}
Since $\sinst{S}$ is cocomplete (\cref{corr:SInst_cocomplete}), we can express this using the Kan
extension formula (cf.\ \cref{eqn:LambdaNonPointwise})
\begin{equation}
    \label{lambdabimods}
  \begin{aligned}
    \Lambda_M(I) &= \int^{r\in\scol{R}} \sinst{R}\bigl(\yoneda(r),I\bigr)\cdot M(r) \\
                 &\iso \int^{r\in\scol{R}}I(r)\cdot M(r)
  \end{aligned}
\end{equation}
where $\cdot$ is the $\Set$-theoretic copower on $\sinst{S}$. Because the Yoneda embedding is fully
faithful, it follows that this Kan extension really is an extension,
i.e.\ \cref{lambda_yoneda} commutes. It also follows that $\Lambda_M$ ``preserves types,'' that is,
that the following diagram commutes:
\begin{equation}\label{eq:Lambda_preserves_types}
  \begin{tikzcd}[column sep=0em]
    & \op{\Type} \ar[dl,"\yoneda"'] \ar[dr,"\yoneda"] & \\
    \sinst{R} \ar[rr,"\Lambda_M"'] && \sinst{S}
  \end{tikzcd}
\end{equation}

A bimodule $\bimod{M}\colon\schema{R}\tickar\schema{S}$ also determines a functor in the other direction,
\begin{equation}\label{defGamma}
\begin{tikzcd}[row sep=.02in]
  \Gamma_M\colon\sinst{S}\ar[r] & \sinst{R} \\
  \qquad\quad J\ar[r,mapsto] & \sinst{S}(M(\textrm{--}),J).
\end{tikzcd}
\end{equation}
The condition \eqref{eqn:bimodule_form2} on $M$ implies that for any object $\tau\in\Type$,
\begin{equation}\label{eq:Gamma_preserves_types}
  (\Gamma_M J)(\tau)=\sinst{S}(M(\tau),J)=\sinst{S}(\yoneda(\tau),J)=J(\tau)
\end{equation}
from which it easily follows that $\Gamma_M(J)$ preserves products of types, hence defines an object
in $\sinst{R}$. We thus say that $\Gamma$ ``preserves
type-algebras'', in the sense that the following diagram commutes:
\[
  \begin{tikzcd}[column sep=0em]
    \sinst{S} \ar[rr,"\Gamma_M"] \ar[dr,"U"']
    && \sinst{R} \ar[dl,"U"] \\
    & \TypeAlg
  \end{tikzcd}
\]

\begin{proposition}
    \label{LambdaGammaadj}
  For any bimodule $\bimod{M}\colon\schema{R}\tickar\schema{S}$, the functor $\Lambda_M$ is left adjoint to
  $\Gamma_M$.
\end{proposition}
\begin{proof}
  This is simply a calculation:
  \begin{align*}
    \sinst{S}(\Lambda_MI,J)
    &= \sinst{S}\left( \int^{r\in\scol{R}}I(r)\cdot M(r), J \right) \\
    &\iso \int_{r\in\scol{R}} \sinst{S}\big(I(r)\cdot M(r), J\big) \\
    &\iso \int_{r\in\scol{R}} \Set\bigl(I(r), \sinst{S}(M(r),J) \bigr) \\
    &= \int_{r\in\scol{R}} \Set\bigl(I(r), (\Gamma_M J)(r) \bigr)
    \iso \sinst{R}(I,\Gamma_MJ).
  \end{align*}
  The first isomorphism follows because homs take colimits in their first variable to limits, while
  the second is the definition of copower.
\end{proof}

\begin{lemma}
    \label{lem:lambda_props}
  We collect here several easy but useful properties of $\Lambda$:
  \begin{enumerate}[nosep]
    \item For any schema $\schema{S}$, there is an isomorphism of functors $\Lambda_{\yoneda}\iso\id_{\sinst{S}}$.
    \item For any bimodule $M\colon\op{\scol{R}}\to\sinst{S}$ and any left adjoint
      $L\colon\sinst{S}\to\sinst{T}$, there is an isomorphism of functors $\Lambda_{L\circ M}\iso
      L\circ\Lambda_M$. In particular,
    \item For any bimodules $\bimod{M}\colon\schema{R}\tickar\schema{S}$ and
      $\bimod{N}\colon\schema{S}\tickar\schema{T}$, there is an isomorphism of functors
      $\Lambda_{\Lambda_N\circ M}\iso\Lambda_N\circ\Lambda_M$.
  \end{enumerate}
\end{lemma}
\begin{proof}
  Property 1 is simply the fact that $\yoneda$ is dense (see \cref{representableinstance}),
  while property 2 is the fact that left adjoints preserve colimits, hence preserve pointwise
  left Kan extensions. Finally, property 3 follows from property 2 using
  \cref{LambdaGammaadj}.
\end{proof}

\begin{remark}
  The $\Lambda_M\dashv\Gamma_M$ adjunction is an instance of the general geometric realization/nerve
  adjunction $\Set^{\op{\cat{S}}}\leftrightarrows\cat{C}$ induced by a functor
  $F\colon\cat{S}\to\cat{C}$ into a cocomplete category $\cat{C}$ (see e.g.\ \cite[pp.\
  244--245]{Leinster:2004a} or \cite{Nlab:nerve-and-realization}). In this case, $F$ is the functor
  $M\colon\op{\scol{R}}\to\sinst{S}$. The conditions in \cref{def:bimodule} guarantee that the nerve functor lands
  in the full subcategory $\sinst{R}\subseteq\Fun{\scol{R},\Set}$.
\end{remark}

\subsection{Equivalent definitions of bimodules}
  \label{LAdjRAdj}

In \cref{thm:bimod_equivalence} we give five equivalent definitions of bimodules, and we will give a
few others throughout the section, e.g.\ in
\cref{componentwisebimodule,prop:semidecomposed_bimodule,cor:another_equivalence}. The ones we
discuss here are aligned with the analogy presented in \cref{ssec:profunctor_matrix}, by which
profunctors between categories and linear transformations between vector spaces can be related. The
only complication here is that all of our structures must deal carefully with the algebraic theory
$\Type$, as we now make explicit.

Consider the coslice 2-category $\op{\Type}/\CCat$. An object is a pair $(\cat{C},F)$, where
$\cat{C}\in\CCat$ is a category and $F\colon\op{\Type}\to\cat{C}$ is a functor; a morphism
$(\cat{C},F)\to(\cat{D},G)$ is a functor $H\colon\cat{C}\to\cat{D}$ such that $H\circ F=G$; and a
2-cell $H\to H'$, where $H,H'\colon(\cat{C},F)\to(\cat{D},G)$, is a natural transformation
$\alpha\colon H\Rightarrow H'$ such that $\alpha F=\id_G$.

For any schema $\schema{S}$, both $\sinst{S}$ and $\op{\scol{S}}$ can be considered objects in
$\op{\Type}/\CCat$ (via $\yoneda$ and $\op{\inctypes}$). Similarly, $\sinst{S}$ can be considered an
object in the slice 2-category $\CCat/\TypeAlg$, where the functor $\sinst{S}\to\TypeAlg$ simply
sends an instance $\scol{S}\to\Set$ to its restriction along the inclusion
$\inctypes\colon\Type\to\scol{S}$.

\begin{theorem}
    \label{thm:bimod_equivalence}
  Let $\schema{R}$ and $\schema{S}$ be schemas. The following are equivalent:
  \begin{enumerate}
  \item The category $\Bimod{\schema{R}}{\schema{S}}$ of bimodules $\schema{R}\tickar\schema{S}$.
    \labelold[item]{item:bimod}
  \item The category $(\op{\Type}/\CCat)(\op{\scol{R}},\sinst{S})$.
    \labelold[item]{item:bimod_curried}
  \item The category of profunctors $\scol{R}\tickar\scol{S}$ satisfying the conditions of
  \cref{rem:symmetric_version_bimodule}.
  	\labelold[item]{item:remark_version}
  \item The category $\LAdj_{\Type}(\sinst{R},\sinst{S})$, which we define to be the full
    subcategory of $(\op{\Type}/\CCat)(\sinst{R},\sinst{S})$ spanned by left adjoint functors.
    \labelold[item]{item:LAdj}
  \item The category $\op{\RAdj_{\Type}(\sinst{S},\sinst{R})}$, whose opposite
  is defined to be the full subcategory of
    $(\CCat/\TypeAlg)(\sinst{S},\sinst{R})$
    spanned by right adjoint functors.
    \labelold[item]{item:RAdj}
  \end{enumerate}
\end{theorem}
\begin{proof}
  \Cref{item:bimod,item:bimod_curried} are equivalent by \cref{def:bimodule}, and it is easy to
  check the equivalence between \cref{item:bimod,item:remark_version}.

  For the equivalence of \cref{item:bimod_curried,item:LAdj}, consider the functor
  \[
    (\textrm{--}\circ\yoneda) \colon
      \LAdj_{\Type}(\sinst{R},\sinst{S}) \to (\op{\Type}/\Cat)(\op{\scol{R}},\sinst{S}).
  \]
  Its inverse is $\Lambda_{-}$, the left Kan extension along
  $\yoneda\colon\op{\scol{R}}\to\sinst{R}$, which lands in $\LAdj_{\Type}(\sinst{R},\sinst{S})$ by
  \cref{eq:Lambda_preserves_types,LambdaGammaadj}. To see that these are inverses, note that by
  commutative \cref{lambda_yoneda} we have $\Lambda_M\big(\yoneda(r)\big) =  M(r)$. For the other
  direction, we have by \cref{lem:lambda_props}
  \begin{equation*}
    \Lambda_{L\circ\yoneda} \iso L\circ\Lambda_{\yoneda}
      \iso L\circ\id_{\sinst{S}} = L.
  \end{equation*}

  Finally, we show that \cref{item:LAdj,item:RAdj} are equivalent. The equivalence
  $\LAdj(\sinst{R},\sinst{S})\equiv \op{\RAdj(\sinst{S},\sinst{R})}$ is standard, so we only need to
  show that this equivalence respects the restrictions concerning $\Type$. In one direction, if
  $L\colon \sinst{R}\to\sinst{S}$ is a left adjoint satisfying \cref{eq:Lambda_preserves_types},
  then we check that the right adjoint $G$ of $L$ satisfies \cref{eq:Gamma_preserves_types}:
  \begin{align*}
    (GJ)(\tau) &\iso \sinst{R}\big(\yoneda(\tau),GJ\big) \\
               &\iso \sinst{S}\Big(L\big(\yoneda(\tau)\big),J\Big) \\
               &\iso \sinst{S}\big(\yoneda(\tau),J\big) \\
               &\iso J(\tau). \\
    \intertext{Conversely, if $G\colon\sinst{S}\to\sinst{R}$ is a right adjoint satisfying
    \cref{eq:Gamma_preserves_types}, then}
    \mathllap{\sinst{S}\Big(L\big(\yoneda(\tau)\big),J\Big)} 
               &\iso \sinst{R}\big(\yoneda(\tau),RJ\big) \\
               &\iso (GJ)(\tau) \\
               &\iso J(\tau) \\
               &\iso \sinst{S}\big(\yoneda(\tau),J\big),
  \end{align*}
  hence by the Yoneda lemma, $L\big(\yoneda(\tau)\big)\iso\yoneda(\tau)$, for any $\tau\in\Type$.
\end{proof}

\begin{proposition}
    \label{prop:bimod_colimits}
  For any schemas $\schema{R}$ and $\schema{S}$, the category $\Bimod{\schema{R}}{\schema{S}}$ has
  finite colimits.
\end{proposition}
\begin{proof}
  The initial object of $\Bimod{\schema{R}}{\schema{S}}$ is given by the left Kan extension of the
  Yoneda embedding $\op{\Type}\to\sinst{S}$ along the collage inclusion $\op{\Type}\to
  \op{\scol{R}}$. Concretely, the initial bimodule $\bimod{0}$ can be described by cases:
  \[
    0(r,s) =
    \begin{cases}
      \scol{R}(r,s) & \text{if $s$ is a type} \\
      \emptyset & \text{otherwise.}
    \end{cases}
  \]
  To complete the proof, we need to show that $\Bimod{\schema{R}}{\schema{S}}$ has pushouts. By
  \cref{thm:bimod_equivalence}, $\Bimod{\schema{R}}{\schema{S}} \equiv
  (\op{\Type}/\CCat)(\op{\scol{R}},\sinst{S})$, and by \cref{corr:SInst_cocomplete}, $\sinst{S}$ is
  cocomplete. Let us fix a choice of pushouts in $\sinst{S}$, such that the chosen pushout of the
  constant span on an instance $I$ is $I$. Then it is easy to check that
  $\Bimod{\schema{R}}{\schema{S}}$ is closed under the induced chosen pointwise pushouts in
  $\CCat(\op{\scol{R}},\sinst{S})$, and that these are in fact pushouts in the subcategory
  $(\op{\Type}/\CCat)(\op{\scol{R}},\sinst{S})$.
\end{proof}

\subsection{The equipment \texorpdfstring{$\Data$}{Data}}
  \label{restrictionsofbimodules}

We are now ready to assemble schemas, schema morphisms, and bimodules into a single double category
$\Data$, which we define in \cref{def:Data}, and which we show to be an equipment in
\cref{prop:Data_is_equipment}. In order to define the double category structure, we will need the
easy notion of restriction of bimodules along schema morphisms.

Suppose we have a bimodule $\bimod{N}\colon\schema{R'}\tickar\schema{S'}$, and schema mappings
$\map{F}\colon\schema{R}\to\schema{R'}$ and $\map{G}\colon\schema{S}\to\schema{S'}$. Thinking of
$\bimod{N}$ as a functor $N\colon \op{\scol{R}'}\to\sinst{S'}$ as in \cref{def:bimodule}, we can
form the bottom composite
\[
\begin{tikzcd}
  \op{\scol{R}} \ar[d,"\op{\mcol{F}}"'] \ar[r,dashed,"\subsub{N}{F}{G}"]
  & \sinst{S} \\
  \op{(\scol{R}')} \ar[r,"N"']
  & \sinst{S'} \ar[u,"\Delta_G"']
\end{tikzcd}
\]
and define a bimodule $\bimod{\subsub{N}{F}{G}}\colon\schema{R}\tickar\schema{S}$ so that the square
commutes. This construction defines a functor
$\Bimod{F}{G}\colon\Bimod{\schema{R'}}{\schema{S'}}\to\Bimod{\schema{R}}{\schema{S}}$. By computing
the composite $\Delta_F\circ N\circ\op{\mcol{F}}$ on objects, it easily follows that
$\big(\subsub{N}{F}{G}(r)\big)(s)=\bcol{N}(\mcol{F}r,\mcol{G}s)$ for any $r\in\scol{R}$ and $s\in\scol{S}$,
where $\bimod{N}$ is viewed as $\bcol{N}\colon\scol{R'}\tickar\scol{S'}$. This is
relevant to \cref{compconjData}.

\begin{definition}
    \label{def:Data}
  We define the double category $\Data$ as follows: the objects of $\Data$ are schemas, the vertical
  morphisms are schema mappings, and the horizontal morphisms are bimodules. We define a 2-cell of
  the form
  \begin{equation}
      \label{eqn:2-cell}
    \begin{tikzcd}
      \schema{R} \ar[r,tick,"\bimod{M}" domA] \ar[d,"\map{F}"']
      & \schema{S} \ar[d,"\map{G}"] \\
      \schema{R}' \ar[r,tick,"\bimod{N}"' codA]
      & \schema{S}'
      \twocellA{\twoCell{\theta}}
    \end{tikzcd}
  \end{equation}
  to be a natural transformation $\theta\colon M\to\Delta_G\circ N\circ\op{\mcol{F}}$:
  \begin{equation*}
    \begin{tikzcd}
      \op{\scol{R}} \ar[r,"M" domA] \ar[d,"\op{\mcol{F}}"']
      & \sinst{S} \\
      \op{\scol{R'}} \ar[r,"N"' codA]
      & \sinst{S'}, \ar[u,"\Delta_G"']
      \twocellA{\theta}
    \end{tikzcd}
  \end{equation*}
  i.e.\ a morphism $\twoCell{\theta}\in\Bimod{\schema{R}}{\schema{S}}(\bimod{M},\bimod{\subsub{N}{F}{G}})$.
  Equivalently, it is a 2-cell $\twocol{\theta}\colon\bcol{M}\Rightarrow\bcol{N}$ in $\dProf$ with frames
  $\lframe(\twocol{\theta})=\mcol{F}$ and
  $\rframe(\twocol{\theta})=\mcol{G}$, and which has identity components on $r\in\Type$.

  Given bimodules $\bimod{M}\colon\schema{R}\tickar\schema{S}$ and
  $\bimod{N}\colon\schema{S}\tickar\schema{T}$, we define their composite $\bimod{M}\odot\bimod{N}$
  by
  \begin{equation}
      \label{bimodulecomposition}
    M\odot N\coloneqq \Lambda_N\circ M\colon\op{\scol{R}}\to\sinst{S}\to\sinst{T}.
  \end{equation}
  where $\Lambda_N$ is as defined in \cref{lambdabimods}.
  The unit bimodule $\unit_{\schema{R}}\colon\schema{R}\tickar\schema{R}$ for any schema
  $\schema{R}$ is given by the Yoneda embedding $\yoneda\colon\op{\scol{R}}\to\sinst{R}$, since
  $\Lambda_M\circ\yoneda\cong M$ by \cref{lambda_yoneda}. It corresponds to the unit in $\dProf$,
  $\twocol{\unit_{\schema{R}}}\coloneqq\unit_{\scol{R}}:\scol{R}\tickar\scol{R}$.

  The horizontal composition of 2-cells
  \begin{equation}
      \label{2cellcomp}
    \begin{tikzcd}
      \schema{R} \ar[r,tick,"\bimod{M}" domA] \ar[d,"\map{F}"']
      & \schema{S} \ar[r,tick,"\bimod{N}" domB] \ar[d,"\map{G}"]
      & \schema{T} \ar[d,"\map{H}"] \\
      \schema{R}' \ar[r,tick,"\bimod{M}'"' codA]
      & \schema{S}' \ar[r,tick,"\bimod{N}'"' codB]
      & \schema{T}'
      \twocellA{\twoCell{\theta}}
      \twocellB{\twoCell{\phi}}
    \end{tikzcd}
  \end{equation}
  is defined by the composition
  \[
    \begin{tikzcd}[column sep=0]
      \op{\scol{R}} \ar[rr,"M" domA] \ar[d,"\op{\mcol{F}}"']
      &[2.6em]&[-1.2em] |[alias=domB]| \sinst{S} \ar[dr,"\Sigma_G" pos=.6] \ar[rr,"\Lambda_N" domC]
      &[-1.2em]&[2.6em] \sinst{T} \\
      \op{\scol{R'}} \ar[r,"M'"' codA]
      & \sinst{S'} \ar[ur,"\Delta_G" pos=.4] \ar[rr,equal,""codB]
      && \sinst{S'} \ar[r,"\Lambda_{N'}"' codC]
      & \sinst{T'} \ar[u,"\Delta_H"']
      \twocellA{\theta}
      \twocellB[pos=.6]{}
      \twocellC{\Lambda_{\phi}}
    \end{tikzcd}
  \]
  where the middle triangle is the counit of the $\Sigma_G\dashv\Delta_G$ adjunction. Vertical
  composition of 2-cells works in the evident way.
\end{definition}

The data above satisfy the axioms of a double category as in \cref{def:double_cat}, with vertical
category $\Data_0=\Schema$ and horizontal
$\HHor(\Data)(\schema{R},\schema{S})=\Bimod{\schema{R}}{\schema{S}}$.

\begin{proposition}
    \label{prop:Data_is_equipment}
  The double category $\Data$ is an equipment.
\end{proposition}
\begin{proof}
  It is clear from the definition of 2-cells in $\Data$ that given a niche
\begin{equation*}
    \begin{tikzcd}
      \schema{R} \ar[d,"\map{F}"'] & \schema{S} \ar[d,"\map{G}"] \\
      \schema{R'} \ar[r,tick,"\bimod{N}"'] & \schema{S'}
    \end{tikzcd}
\end{equation*}
  there is a cartesian filler with the bimodule $\bimod{\subsub{N}{F}{G}}$ from
  \cref{restrictionsofbimodules} on top.
\end{proof}

\begin{remark}\label{compconjData}
We deduce that the companion and conjoint of a schema mapping $\map{F}\colon\schema{R}\to\schema{S}$
are the bimodules given by the following formulas:
\begin{equation}\label{eqn:compconj}
\begin{alignedat}{2}
\comp{F}&=\yoneda\circ\op{\scol{F}}\;&&:\;\op{\scol{R}}\to\op{\scol{S}}\to\sinst{S}\\
\conj{F}&=\Delta_F\circ \yoneda&&:\;\op{\scol{S}}\to\sinst{S}\to\sinst{R}
\end{alignedat}
\end{equation}
These bimodules turn out to be equivalent, via \cref{thm:bimod_equivalence}, to the companion and
conjoint of the induced $\mcol{F}\colon\scol{R}\to\scol{S}$ in the equipment $\dProf$ \cref{repprofs}.
Moreover, due to \cref{rmk:SInst_colims}, $\comp{F}\odot N\odot\conj{G}$ in $\Data$
coincides with $\hat{\tilde{F}}\odot\bcol{N}\odot\check{\tilde{G}}$ in $\dProf$,
even though the horizontal compositions differ. This will be put into a
larger context in \cref{rem:collage_double_functor}.
\end{remark}

Recall from \cref{defGamma} the definition of $\Gamma$, which is right adjoint to $\Lambda$ by \cref{LambdaGammaadj}.

\begin{proposition}
    \label{Dataisclosed}
  The equipment $\Data$ is right closed, with $\bimod{N}\rhd\bimod{P}\coloneqq \Gamma_N\circ P$.
\end{proposition}
\begin{proof}
  Let $\bimod{M}\colon\schema{R}\tickar\schema{S}$, $\bimod{N}\colon\schema{S}\tickar\schema{T}$ and
  $\bimod{P}\colon\schema{R}\tickar\schema{T}$. By \cref{closedequipment}, it is enough to establish
  a natural bijection $\Bimod{\schema{R}}{\schema{T}}(\bimod{M}\odot\bimod{N},\bimod{P})\iso
  \Bimod{\schema{R}}{\schema{S}}(\bimod{M},\bimod{N}\rhd\bimod{P})$. This follows directly from the
  $\Lambda_N\dashv\Gamma_N$ adjunction:
  \begin{align*}
    \Bimod{\schema{R}}{\schema{T}}(\bimod{M}\odot\bimod{N},\bimod{P})
    &= \sinst{T}^{\op{\scol{R}}}(\Lambda_N\circ M,P) \\
    &\iso \sinst{S}^{\op{\scol{R}}}(M,\Gamma_N\circ P) \\
    &= \Bimod{\schema{R}}{\schema{S}}(\bimod{M},\bimod{N}\rhd\bimod{P}),
  \end{align*}
  completing the proof.
\end{proof}

In the following, $\LAdj_{\Type}\subseteq\op{\Type}/\CCat$ and
$\RAdj_{\Type}\subseteq\CCat/\TypeAlg$
are the obvious sub-2-categories of the (co)slices described in \cref{LAdjRAdj}.

\begin{proposition}
    \label{lambdaofcomposites}
  There is a commutative diagram of pseudofunctors and bicategories, each of which is a local
  equivalence:
  \[
    \begin{tikzcd}[column sep=large, row sep=0pt]
      & \LAdj_{\Type} \ar[dd,"\simeq"] \\
      \HHor(\Data) \ar[ru,"\Lambda_-"] \ar[rd,"\Gamma_-"'] & \\
      & \op{\RAdj_{\Type}}
    \end{tikzcd}
  \]
\end{proposition}
\begin{proof}
  On objects, $\Lambda_{\textrm{--}}$ maps a schema
  $\schema{S}$ to the functor $\yoneda\colon\op{\Type}\to\sinst{S}$; on bimodules and 2-cells,
  it is was described in \cref{thm:bimod_equivalence}.
  Then for any bimodules
  $\bimod{M}\colon\schema{R}\tickar\schema{S}$ and $\bimod{N}\colon\schema{S}\tickar\schema{T}$, we have
  $\Lambda_{M\odot N} \coloneq \Lambda_{\Lambda_N\circ M} \iso \Lambda_N\circ\Lambda_M$ and
  $\Lambda_{\unit_{\schema{S}}} \coloneq \Lambda_{\yoneda} \iso \id_{\sinst{S}}$ by
  \cref{lem:lambda_props}.

  By checking that the coherence axioms are satisfied,
  this establishes that $\Lambda_-$ is a pseudofunctor. The result follows easily from there.
\end{proof}

The following lemma establishes a certain relationship
between $\Lambda,\Gamma$ and the data migration functors of \cref{ssec:DataMigrationFunctors}.
Recall that every schema mapping $\map{F}\colon\schema{S}\to\schema{T}$ induces a triple adjunction
as on the left below, and that every bimodule $\bimod{M}\colon\schema{S}\tickar\schema{T}$ induces
an adjunction as on the right:
\begin{equation*}
  \begin{tikzcd}
    \sinst{T}\ar[rr,"{\Delta_F}"description] &&
    \sinst{S}\ar[ll,bend left=21,"{\Pi_F}","{\bot}"']
    \ar[ll,bend right=21,"{\Sigma_F}"',"{\bot}"]
  \end{tikzcd}
  \qquad
  \begin{tikzcd}
    \sinst{T} \ar[r,shift right=1.5,"\Gamma_M"'] \ar[r,phantom,"\scriptstyle\bot"]
    & \sinst{S} \ar[l,shift right=1.5,"\Lambda_M"']
  \end{tikzcd}
\end{equation*}
Recall also that every schema mapping $\map{F}\colon\schema{S}\to\schema{T}$ induces a pair of
bimodules $\comp{F}\colon\schema{S}\tickar\schema{T}$ and
$\conj{F}\colon\schema{T}\tickar\schema{S}$.

\begin{lemma}\label{lambdaofcompanionconjoint}
For any schema mapping $\map{F}$, we have the following isomorphisms and adjunctions:
\[
(\Sigma_F\cong\Lambda_{\comp{F}})
\;\dashv\;
(\Delta_F\cong\Lambda_{\conj{F}}\cong\Gamma_{\comp{F}})
\;\dashv\;
(\Pi_F\cong\Gamma_{\conj{F}})
\]
\end{lemma}
\begin{proof}
The adjunctions are given in \cref{Sigmadef,Pidef}, so we provide the isomorphisms.
The companion and conjoint of schema mappings are given in \cref{eqn:compconj}.

For $\map{F}\colon\schema{S}\to\schema{T}$ and an instance
$\inst{I}\in\sinst{S}$, we have an isomorphism
\[
\Lambda_{\comp{F}}(I) = \int^{s\in\scol{S}}I(s)\cdot\comp{F}(s)\cong
\int^{s\in\scol{S}}I(s)\cdot\scol{T}(\mcol{F}s,-)\cong\Sigma_F(I)
\]
by \cref{lambdabimods}, \cref{compconjData} and \cref{formulaSigma}. For any $s\in\scol{S}$ we also have an isomorphism
\[
\Gamma_{\conj{F}}I(s) = \sinst{S}(\conj{F}(s),I)\cong\sinst{S}(\Delta_F(\yoneda s),I)\cong\Pi_FI(s)\]
by \cref{defGamma,compconjData,Pidef}. The remaining isomorphisms (for $\Delta$) follow
by \cref{LambdaGammaadj}.
\end{proof}

\subsection{Decomposing bimodules}
  \label{decomposing_bimodules}

Let $\bimod{M}\colon\schema{R}\tickar\schema{S}$ be a bimodule. Since $\dProf$ has extensive
collages by \cref{Profextensivecollages}, the respective profunctor
$\bcol{M}\colon\scol{R}\tickar\scol{S}$ determines an $(\satts{R},\satts{S})$-simplex in the sense
of \cref{def:simplices}: four profunctors
$M_{\tn{ee}}\colon\snodes{R}\tickar\snodes{S}$, $M_{\tn{et}}\colon\snodes{R}\tickar\Type$,
$M_{\tn{te}}\colon\Type\tickar\snodes{S}$, and $M_{\tn{tt}}\colon\Type\tickar\Type$, obtained via
the restriction of $M$ along the obvious inclusions, together with four 2-cells $M_{\tn{e}*}$,
$M_{\tn{t}*}$, $M_{*\tn{e}}$, and $M_{*\tn{t}}$.

The conditions of \cref{rem:symmetric_version_bimodule} force $M_{\tn{te}}$ to be the initial
profunctor (i.e.\ the constant functor $\op{\Type}\times\snodes{S}\to\Set$ with value the empty
set), $M_{\tn{tt}}$ to be the unit profunctor (i.e.\ the hom functor $\op{\Type}\times\Type\to\Set$)
and $M_{\tn{et}}$ to be algebraic. Because $M_{\tn{te}}$ is initial, and because tensor product of
profunctors preserves colimits, the 2-cells $M_{\tn{t}*}$ and $M_{*\tn{e}}$ are unique, and hence
don't need to be specified. Thus we have proven the following proposition, in which we let
$\bnodes{M}\coloneqq M_{\tn{ee}}$, $\bterms{M}\coloneqq M_{\tn{et}}$,
$\batts{M}\coloneqq M_{\tn{e}*}$, and $\bret{M}\coloneqq M_{*\tn{t}}$.%
\footnote{
  The mnemonic for $\bret{M}$ comes from its role as "return clause" in queries; see
  \cref{sec:query_bimodules}.
}

\begin{proposition}
    \label{componentwisebimodule}
  A \emph{bimodule} $\bimod{M}\colon\schema{R}\tickar\schema{S}$ is equivalent to a tuple
  $(\bnodes{M},\bterms{M},\batts{M},\bret{M})$, where $\bnodes{M}\colon\snodes{R}\tickar\snodes{S}$
  is a profunctor, $\bterms{M}\colon\snodes{R}\tickxar\Type$ is an algebraic profunctor, and
  $\batts{M}$ and $\bret{M}$ are profunctor morphisms
  \begin{equation}\label{pic_componentwisebimodule}
    \begin{tikzcd}[row sep=2ex,baseline=(B.base)]
      & |[alias=domA]| \snodes{S} \ar[dr,bend left=25,tickx,"\satts{S}"] & \\
      |[alias=B]| \snodes{R} \ar[ur,bend left=25,tick,"\bnodes{M}"]
      \ar[rr,tickx,"\bterms{M}"'{codA,codB}]
      \ar[rr,tickx,bend right=60,looseness=.9,"\satts{R}"' domB]
      && \Type
      \twocellA{\batts{M}}
      \twocellalt[pos=.4]{B}{\Uparrow\bret{M}}
    \end{tikzcd}
  \end{equation}
\end{proposition}

\begin{example}
    \label{ex:unit_comp_conj_bimod}
  For any schema $\schema{R}$, the unit bimodule $\unit_{\schema{R}}:\schema{R}\tickar\schema{R}$ is
  given by
  \begin{equation}\label{unitdecomposition}
    \begin{tikzcd}[row sep=1.5ex,column sep=3ex,baseline=(B.base)]
    & |[alias=domA]| \snodes{R} \ar[dr,bend left=20,tickx,"\satts{R}"] & \\
    |[alias=B]| \snodes{R} \ar[ur,bend left=20,tick,"\snodes{R}"]
    \ar[rr,tickx,"\satts{R}"'{codA,codB}]
    \ar[rr,tickx,bend right=60,looseness=.9,"\satts{R}"' domB]
    && \Type
    \twocellalt[pos=.6]{A}{\id}
    \twocellalt[pos=.4]{B}{\id}
    \end{tikzcd}
  \end{equation}
  and the companion/conjoint of a schema mapping
  $\map{F}=(\snodes{F},\satts{F})\colon\schema{R}\to\schema{S}$ decompose as
  \[
    \begin{tikzcd}[row sep=2ex]
      & |[alias=domA]| \snodes{S} \ar[dr,bend left=25,tickx,"\satts{S}"] & \\
      \snodes{R} \ar[ur,bend left=25,tick,"\snodes{\comp{F}}"]
      \ar[rr,bend left=6,tickx,"\snodes{\comp{F}}\odot\satts{S}"'{codA,codB}]
      \ar[rr,tickx,bend right=60,looseness=.9,"\satts{R}"' domB]
      && \Type
      \twocellA{\id}
      \twocellalt[pos=.4]{B}{\Uparrow\comp{\satts{F}}}
    \end{tikzcd}
    \qquad
    \begin{tikzcd}[row sep=2ex]
      & |[alias=domA]| \snodes{R} \ar[dr,bend left=25,tickx,"\satts{R}"] & \\
      \snodes{S} \ar[ur,bend left=25,tick,"\snodes{\conj{F}}"]
      \ar[rr,tickx,"\satts{S}"'{codA,codB}]
      \ar[rr,tickx,bend right=60,looseness=.9,"\satts{S}"' domB]
      && \Type
      \twocellA{\conj{\satts{F}}}
      \twocellalt[pos=.4]{B}{\Uparrow\id}
    \end{tikzcd}
  \]
  where 2-cells $\comp{\satts{F}}$, $\conj{\satts{F}}$ are as in \cref{sec:adjunction_reps} for
  $\dProf$. This is `component-wise' \cref{compconjData}.
\end{example}

The equivalence $\res{F}{G}$ (\cref{res_functor}) for the extensive collages equipment $\dProf$,
which on objects resulted in \cref{componentwisebimodule}, also gives an equivalent expression of a
2-cell $\twoCell{\theta}$ in $\Data$, viewed as $M\to\widehat{\tilde{F}}\odot N\odot\widecheck{\tilde{G}}$
inside $\HHor(\dProf)(\scol{S},\scol{R})$ (see \cref{compconjData}).

\begin{proposition}
    \label{decomposed_2cell}
  A 2-cell $\twoCell{\theta}$ in $\Data$ \eqref{eqn:2-cell} is equivalently a pair of profunctor
  morphisms
  \begin{equation*}
    \begin{tikzcd}
      \snodes{R} \ar[r,tick,"\bnodes{M}" domA] \ar[d,"\mnodes{F}"']
      & \snodes{S} \ar[d,"\mnodes{G}"] \\
      \snodes{R}' \ar[r,tick,"\bnodes{M'}"' codA]
      & \snodes{S}'
      \twocellA{\twonodes{\theta}}
    \end{tikzcd}
    \qquad
    \begin{tikzcd}
      \snodes{R} \ar[r,tickx,"\bterms{M}" domA] \ar[d,"\mnodes{F}"']
      & \Type \ar[d,equal] \\
      \snodes{R}' \ar[r,tickx,"\bterms{M'}"' codA]
      & \Type
      \twocellA{\twoterms{\theta}}
    \end{tikzcd}
  \end{equation*}
  satisfying the equations
  \begin{equation}\label{2cellaxioms}
  \begin{aligned}
    \begin{tikzcd}[row sep=1cm,ampersand replacement=\&]
      \snodes{R} \ar[r,tick,"\bnodes{M}" domA] \ar[d,"\mnodes{F}"']
      \& \snodes{S} \ar[r,tickx,"\satts{S}" domB] \ar[d,"\mnodes{G}" description]
      \& \Type \ar[d,equal] \\
      \snodes{R}' \ar[r,tick,"\bnodes{M'}" codA] \ar[rr,bend right=35,tickx,"\bterms{M'}"' codC]
      \& |[alias=domC]| \snodes{S}' \ar[r,tickx,"\satts{S}'" codB]
      \& \Type
      \twocellA{\twonodes{\theta}}
      \twocellB{\matts{G}}
      \twocellC[pos=.6]{\batts{M'}}
    \end{tikzcd} \quad &= \quad
    \begin{tikzcd}[row sep={},baseline=-9pt,ampersand replacement=\&]
      \& |[alias=domA]| \snodes{S} \ar[dr,bend left=15,tickx,"\satts{S}"] \& \\[between origins,1.5em]
      \snodes{R} \ar[ur,bend left=15,tick,"\bnodes{M}"]
      \ar[rr,bend right=10,tickx,"\bterms{M}"' {codA,domB}]
      \ar[d,"\mnodes{F}"']
      \&\& \Type \ar[d,equal] \\[1.8em]
      |[alias=B]| \snodes{R}' \ar[rr,tickx,"\bterms{M'}"' codB]
      \&\& \Type
      \twocellA{\batts{M}}
      \twocellB{\twoterms{\theta}}
    \end{tikzcd}
    \\
    \begin{tikzcd}[ampersand replacement=\&]
      \snodes{R} \ar[r,bend left=45,tickx,"\satts{R}" domA] \ar[r,tickx,"\bterms{M}"'{codA,domB}]
      \ar[d,"\mnodes{F}"']
      \& \Type \ar[d,equal] \\
      \snodes{R}' \ar[r,tickx,"\bterms{M'}"' codB]
      \& \Type
      \twocellA{\bret{M}}
      \twocellB{\twoterms{\theta}}
    \end{tikzcd}
    \quad &= \quad
    \begin{tikzcd}[ampersand replacement=\&]
      \snodes{R} \ar[r,tickx,"\satts{R}" domA] \ar[d,"\mnodes{F}"']
      \& \Type \ar[d,equal] \\
      \snodes{R}' \ar[r,tickx,"\satts{R}'"{codA,domB}]
      \ar[r,bend right=45,tickx,"\bterms{M'}"' codB]
      \& \Type
      \twocellA{\matts{F}}
      \twocellB{\bret{M'}}
    \end{tikzcd}
  \end{aligned}
  \end{equation}
\end{proposition}

\begin{corollary}
    \label{componentwisecartesian}
  A 2-cell $\twoCell{\theta}$ in $\Data$ is cartesian if and only if the 2-cells $\twonodes{\theta}$
  and $\twoterms{\theta}$ from \cref{decomposed_2cell} are cartesian in $\dProf$.
\end{corollary}

Because it will be convenient later, we now present yet another equivalent representation of
bimodules, which is in some sense intermediate between \cref{def:bimodule} and the completely decomposed
representation of \cref{componentwisebimodule}. Recall from \cref{ex:underlying_typealg}
that for any $\schema{S}$-instance $\inst{I}$, the underlying $\Type$-algebra is given by
$\iterms{I}=\Delta_{!_S}(\inst{I})$, where $!_S\colon\emptyset\to\schema{S}$ is the unique map.

\begin{proposition}
    \label{prop:semidecomposed_bimodule}
  A \emph{bimodule} $\bimod{M}\colon \schema{R}\tickar\schema{S}$ is equivalent to a functor
  $\bint{M}\colon \op{\snodes{R}}\to\sinst{S}$ together with a natural transformation
  \[
    \begin{tikzcd}[row sep=0em]
      \op{\snodes{R}} \ar[rr,bend left=20,"\satts{R}" domA] \ar[dr,bend right=15,"\bint{M}"'] && \TypeAlg \\
      & |[alias=codA]| \sinst{S} \ar[ur,bend right=15,"\Delta_{!_S}"'] &
      \twocellA{\bret{M}}
    \end{tikzcd}
  \]
\end{proposition}
\begin{proof}
The functor $\op{M}\colon\scol{R}\to\op{\sinst{S}}$, opposite to the one given in \cref{def:bimodule},
can equivalently be defined \dash using the universal property \cref{univpropcollages} of collages in
$\dProf$ \dash as a functor $\op{\bint{M}}\colon\snodes{R}\to\op{\sinst{S}}$, along with a natural
transformation
   \[
    \begin{tikzcd}[column sep=.6in]
      \snodes{R} \ar[r,tick,"\satts{R}" domA] \ar[d,"\op{\bint{M}}"']
      & \Type \ar[d,"\op{\yoneda}"] \\
      \op{\sinst{S}} \ar[r,tick,"\op{\sinst{S}}"' codA] & \op{\sinst{S}}.
      \twocellA{\bret{M}}
    \end{tikzcd}
  \]
since by definition, types are mapped to representables.
This natural family of functions $\satts{R}(r,\tau)\to\sinst{S}\left(\yoneda(\tau),\bint{M}(r)\right)$
equivalently define $\bret{M}$ as a natural transformation $\satts{R}\Rightarrow \Delta_{!_S}\circ\bint{M}$
by Yoneda: $\sinst{S}\big(\yoneda(\tau),\bint{M}(r)\big)\iso\bint{M}(r)(\tau)$.
\end{proof}

\subsection{Instances in terms of bimodules}
  \label{bimodinstance}

The category of instances on a schema $\schema{S}$ can be viewed entirely in terms of bimodules.
Indeed, if $\schema{U}=(\singleton,\kappa)$ is the unit schema from \cref{ex:schema_constants}, then
we have an isomorphism of categories \[\HHor(\Data)(\schema{U},\schema{S})\iso\sinst{S}.\] This
follows by comparing their decomposed forms \dash see
\cref{decomposing_instances,componentwisebimodule,decomposed_2cell} \dash and using the fact that
$\kappa\colon\singleton\tickar\Type$ is the initial $\Type$-algebra.

It also follows that $\Lambda_N(\textrm{--})$ is simply given by bimodule composition. Indeed, by
\cref{bimodulecomposition}, for any bimodule $\bimod{N}\colon\schema{S}\tickar\schema{T}$ and
$\schema{S}$-instance $\inst{J}$, considered as a $(\schema{U},\schema{S})$-bimodule, one has
\begin{equation}
   \label{Lambda_applied_dot}
  \Lambda_N(\inst{J}) \iso \inst{J}\odot\bimod{N}.
\end{equation}
Similarly, for any $\schema{T}$-instance $I$,
\begin{equation*}
  \Gamma_N(\inst{I}) \iso \inst{J}\rhd\bimod{N}.
\end{equation*}

\subsection{Data migration functors in terms of bimodules}\label{datamigrationinData}
  We can also recover the fundamental data migration functors from the structure of $\Data$, using
  \cref{lambdaofcompanionconjoint,Lambda_applied_dot}. That is, if we consider instances as bimodules
$\inst{I}\colon\schema{U}\tickar\schema{R}$ and $\inst{J}\colon\schema{U}\tickar\schema{S}$, then
composing and exponentiating them with companions and conjoints of $\map{F}\colon\schema{R}\to\schema{S}$
is equivalent to applying $\Sigma,\Delta,\Pi$:
\[
\Sigma_F(I)\cong\inst{I}\odot\comp{\map{F}},\qquad
\Delta_F(J)\cong\inst{J}\odot\conj{\map{F}}\cong\comp{\map{F}}\rhd\inst{J},\qquad
\Pi_F(I)\cong\conj{\map{F}}\rhd\inst{I}
\]

\subsection{Collages in $\Data$}\label{CollagesinData}
  We now consider collages (see \cref{def:collage}) in the proarrow equipment $\Data$.
  Using \cref{componentwisebimodule} and the fact that $\dProf$ has extensive collages
  (\cref{Profextensivecollages}), we can fully express a collage in $\Data$ in terms of profunctor
  collages.

  Let $\bimod{M}=(\bnodes{M},\bterms{M},\batts{M},\bret{M})\colon\schema{R}\tickar\schema{S}$ be a
  bimodule as in \cref{pic_componentwisebimodule}. Its collage will be a schema $\Col{M}$,
  together with two schema mappings $\schema{R}\to\Col{M}\from\schema{S}$ and a universal 2-cell
  $\twoCell{\mu}\colon\bimod{M}\Rightarrow\unit_{\Col{M}}$ in $\Data$. We begin by describing $\Col{M}$.

\subsection{The schema of a bimodule collage}\label{bimodulecollage}
  The entity category of the collage $\Col{M}$ is the collage of the profunctor
  $\bnodes{M}\colon\snodes{R}\tickar\snodes{S}$
  \[\Colnodes{M}\coloneqq\scol{\bnodes{M}},\]
  and the observables profunctor
  $\Colatts{M}\colon\scol{\bnodes{M}}\tickxar\Type$ is the one uniquely corresponding, via
  the universal property of the lax colimit $\scol{\bnodes{M}}$ (dual of
  \cref{proarrowsinoutcollage}), to the cocone
\begin{equation*}
\begin{tikzcd}[row sep=.15in,column sep=.7in]
    \snodes{S} \ar[dr,tickx,"\satts{S}"] & \\
    & |[alias=codA]| \Type \\
    \snodes{R}\ar[uu,tick,"\bnodes{M}"domA] \ar[ur,tickx,"\bterms{M}"'] &
    \twocellA[pos=.35]{\batts{M}}
  \end{tikzcd}
\end{equation*}
In simple words, the functor $\Colatts{M}\colon\op{\scol{\bnodes{M}}}\times\Type\to\Set$
is given by $\bterms{M}$ on the $\snodes{R}$-side of $\scol{\bnodes{M}}$,
by $\satts{S}$ on the $\snodes{S}$-side of $\scol{\bnodes{M}}$, and by $\batts{M}$ on the morphisms in between.
The profunctor $\Colatts{M}$ is algebraic, because $\bterms{M}$ and $\satts{S}$ are.

\subsection{The schema mappings of a bimodule collage}\label{i_Ri_Sconstruction}
We now define the collage inclusions
$\map{i_{\schema{R}}}\colon\schema{R}\to\Col{M}\from\schema{S}\cocolon\map{i_{\schema{S}}}.$ They are
schema mappings as in \cref{sch_mapping}, thus each consists of a functor between entity
categories and a 2-cell in $\dProf$. The functors between entity categories are the collage inclusions
from $\dProf$ (see \cref{ex:prof_collages}):
\[
\mnodes{(i_\schema{R})}\coloneqq i_{\snodes{R}}\colon\snodes{R}\to\scol{\bnodes{M}}
\qquad\tn{and}\qquad
\mnodes{(i_\schema{S})}\coloneqq i_{\snodes{S}}\colon\snodes{S}\to\scol{\bnodes{M}}.
\]
The 2-cells $\matts{(i_\schema{R})}$ and $\matts{(i_\schema{S})}$ in $\dProf$ are defined respectively as follows:
\begin{equation}\label{i_R&i_S}
\begin{tikzcd}[column sep=.6in,row sep=.4in]
\snodes{R} \ar[r,bend left=40,tickx,"\satts{R}" domA]
\ar[r,tickx,"\bterms{M}"'{codA,domB}]
\ar[d,"i_{\snodes{R}}"']
& \Type \ar[d,equal] \\
\scol{\bnodes{M}} \ar[r,tickx,"\Colatts{M}"' codB]
& \Type \twocellA{\bret{M}}
\twocellB{\tn{cart}}
\end{tikzcd}\qquad\textrm{ and }\qquad
\begin{tikzcd}[column sep=.6in,row sep=.4in]
\snodes{S} \ar[r,tickx,"\satts{S}" domA] \ar[d,"i_{\snodes{S}}"']
& \Type \ar[d,equal] \\
\scol{\bnodes{M}} \ar[r,tickx,"\Colatts{M}"' codA] & \Type
\twocellA{\tn{cart}}
\end{tikzcd}
\end{equation}
The fact that the indicated 2-cells are cartesian follows by definition of $\Col{M}$ as a
lax colimit; see \cref{prop:collage_lax_limit}.

\subsection{The 2-cell of a bimodule collage}\label{mu_construction}
We now define the 2-cell $\twoCell{\mu}$ in $\Data$
\begin{equation}\label{Datamu}
\begin{tikzcd}[column sep=.6in,row sep=.4in]
\schema{R} \ar[r,tick,"\bimod{M}" domA] \ar[d,"\map{i_\schema{R}}"'] &
\schema{S} \ar[d,"\map{i_\schema{S}}"] \\
\Col{M} \ar[r,tick,"\Col{M}"' codA] & \Col{M}
\twocellA{\twoCell{\mu}}
\end{tikzcd}
\end{equation}
in its decomposed form (see \cref{decomposed_2cell}) to be the pair $\twoCell{\mu}=(\twonodes{\mu},\twoterms{\mu})$
  \begin{displaymath}
    \begin{tikzcd}[column sep=.6in,row sep=.4in]
      \snodes{R} \ar[r,tick,"\bnodes{M}"domA] \ar[d,"i_{\snodes{R}}"']
      & \snodes{S} \ar[d,"i_{\snodes{S}}"] \\
      \scol{\bnodes{M}} \ar[r,tick,"\scol{\bnodes{M}}"'codA]
      & \scol{\bnodes{M}}
      \twocellA{\twonodes{\mu}}
    \end{tikzcd}
    \qquad
    \begin{tikzcd}[column sep=.6in,row sep=.4in]
      \snodes{R} \ar[r,tickx,"\bterms{M}"domA] \ar[d,"i_{\snodes{R}}"']
      & \Type \ar[d,equal] \\
      \scol{\bnodes{M}} \ar[r,tickx,"\Colatts{M}"'codA]
      & \Type
      \twocellA{\twoterms{\mu}}
    \end{tikzcd}
  \end{displaymath}
where $\twonodes{\mu}$ is the universal 2-cell for $\scol{\bnodes{M}}$ in $\dProf$ (see
\cref{ex:prof_collages}) and $\twoterms{\mu}$ is the (cartesian) square shown to
the left in \cref{i_R&i_S}. The components $\twonodes{\mu}$ and $\twoterms{\mu}$
satisfy the equations \cref{2cellaxioms} by the unit bimodule decomposition \cref{unitdecomposition}
and the universal property of lax colimit \cref{1-dim_univ_prop}.

\begin{proposition}\label{Datanormalcollages}
The equipment $\Data$ has normal collages.
\end{proposition}
Note that $\Data$ does not have extensive collages. In particular, $i_R$ is not in general fully
faithful.
\begin{proof}
We must first verify that $\Data$ has collages, i.e.\ that the 2-cell $\twoCell{\mu}$ defined in \cref{Datamu}
has the required universal property \cref{univpropcollages}.
Suppose that $\schema{X}$ is a schema and that $\twoCell{\phi}\colon\bimod{M}\Rightarrow\unit_\schema{X}$ is a
2-cell from $\bimod{M}$ to the unit bimodule. We must show that $\twoCell{\phi}$ factors uniquely through $\twoCell{\mu}$:
\begin{displaymath}
 \begin{tikzcd}
    \schema{R} \ar[r,tick,"\bimod{M}"domA] \ar[d,"\map{F}"']
    & \schema{S}  \ar[d,"\map{G}"]  \\
    \schema{X} \ar[r,tick,"\schema{X}"'codA] & \schema{X}
    \twocellA{\twoCell{\phi}}
  \end{tikzcd}
  \quad = \quad
  \begin{tikzcd}[column sep =.45in]
\schema{R} \ar[r,tick,"\bimod{M}" domA] \ar[d,"\map{i_\schema{R}}"'] &
\schema{S} \ar[d,"\map{i_\schema{S}}"] \\
\Col{M} \ar[r,tick,"\Col{M}"'{codA,domB}] \ar[d,"\map{\bar{\phi}}"']
& \Col{M} \ar[d,"\map{\bar{\phi}}"] \\
\schema{X} \ar[r,tick,"\schema{X}"'codB] & \schema{X}
\twocellA{\twoCell{\mu}}
\twocellB[pos=.6]{\map{\bar{\phi}}}
  \end{tikzcd}
\end{displaymath}
We work with components, writing $\twoCell{\phi}=(\twonodes{\phi},\twoterms{\phi})$ as in
\cref{decomposed_2cell}. Firstly, since $\twonodes{\mu}$ is the universal 2-cell
for a collage in $\dProf$, we
have that $\twonodes{\phi}=\unit_{\twonodes{\bar{\phi}}}\circ\twonodes{\mu}$
for a unique functor
$\twonodes{\bar{\phi}}\colon\scol{\bnodes{M}}\to\snodes{X}$.
Also, $\twoterms{\phi}=\twoterms{\bar{\phi}}\circ\twoterms{\mu}$ as in
\[
\begin{tikzcd}[column sep=.55in]
	\snodes{R}\ar[r,tickx,"\bterms{M}" domA]\ar[d,"\mnodes{F}"']&\Type\ar[d,equal]
	\\
	\snodes{X}\ar[r,tickx,"\satts{X}"' codA]&\Type
	\twocellA{\twoterms{\phi}}
\end{tikzcd}
\quad = \quad
\begin{tikzcd}[column sep=.55in]
	\snodes{R}\ar[r,tickx,"\bterms{M}" domA]\ar[d,"i_{\snodes{R}}"']
	&\Type\ar[d,equal]
	\\
	\scol{\bnodes{M}}\ar[r,tickx,"\Colatts{M}"' {codA,domB}]\ar[d,"\twonodes{\bar{\phi}}"']
	&\Type\ar[d,equal]
	\\
	\snodes{X}\ar[r,tickx,"\satts{X}"' codB]
	&\Type
	\twocellA{\twoterms{\mu}}
	\twocellB{\twoterms{\bar{\phi}}}
\end{tikzcd}
\]
for a unique 2-cell $\twoterms{\bar{\phi}}$, obtained via the 2-dimensional universal
property of the lax colimit $\scol{\bnodes{M}}$ (\cref{{prop:collage_lax_limit}}).
This profunctor morphism $\twoterms{\bar{\phi}}\colon\Colatts{M}\Rightarrow
\twonodes{\comp{\bar{\phi}}}\odot\satts{X}$ is concretely defined, omitting the details,
by natural components
\begin{align*}
(\twoterms{\bar{\phi}})_{{r\tau}} &=  \Colatts{M}(r,\tau)\simrightarrow\bterms{M}(r,\tau)
\xrightarrow{(\twoterms{\phi})_{r\tau}}
\satts{X}(\mnodes{F}r,\tau)=\satts{X}(\twonodes{\bar{\phi}}r,\tau) \\
(\twoterms{\bar{\phi}})_{{s\tau}} &=  \Colatts{M}(s,\tau)\simrightarrow\satts{S}(s,\tau)
\xrightarrow{(\matts{G})_{s\tau}}\satts{X}(\mnodes{G}s,\tau)=\satts{X}(\twonodes{\bar{\phi}}s,\tau)
\end{align*}
for $r\in\snodes{R}$, $s\in\snodes{S}$ and $\tau\in\Type$.
Defining $\twoCell{\bar{\phi}}$ to be the pair $(\unit_{\twonodes{\bar{\phi}}},\twoterms{\bar{\phi}})\colon
\Col{M}\to\schema{X}$, we have $\twoCell{\phi}=\unit_{\map{\bar{\phi}}}\circ\twoCell{\mu}$ as desired.

Moreover, collages in $\Data$ are normal, as in \cref{normalcollages}, since
the 2-cell $\twoCell{\mu}$ constructed in \cref{mu_construction} is cartesian: by \cref{componentwisecartesian},
it is enough that $\twonodes{\mu}$ and $\twoterms{\mu}$ are cartesian liftings in $\dProf$.
\end{proof}

\begin{remark}
	\label{rem:collage_double_functor}
Although we will not use this fact, we note that the
collage correspondences from
\cref{collageschema,collagefunctor,rem:symmetric_version_bimodule}
provide a \emph{lax double functor}
(see e.g.~\cite{Grandis.Pare:2004a})
\[\scol{(-)}\colon\Data\to\dProf.\]
This functor is only lax, because for bimodules
$\schema{R}\xtickar{\bimod{M}}\schema{S}\xtickar{\bimod{N}}\schema{T}$
in $\Data$, the natural map $\mcol{M}\odot\mcol{N}\to\mcol{M\odot N}$
in $\dProf$ given by the unique transformation in $[\scol{T},\Set]$
$$\int^{s\in\scol{S}}\bcol{M}(r,s)\times\bcol{N}(s,-)\Rightarrow
\int^{s\in\scol{S}}M(r)(s)\cdot N(s)$$
between the pointwise colimit \cref{eqn:coendComp} and the $\sinst{T}$-one \cref{bimodulecomposition},
is not an isomorphism (\cref{rmk:SInst_colims}).
Since lax double functors between equipments automatically preserve cartesian liftings
(see~\cite[Prop.~6.4]{Shulman:2008a}),
this fact also explains \cref{compconjData}.
\end{remark}

\subsection{Bimodules in terms of data migration}
We will now see that any bimodule, considered as an adjoint functor on instance categories
via \cref{thm:bimod_equivalence}, is equivalent to a composite of data migration functors.

\begin{corollary}\label{Gamma=DeltaPie}
Let $\bimod{M}\colon\schema{R}\tickar\schema{S}$ be a bimodule, and let
${\map{i_\schema{R}}}\colon\schema{R}\to\Col{M}$ and ${\map{i_\schema{S}}}\colon\schema{S}\to\Col{M}$ be the
collage inclusions. We have isomorphisms
\begin{displaymath}
\Lambda_M\cong\Delta_{\incnodes{\schema{S}}}\circ\Sigma_{\incnodes{\schema{R}}}
\quad\tn{and}\quad
\Gamma_M\cong\Delta_{\incnodes{\schema{R}}}\circ\Pi_{\incnodes{\schema{S}}}
\end{displaymath}
\end{corollary}
\begin{proof}
Since $\Data$ has normal collages, \cref{Datamu} is a cartesian 2-cell,
hence $\bimod{M}\cong\comp{i}_{\schema{R}}\odot\unit_{\Col{M}}\odot\conj{i}_{\schema{S}}$.
Therefore, by \cref{lambdaofcompanionconjoint,lambdaofcomposites}, we have
\begin{align*}
\Lambda_M & \cong\Lambda_{\comp{i}_{\schema{R}}\odot \unit_{\Col{M}}\odot\conj{i}_{\schema{S}}}\cong
\Lambda_{\conj{i}_{\schema{S}}}\circ\Lambda_{\comp{i}_{\schema{R}}}
\cong\Delta_{\incnodes{\schema{S}}}\circ\Sigma_{\incnodes{\schema{R}}}
\end{align*}
and dually for the right adjoint $\Gamma_M$.
\end{proof}

\begin{corollary}\label{cor:another_equivalence}
Suppose that $\schema{R}$ and $\schema{S}$ be schemas. Let $\cat{L}\subseteq\LAdj_{\Type}(\sinst{R},\sinst{S})$ \emph{[resp.}
let $\cat{R}\subseteq\RAdj_{\Type}(\sinst{S},\sinst{R})$\emph{]} denote the
full subcategory spanned by functors of the form $\Delta_G\circ\Sigma_F$ \emph{[resp.} of the form
$\Pi_G\circ\Delta_F$\emph{]}. This inclusion is an equivalence of categories.
\end{corollary}
\begin{proof}
The inclusion functor is fully faithful by definition and essentially surjective by \cref{Gamma=DeltaPie,thm:bimod_equivalence}.
\end{proof}

\begin{remark}
  \Cref{cor:another_equivalence} says that every right adjoint between instance categories is
  naturally isomorphic to a right pushforward followed by a pullback, $\Delta\circ\Pi$. While we do
  not discuss the details here, there is a similar characterization of \emph{parametric right
  adjoints} (a.k.a.\ \emph{local right adjoints}) between instance categories.

  Recall that a functor $F\colon\cat{C}\to\cat{D}$ is a parametric right adjoint if, for each object
  $c\in\cat{C}$, the slice $F/c\colon\cat{C}/c\to\cat{D}/(Fc)$ is a right adjoint, see e.g.\
  \cite{Weber:2007a}. In our setting, one can show that every parametric right adjoint between
  instance categories is naturally isomorphic to a functor of the form
  $\Sigma_{\tn{dopf}}\circ\Delta\circ\Pi$, where the left pushfoward is along a discrete
  op-fibration, as discussed in \cref{sec:pointwise_sigma}. This generalizes the analogous fact for
  parametric right adjoints between presheaf categories, as shown in \cite[Remark
  2.12]{Weber:2007a}.
\end{remark}

\subsection{Bimodules presentation}
We conclude \cref{ThedoublecategoryData} by discussing presentations of bimodules,
which work very similarly to presentations of profunctors
(\cref{def:profunctor_presentation}). Recall also the definition of
schema presentations, \cref{def:schema_presentation}. Suppose that
  $\Type\iso\Cxt{\Sigma}/E_{\Sigma}$ has algebraic signature $\Sigma$;%
\footnote{Signatures $\Sigma$ should not be confused with data migration functors $\Sigma_-$.
}
 see \cref{def:alg_theory_presentation}.

\begin{definition}
  Let $\schema{R}$ and $\schema{R'}$ be schemas given respectively by presentations $(\Xi,\eqnodes{E},\eqatts{E})$
  and $(\Xi',\eqnodes{E'},\eqatts{E'})$. These present entity category $\snodes{R}\iso\Fr(G_{\Xi})/\eqnodes{E}$ and observables profunctor
  $\satts{R}\iso\FrAlg{\Upsilon_{\Xi}}/\eqatts{E}$, and similarly for $\schema{R'}$.

  A \emph{bimodule signature} $\Omega=(\bnodes{\Omega},\batts{\Omega})$ from $\Xi$ to $\Xi'$ is a pair where
  $\bnodes{\Omega}$ is a profunctor signature from $G_{\Xi}$ to $G_{\Xi'}$, and $\batts{\Omega}$ is a profunctor
  signature from $G_{\Xi}$ to $\Sigma$.

  A bimodule signature has an associated algebraic signature
  $\scol{\Omega}=(S_{\scol{\Omega}},\Phi_{\scol{\Omega}})$, where
  \begin{align*}
    S_{\scol{\Omega}} &= (G_{\Xi})_0 \sqcup (G_{\Xi'})_0 \sqcup S_{\Sigma} \\
    \Phi_{\scol{\Omega}} &= (G_{\Xi})_1 \sqcup (G_{\Xi'})_1 \sqcup \Phi_{\Sigma} \sqcup
      \Upsilon_{\Xi} \sqcup \Upsilon_{\Xi'} \sqcup \bnodes{\Omega} \sqcup \batts{\Omega}.
  \end{align*}

  Say that a set $E_{\Omega}$ of equations over $\scol{\Omega}$ is a set of \emph{bimodule
  equations} if for each equation $\Gamma\vdash (t_1=t_2):s'$ of $E_{\Omega}$, the context is a
  singleton $\Gamma=(x:s)$ with $s\in(G_{\Xi})_0$ and $s'\in(G_{\Xi'})_0\sqcup S_{\Sigma}$. We can
  partition the set $E_{\Omega}=\eqnodes{(E_{\Omega})}\sqcup\eqatts{(E_{\Omega})}$, where $\eqnodes{(E_{\Omega})}$ contains
  all equations where $s'\in(G_{\Xi'})_0$, and $\eqatts{(E_{\Omega})}$ contains all equations where $s'\in
  S_{\Sigma}$.

  Given a pair $(\Omega,E_{\Omega})$, consider the category $\Cxt{\scol{\Omega}}/E_{\scol{\Omega}}$,
  where
  \[
    E_{\scol{\Omega}} = \eqnodes{E}\cup\eqatts{E}\cup\eqnodes{E'}\cup\eqatts{E'}\cup E_{\Omega}.
  \]
  The bimodule $M=\FrAlg{\Omega}/E_{\Omega}$ \emph{presented by $(\Omega,E_{\Omega})$} is defined as
  follows:
  \begin{itemize}[nosep]
    \item for any objects $r\in\scol{R}$ and $s\in\scol{R'}$, the set $M(r,r')$ is defined to be the hom-set
      $(\Cxt{\scol{\Omega}}/E_{\scol{\Omega}})(r,s)$,
    \item the functorial actions are given by substitution.
  \end{itemize}
\end{definition}

\begin{remark}\label{collagebimodule}
  The presented bimodule $M=\FrAlg{\Omega}/E_{\Omega}$ may be easier to understand in terms of its
  collage. We will write $\col{\Omega}=(G_{\col{\Omega}},\Upsilon_{\col{\Omega}})$ for the following
  schema presentation:
  \begin{align*}
    (G_{\col{\Omega}})_0 &= (G_{\Xi})_0\sqcup(G_{\Xi'})_0 \\
    (G_{\col{\Omega}})_1 &= (G_{\Xi})_1\sqcup(G_{\Xi'})_1\sqcup\bnodes{\Omega} \\
    \Upsilon_{\col{\Omega}} &= \Upsilon_{\Xi}\sqcup\Upsilon_{\Xi'}\sqcup\batts{\Omega}.
  \end{align*}
  It is easy to see that the algebraic signature $\scol{\col{\Omega}}$ corresponding to the schema
  signature $\col{\Omega}$ as in \cref{def:schema_presentation} is precisely the same as the
  signature $\scol{\Omega}$ given above. Moreover, the collage $\Col{M}$ of the bimodule $M$ is
  presented by $(\col{\Omega},\eqnodes{E}\cup\eqnodes{E'}\cup\eqnodes{(E_{\Omega})},\eqatts{E}\cup\eqatts{E'}
  \cup\eqatts{(E_{\Omega})})$.

  The inclusions $\map{i_\schema{R}}$ and $\map{i_\schema{S}}$ of the schemas $\schema{R}$ and $\schema{S}$ into the collage
  $\Col{M}$ are also easy to understand in terms of this presentation, as they are both inclusions
  on the level of generators and equations as well.
\end{remark}

\begin{example}
  Let $\map{F}=(\mnodes{F},\matts{F})\colon\schema{R}\to\schema{S}$ be a schema morphism. Both its
  companion $\comp{\map{F}}$ and its conjoint $\conj{\map{F}}$ have very simple presentations.

  The generators of $\comp{\map{F}}\colon\schema{R}\tickar\schema{S}$ are
  $\psi_r\colon r\to\mnodes{F}(r)$ for each $r\in\snodes{R}$. For each edge
  $f\colon r\to r'$ in $\schema{R}$ there is an equation $x.f.\psi_{r'}=x.\psi_r.\mnodes{F}(f)$,
  and for each attribute $\EntOb{att}\colon r\to\tau$ in $\schema{R}$ there is an equation
  $x.\EntOb{att}=x.\psi_r.\matts{F}(\EntOb{att})$, both in context $(x:r)$.

  The generators of $\conj{\map{F}}\colon\schema{S}\tickar\schema{R}$ are
  $\phi_r\colon\mnodes{F}(r)\to r$ for each $r\in\snodes{R}$. For each edge $f\colon r\to r'$ in
  $\schema{R}$ there is an equation $x.\phi_r.f=x.\mnodes{F}(f).\phi_{r'}$, and for each attribute
  $\EntOb{att}\colon r\to\tau$ in $\schema{R}$ there is an equation
  $x.\phi_r.\EntOb{att}=x.\matts{F}(\EntOb{att})$, both in context $(x:\mnodes{F}(r))$.
\end{example}

\section{Queries and uber-queries}\label{Queries}

In this section, we will employ many of the concepts and operations studied so far in order to
describe the process of \emph{querying} an algebraic database. We will also give examples that tie
in with running examples from previous sections.

A query is a question asked of a database, such as "Tell me the set of employees whose manager is
named Alice". Queries are often written using "Select-From-Where" \dash e.g.\ in the
database query language SQL \dash or equivalently using "For-Where-Return"
syntax. This syntax both poses the question and provides a table layout in which to record the
results.

In our current setting, we will express a query on a given $\schema{S}$-instance $\inst{J}$, by
constructing a new schema $\schema{R}$ and a bimodule $\bimod{M}\colon\schema{R}\tickar\schema{S}$.
Running the query will amount to applying the functor $\Gamma_M\colon\sinst{S}\to\sinst{R}$ from
\cref{defGamma}. Classically, a For-Where-Return query returns a single table (with no foreign keys),
so the result schema
$\schema{R}$ has a very specific form; namely, its entity side is the terminal category,
$\snodes{R}=\singleton$.

If we allow arbitrary $\schema{R}$ and arbitrary bimodules $\schema{R}\tickar\schema{S}$, the
"query" $\Gamma_M$ could be thought of as a method of migrating data from $\schema{S}$ to
$\schema{R}$, but it could also be considered as a collection of queries and homomorphisms between them;
we refer to such a setup as an \emph{uber-query}. We will discuss this interpretation of bimodules in \cref{sec:uberqueries}.

\subsection{Queries} We begin by discussing the usual For-Where-Return queries and how they appear in our setup.

\begin{definition}
    \label{def:query}
  Let $\schema{S}$ be a schema given by a presentation $(\Xi,E)$, see \cref{def:schema_presentation}.
  A \emph{query} on   $\schema{S}$ is a 4-tuple $\query{Q}=(\qfor{Q},\qwhere{Q},\qatts{Q},\qret{Q})$, where
  $\qfor{Q}$ is a context over $\scol{\Xi}$, $\qwhere{Q}$ is a
  set of equations in $\qfor{Q}$, $\qatts{Q}$ is a context over (the signature of) $\Type$, and
  $\qret{Q}\colon\qfor{Q}\to\qatts{Q}$ is a context morphism over $\scol{\Xi}$.
\end{definition}

We will adopt the For-Where-Return notation for presenting the data of a query as defined in \cref{def:query}, as follows:
\begin{center}
  \small
  \begin{tabular}{@{}rl@{}}
    \textsf{FOR:} & $\qfor{Q}$ \\
    \textsf{WHERE:} & $\qwhere{Q}$ \\
    \textsf{RETURN:} & $\qret{Q}$
  \end{tabular}
\end{center}
This notation is sometimes called \emph{flower syntax} (an acronym of For-Let-Where-Return) or  \emph{comprehension syntax} \cite{Abiteboul:1995a}. 
 

\begin{example}
    \label{ex:query_syntax}
  Let $\schema{S}$ be the schema from \cref{Sschema}. We give an example query $\query{Q}$ on
  $\schema{S}$:
  \begin{center}
    \small
    \begin{tabular}{@{}rl@{}}
      \textsf{FOR:} & $e:\EntOb{Emp}$, $d:\EntOb{Dept}$ \\
      \textsf{WHERE:} & $e\Mor{.wrk.name}=\mathrm{Admin}$,\\
      &$(e\Mor{.sal}\leq d\Mor{.sec.sal})=\top$ \\
      \textsf{RETURN:} 
          &$\Mor{emp\_last}\coloneqq e\Mor{.last}$ \\
          &$\Mor{dept\_name}\coloneqq d\Mor{.name}$ \\
          &$\Mor{diff}\coloneqq d\Mor{.sec.sal}-e\Mor{.sal}$
    \end{tabular}
  \end{center}
  In this query, $\qfor{Q}$ is the context $(e:\EntOb{Emp},d:\EntOb{Dept})$, $\qwhere{Q}$ is the set
  containing the two equations at the \textsf{WHERE} clause, $\qatts{Q}$ is the context
  $(\Mor{emp\_last}:\Str,\Mor{dept\_name}:\Str,\Mor{diff}:\Int)$, and
  $\qret{Q}\colon\qfor{Q}\to\qatts{Q}$ is the context morphism (\cref{def:context_morphism})
  displayed in the \textsf{RETURN} clause.
\end{example}

\subsection{Query bimodules}\label{sec:query_bimodules}

Any query $\query{Q}$ on a schema $\schema{S}$ gives rise to a schema $\schema{R}$ and bimodule
$\bimod{M}\colon\schema{R}\tickar\schema{S}$. The schema $\schema{R}$ is free on the schema
signature $(G,\Upsilon)$, where $G$ is the graph with one node $*$ and no
edges, and $\Upsilon$ has one function symbol, i.e.\ generating attribute $x\colon{*}\to\tau$, for each variable $x:\tau$
in $\qatts{Q}$. Note that the entity category $\snodes{R}$ is terminal, hence
$\satts{R}\colon\snodes{R}\tickar\Type$ can be identified with a single $\Type$-algebra, the free
algebra $\satts{R}=\FrAlg{\qatts{Q}}$. We may refer to $\schema{R}$ as the \emph{result schema}.

Using \cref{prop:semidecomposed_bimodule}, the data of any $\bimod{M}\colon \schema{R}
\tickar \schema{S}$ is equivalent to a single $\schema{S}$-instance $\bint{M}(*)$ denoted
$\bint{M}$, together with a morphism of $\Type$-algebras $\bret{M}\colon \FrAlg{\qatts{Q}}
\to \iterms{(\bint{M})}$. Equivalently, by $\Sigma_{!_S}\dashv\Delta_{!_S}$ this
is a morphism of $\schema{S}$-instances $\bret{M}\colon \Sigma_{!_S}\FrAlg{\qatts{Q}}
= \FrInst{\qatts{Q}} \to \bint{M}$, using \cref{SigmaIpresent} and \cref{def:instance_presentation}.

We thus define $\bint{M}=\FrInst{\qfor{Q}}/\qwhere{Q}$, precisely presented by the first two clauses
of the flower syntax, while $\bret{M}$ is given by the
context morphism $\qret{Q}$ of the last clause (see \cref{def:presented_algebra}).
Following standard database theory,
we refer to $\bint{M}=\FrInst{\qfor{Q}}/\qwhere{Q}$ as the \emph{frozen instance} of the query
$\query{Q}$.

The bimodule $\bimod{M}$ associated to $\query{Q}$ in turn determines a functor
$\Gamma_{M}\colon \sinst{S} \to \sinst{R}$; we will abuse notation by writing it as
$\Gamma_{\query{Q}}$. It is this functor which carries out the operation of ``querying an
$\schema{S}$-instance using $\query{Q}$''. As the result schema $\schema{R}$ has a single entity, the
output of this functor can be seen as a single table containing the results of the query, with one
column for each variable in $\qatts{Q}$.

\begin{example}
    \label{queryviaGamma}
  Let $\schema{S}$ and $\query{Q}$ be as in \cref{ex:query_syntax}. The query $\query{Q}$ determines
  a schema $\schema{R}$ and a bimodule $\bimod{M}\colon\schema{R}\tickar\schema{S}$ as follows. The
  schema $\schema{R}$ has a single entity \dash call it ``${*}$'' \dash and attributes
  $\Mor{emp\_last},\Mor{dept\_name}\colon{*}\to\Str$, and $\Mor{diff}\colon{*}\to\Int$ coming from
  $\qatts{Q}$.

  The bimodule $\bimod{M}$ is determined by the frozen instance
  $\bint{M}=\FrInst{e:\EntOb{Emp},d:\EntOb{Dept}}/\qwhere{Q}$, where $\qwhere{Q}$ contains the two
  equations from \cref{ex:query_syntax}, together with the morphism
  \[
    \bret{M}\colon \FrInst{\Mor{emp\_last},\Mor{dept\_name}:\Str,\Mor{diff}:\Int} \to
      \FrInst{e:\EntOb{Emp},d:\EntOb{Dept}}/\qwhere{Q}
  \]
  given by the context morphism
  \begin{center}
      $[\Mor{emp\_last}\coloneqq e\Mor{.last}$,\;\;
      $\Mor{dept\_name}\coloneqq d\Mor{.name}$, \;\;
      $\Mor{diff}\coloneqq d\Mor{.sec.sal}-e\Mor{.sal}]$.
  \end{center}
  Note that the schema $\schema{R}$ is isomorphic to the one from \cref{sch_map_F}, and that the
  frozen instance $\bint{M}$ is the instance from \cref{frozeninstance}.

  Let $\inst{J}\in\sinst{S}$ be the instance from \cref{Sinstance}. We will now compute the result
  $\Gamma_Q(\inst{J})\in\sinst{R}$ of querying $\inst{J}$ with $\query{Q}$. On the single entity of
  $\schema{R}$, we have by \cref{defGamma} that
  $(\Gamma_QJ)({*})=\sinst{S}(\bint{M},\inst{J})$, and we saw in
  \cref{frozeninstancetransform} that this set has three elements. The $\Type$-algebra
  $\iterms{(\Gamma_QJ)}$ is the same as $\iterms{J}$, by \cref{eq:Gamma_preserves_types}.
  The values of the attributes of $\schema{R}$ are determined using the morphism $\bret{M}$:
  \begin{equation}
      \label{eq:query_atts}
    (\Gamma_QJ)({*}) = \sinst{S}\bigl(\, \FrInst{\qfor{Q}}/\qwhere{Q},\;\inst{J}\, \bigr)
      \longrightarrow
    \sinst{S}\bigl(\, \FrInst{\qatts{Q}},\;\inst{J}\, \bigr)
      \iso \smashoperator{\prod_{(x:\tau)\in\qatts{Q}}} \iterms{J}(\tau).
  \end{equation}
  A transform $\FrInst{\qfor{Q}}\to\inst{J}$ has an underlying context morphism $\Phi\to\qfor{Q}$,
  where $\Phi$ is the context of the canonical presentation of $\inst{J}$ (see
  \cref{rmk:canonical_presentation_inst}). We can express \cref{eq:query_atts} using context
  morphisms: given a transform $\FrInst{\qfor{Q}}/\qwhere{Q}\to\inst{J}$ corresponding to an element
  of $(\Gamma_QJ)({*})$, simply compose its underlying context morphism $\Phi\to\qfor{Q}$
  with $\qret{Q}\colon\qfor{Q}\to\qatts{Q}$. The attributes of this row of the table ``${*}$'' can
  be read off of the resulting context morphism $\Phi\to\qatts{Q}$.

  Doing this, we obtain the result
  \begin{center}
    \small
    \begin{tabular}{@{}>{\itshape}c|c>{}c>{}cr@{}}
      \toprule
      ${*}$ & \Mor{emp\_last} & \Mor{dept\_name} & \Mor{diff} \\
      \midrule
      1     & Noether         & HR               & $100$      \\
      2     & Euclid          & HR               & $150$      \\
      3     & Euclid          & Admin            & $0$        \\
      \bottomrule
  \end{tabular}
  \end{center}
  (Note that the row-ids are arbitrary.) For example, the first row corresponds to the transform
  $\FrInst{\qfor{Q}}\to\inst{J}$ given by $[e\coloneq\mathit{e2},d\coloneq\mathit{d1}]$. Composing
  this with $\qret{Q}$ gives $[\Mor{emp\_last}\coloneq\mathit{e2}.\Mor{last},
  \Mor{dept\_name}\coloneq\mathit{d1}.\Mor{name},
  \Mor{diff}\coloneq\mathit{d1}.\Mor{sec.sal}-\mathit{e1}.\Mor{sal}]$, which simplifies to the first
  row of the table above.
\end{example}

\begin{remark}
    \label{pideltaquery}
  By \cref{Gamma=DeltaPie}, the result of any query $\query{Q}$ on $\schema{S}$, with result schema
  $\schema{R}$ and associated bimodule $\bimod{M}$, is equivalently obtained as the composite of
  data migration functors
  \begin{equation}
      \label{eq:pideltaquery}
  \begin{tikzcd}[row sep=1ex]
    \sinst{S} \ar[r,"\Pi_{i_\schema{S}}"]
      & \sinst{\Col{M}} \ar[r,"\Delta_{i_\schema{R}}"]
      & \sinst{R} \\
    \inst{J} \ar[r,mapsto]
      & \Pi_{i_\schema{S}}(\inst{J}) \ar[r,mapsto]
      & \Delta_{i_\schema{R}} \bigl(\Pi_{i_{\schema{S}}}(\inst{J})\bigr)\cong\Gamma_Q(\inst{J})
  \end{tikzcd}
  \end{equation}
  where the schema mappings
  $\map{i_\schema{R}}\colon\schema{R}\to\Col{M}\from\schema{S}\cocolon\map{i_\schema{S}}$ into the
  bimodule collage $\Col{M}$ are as in \cref{i_Ri_Sconstruction}.

  For example, the query in \cref{queryviaGamma} gave the same result as we found using
  \cref{Piexample,Deltaexample}. One can check that the bimodule collage
  is $\Col{M}\cong\schema{T}$ given in \cref{Tpicture}, and the mappings $\map{F}$, $\map{G}$
  of the mentioned examples are the collage inclusions. Hence this is an instance of \cref{eq:pideltaquery}.
\end{remark}

\begin{remark}
  In practice, one would like a guarantee that a query result $\Gamma_Q(\inst{J})$ is finite
  whenever $\inst{J}$ is finite. To achieve this, one has to place the extra condition on the query
  $\query{Q}$ that only entities \dash no types \dash appear in $\qfor{Q}$. This condition also ensures the
  \emph{domain independence} \cite{Abiteboul:1995a} of the query, meaning that it is not necessary to enumerate the elements of a type to compute the query result.   
\end{remark}

\subsection{Uber-queries}
  \label{sec:uberqueries}

If queries correspond to $(\schema{R},\schema{S})$-bimodules where $\schema{R}$ has only one entity,
then we need a name for more general bimodules; we call them \emph{uber-queries}. An uber-query is
roughly a diagram of queries. The morphisms in this diagram will be called Keys, and our syntax is
accordingly extended to be of the form For-Where-Keys-Return.

\begin{example}
  To describe a bimodule of the following form
  \begin{displaymath}
    \begin{tikzpicture}[schema]
      \node[entity]                            (A)    {A};
      \node[entity,below=of A]                 (A')   {A'};
      \node[type,right=1.2 of A]               (Int)  {Int};
      \node[type,below=of Int]                 (Str)  {Str};
      \node[type,below right=.5 and .5 of Int] (Bool) {Bool};
      \path (A)  edge["dept\_name",inner sep=2pt] (Str)
                 edge["diff"]                     (Int)
            (A') edge["f"]                        (A)
                 edge["last"']                    (Str);
      \node[entity,right=2.4 of Str]            (Dept) {Dept};
      \node[entity,above=of Dept]               (Emp)  {Emp};
      \node[type,right=1.2 of Emp]              (Int')  {Int};
      \node[type,below=of Int']                 (Str')  {Str};
      \node[type,below right=.5 and .5 of Int'] (Bool') {Bool};
      \path (Emp)  edge["wrk"'name=wrk,bend right,
                        inner sep=2pt,pos=.4]    (Dept)
                  edge["mgr"name=mgr,loop above] (Emp)
                  edge["sal"]                    (Int')
                  edge["last",inner sep=2pt]     (Str')
            (Dept) edge["sec"',bend right,
                        inner sep=2pt,pos=.4]    (Emp)
                  edge["name"']                  (Str');
      \path (Str'.south east) ++(9pt,0) node[anchor=south west,schema eqs] (Eqs) {
          plus eqs\\from \cref{Tpicture}
      };
      \node[fit=(wrk) (mgr) (Eqs) (Bool')] (Sbox) {};
      \node[schema name,below left=of Sbox.north east] {$\schema{S}$};
      \node[fit=(mgr.north -| A) (Bool) (A')] (Tbox) {};
      \node[schema name,below left=of Tbox.north east] {$\schema{L}$};
      \draw[->,tick,very thick] (Tbox) to node[above] {$\bimod{N}$} (Sbox);
    \end{tikzpicture}
  \end{displaymath}
  we will need two instances $\inst{I}\coloneqq N(\EntOb{A})$ and $\inst{I'}\coloneqq
  N(\EntOb{A}')$, and a transform $N(f)\colon\inst{I'}\to\inst{I}$ between them, as well as three terms
  $\Mor{diff}, \Mor{name}, \Mor{last}$ of the specified types in $\inst{I}$ and $\inst{I'}$. Indeed,
  this gives a functor $\op{\scol{L}}\to\sinst{S}$ (where objects $\Int$, $\Str$, and $\Bool$ in
  the type side are sent to the corresponding representable instances, as usual;
  \cref{def:bimodule}).

  In \cref{ex:presented_transform}, we constructed two $\schema{S}$-instances
  and a transform $\inst{I'}\to\inst{I}$ between them . We will rewrite them, together with the three terms, in
  For-Where-Keys-Return syntax below.
  \begin{align*}
    \EntOb{A}'= && \EntOb{A}= & \\[1ex]
    \textsf{FOR: } & e':\EntOb{Emp}
    & \textsf{FOR: } & e:\EntOb{Emp},d:\EntOb{Dept} \\
    \textsf{WHERE: } & e'\Mor{.wrk.name}=\tn{Admin}
    & \textsf{WHERE: } & e\Mor{.wrk.name}=\tn{Admin} \\
    & e'\Mor{.sal}\leq e'\Mor{.wrk.sec.sal}
    && e\Mor{.sal}\leq d\Mor{.sec.sal} \\
    \textsf{KEYS: } & \Mor{f}\coloneqq\EntOb{Emp}[e\coloneqq e', d\coloneqq e'\Mor{.wrk}] && \\
    \textsf{RETURN: } & \Mor{last}\coloneqq e'\Mor{.last}
    & \textsf{RETURN: } & \Mor{dept\_name}\coloneqq e\Mor{.wrk.name} \\
    &&& \Mor{diff}\coloneqq d\Mor{.sec.sal}-e\Mor{.sal}
  \end{align*}

  For any $\schema{S}$-instance $\inst{J}$, we can apply $\Gamma_N\colon\sinst{S}\to\sinst{L}$. If
  $\inst{J}$ is as in \cref{Sinstance}, then $\Gamma_N(\inst{J})$ is the following
  $\schema{L}$-instance:
  \begin{center}
    \small
    \begin{tabular}[c]{@{}>{\itshape}c|c>{}c>cr@{}}
      \toprule
      \EntOb{A}'  & \Mor{last} & \Mor{f}     \\
      \midrule
      0           & Euclid     & \itshape{3} \\
      \bottomrule
  \end{tabular}
  \qquad
  \begin{tabular}[c]{@{}>{\itshape}c|c>{}c>{}cr@{}}
    \toprule
    \EntOb{A} & \Mor{dept\_name} & \Mor{diff} \\
    \midrule
    1         & HR               & $100$      \\
    2         & HR               & $150$      \\
    3         & Admin            & $0$        \\
    \bottomrule
  \end{tabular}
  \end{center}
\end{example}

\section{Implementation}\label{Impl}

The mathematics developed in this paper has been implemented in a language we call CQL, which can be downloaded from \url{http://categoricaldata.net}.  In this section we briefly discuss implementation issues that arise, namely in negotiating between syntactic presentations (e.g.\ those discussed in \cref{sec:presentations_syntax}) and the objects they denote.  An in-depth discussion is available in~\cite{Schultz_Wisnesky_2017}.

Most constructions involving finitely-presented categories, including query evaluation and collage
construction, depend crucially on solving word problems in categories, and these problems are not in
general decidable. In \cref{solving_word_problems} we describe our approach to solving word
problems, and in \cref{saturation_and_qeval} we describe how we use word problems to compute
collages and evaluate queries.

\subsection{Solving Word Problems}
  \label{solving_word_problems}

Given a category presentation $(G,E)$ as described in \cref{ssec:category_presentations}, the word
problem is to decide if two terms (words) in $G$ are equivalent under $E$.  The word problem is
obviously semi-decidable: to prove if two terms $p$ and $q$ in $G$ are equal under $E$, we can
systematically enumerate all of the (usually infinite) consequences of $E$ until we find $p = q$.
However, if $p$ and $q$ are not equal, then this enumeration will never stop. In practice, not only
is enumeration computationally infeasible, but for query evaluation and collage construction, we
require a true decision procedure: an algorithm which, when given $p$ and $q$ as input, will always
terminate with ``equal'' or ``not equal''. Hence, we must look to efficient, but incomplete,
automated theorem proving techniques to decide word problems.

The \FPQL\ tool allows any theorem prover to be used to decide word problems. In addition, the
\FPQL\ tool also provides a default, built-in theorem prover based on Knuth-Bendix completion
\cite{Knuth:1970a}: from $(\Sigma,E)$, it attempts to construct a system of re-write rules (oriented
equations), $R$, such that $p$ and $q$ are equal under $E$ if and only if $p$ and $q$ re-write to
syntactically equal terms (so-called \emph{normal forms}) under $R$.  We demonstrate this with an
example.  Consider a presentation of the algebraic theory of groups, on the left, below.
Knuth-Bendix completion yields the re-write system on the right, below: \footnote{Because there is
only one sort, say $S_\Sigma=\{G\}$, we drop the contexts in the Axiom side. For example, the second
equation \dash \emph{axiom} \dash should be $x:G\vdash (x^{-1} \ast x = 1):G$, according to
\cref{sec:presentations_syntax}.}
\begin{center}
  \begin{tabular}{@{}r@{\hspace{3em}}l@{}}
    Axioms & Re-write rules \\
    $1 \ast x = x$ & $1 \ast x \rightsquigarrow x$ \\
    $x^{-1} \ast x = 1$ & $x^{-1} \ast x \rightsquigarrow 1$ \\
    $(x \ast y) \ast z = x \ast (y \ast z)$ & $(x \ast y) \ast z \rightsquigarrow x \ast (y \ast z)$ \\
    & $x^{-1} \ast (x \ast y) \rightsquigarrow y$ \\
    & $1^{-1} \rightsquigarrow 1$ \\
    & $x \ast 1 \rightsquigarrow x$ \\
    & $(x^{-1})^{-1} \rightsquigarrow x$ \\
    & $x \ast x^{-1} \rightsquigarrow 1$ \\
    & $x \ast (x^{-1} \ast y) \rightsquigarrow y$ \\
    & $(x \ast y)^{-1} \rightsquigarrow y^{-1} \ast x^{-1}$
  \end{tabular}
\end{center}

To see how these re-write rules are used to decide the word problem, consider the two terms $(a^{-1}
\ast a) \ast (b \ast b^{-1})$ and $b \ast ((a \ast b)^{-1} \ast a)$.  Both of these terms re-write
to $1$ under the above re-write rules; hence, we conclude that they are equal. In contrast, the two
terms $1 \ast (a \ast b)$ and $b \ast (1 \ast a)$ re-write to $a \ast b$ and $b \ast a$,
respectively, which are not syntactically the same; hence, we conclude that they are not equal.

The details of how the Knuth-Bendix algorithm works are beyond the scope of this paper.  However, we
make two remarks.  First, Knuth and Bendix's original algorithm (\cite{Knuth:1970a}) can fail even
when a re-write system to decide a word problem exists; for this reason, we use the more modern,
``unfailing'' variant of Knuth-Bendix completion \cite{Bachmair:1989a}.  Second, we remark that
Buchberger's algorithm for computing Gr\"obner bases is a very similar algorithm used in many computer
algebra systems, and it may be seen as the instantiation of the Knuth-Bendix algorithm in the theory
of polynomial rings \cite{Marche:1996a}.

\subsection{Saturation and Query Evaluation}
  \label{saturation_and_qeval}

Given a category presentation $(G,E)$, a decision procedure for the word problem allows us to (semi)
compute the category $\cat{C}$ that $(G,E)$ presents.  To do this, we construct $\cat{C}$ in stages:
first, we find all non-equal terms of size $0$ in $G$; \footnote{By the \emph{size} of a term, we
  mean the height of the associated syntax tree. For example
$\Mor{max}(x\Mor{.sal},x\Mor{.mgr.sal})$ has size of three.} call this $\cat{C}^0$. Then, we add to
$\cat{C}^0$ all non-equal terms of size $1$ that are not equal to a term in $\cat{C}^0$; call this
$\cat{C}^1$.  We iterate this procedure, obtaining a sequence $\cat{C}^0, \cat{C}^1, \ldots$.  If
$\cat{C}$ is indeed finite, then there will exist some $n$ such that $\cat{C}^n = \cat{C}^{n+1} =
\cat{C}$ and we can stop. Otherwise, our attempt to construct $\cat{C}$ will run forever: it is not
decidable whether a given presentation $(G,E)$ generates a finite category. The category $\cat{C}$
will be isomorphic to the category $\Fr(G)/E$ obtained by quotienting the free category on $G$ by
the equations in $E$ (\cref{def:category_presentation}); essentially, $\cat{C}$ represents
equivalence classes of terms by a smallest possible representative, as explained in detail in
\cref{sec:presentations_syntax}. 
(Note that the normal forms chosen by the internal Knuth-Bendix theorem prover 
for the purposes of deciding the word problem need not in general
be the same as the chosen representatives of equivalences classes in $\cat{C}$ described above.) 

Most uses of the \FPQL\ tool involve \emph{saturating} instance presentations into collages so that
they may be examined as tables (e.g.\ \cref{frozeninstance}). This just means replacing part of the
presentation with a canonical presentation (see
\cref{rmk:canonical_presentation,rmk:canonical_presentation_alg,rmk:canonical_presentation_inst}\footnote{For explanatory reasons these particular examples saturate a frozen instance associated 
with a query, but the implementation does not need to saturate frozen instances to evaluate queries.})
The saturation process is very similar to the process described in the preceding paragraph, with one
small difference.  In general, the `type side' of the collage (see \cref{collageschema}) will denote
an infinite category.  For example, if $\Type$ is the free group on one generator $\{a\}$, it will
contain $a$, $a \ast a$, $a \ast a \ast a$, and so on. Hence, it is usually not possible to saturate
the type side of an instance. So, the \FPQL\ tool saturates only the entity side of an instance,
which will often be finite in practice.  From the saturated entity side presentation and a set of
re-write rules for the collage, it is possible to construct a set of tables that faithfully
represent the instance.  The tables for the entity side of the instance are simply a tabular
rendering of the finite category corresponding to the saturated entity side of the instance's
collage ($\cat{C}$ in the preceding paragraph).  The tables for the attributes of the instance must
also contain representatives of equivalence classes of terms, but unlike the entity side case, where
representatives are chosen based on size, it is less clear which representative to choose. For
example, there is an implicit preference to display 1,024 instead of $2^{10}$, even though the
size (as defined in the previous footnote) of the former is greater than the size of the latter.
The \FPQL\ tool allows these representatives to be computed by
external programs, thereby providing a ``hook'' for the tool to interface with other programming languages and systems.
For example, users can provide a Java implementation of natural numbers for the commutative ring type side
used in this paper, and the java compiler will normalize terms like $2^{10}$ into $1024$.  By default, the
\FPQL\ tool will display the normal forms computed by the internal Knuth-Bendix theorem prover in the attribute tables.

\begin{figure}
  \includegraphics[width=6in]{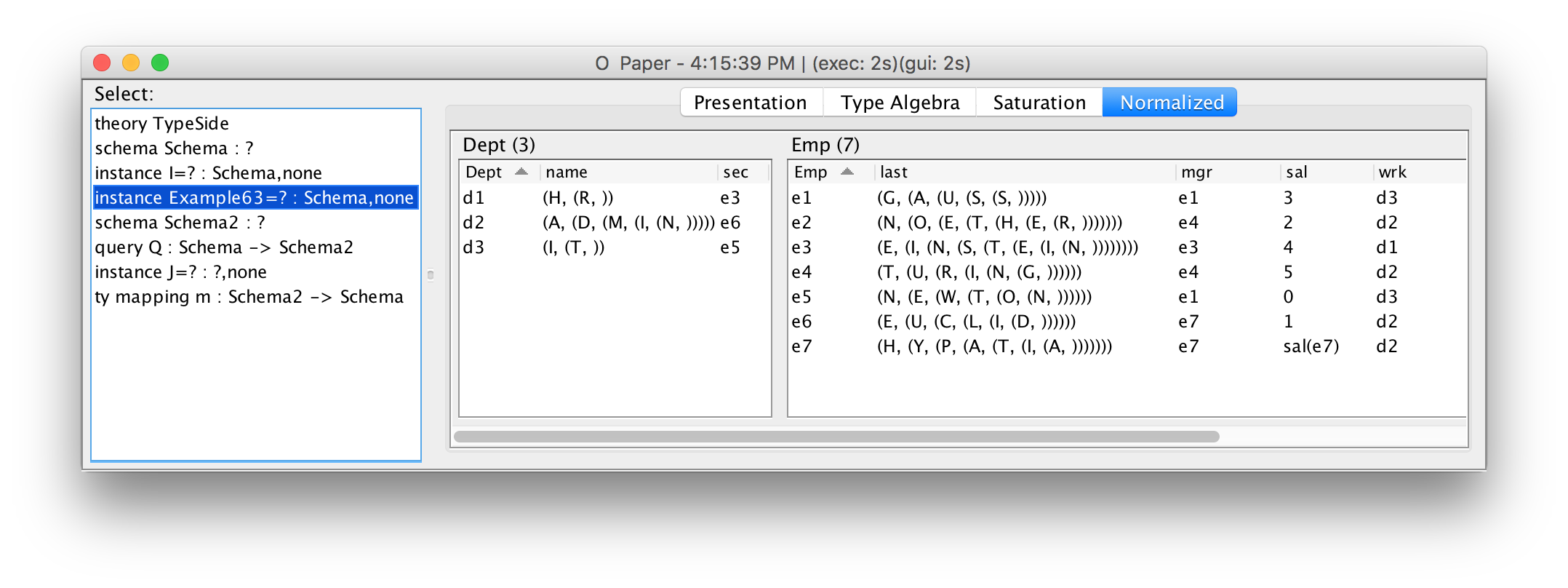}
  \caption{CQL displaying the instance from \cref{Sinstance}}
  \label{fig:OPL_pic}
\end{figure}

To evaluate a query $\query{Q}$ such as that in \cref{queryviaGamma} on a presented
instance $\inst{I}$, we first saturate the entity side of $\inst{I}$ as described in the preceding
paragraph. Evaluation of the query, $\Gamma_Q(\inst{I})$ as in \cref{sec:query_bimodules},
proceeds similarly to evaluation of
`For-Where-Return' queries in traditional SQL systems \cite{Abiteboul:1995a}: first, we compute a
(typically large) set of tuples corresponding to the FOR clause by repeatedly looping through
$\inst{I}$. Then, we filter this set of tuples by the WHERE clause; here we must be sure to decide
equality of tuples under the equational theory for $\inst{I}$, using Knuth-Bendix as described
above.  Finally, we project out certain components of these tuples, according to the RETURN clause.
The result of the query will be a saturated instance, which has a canonical presentation as in
\cref{rmk:canonical_presentation_inst}.

The \FPQL\ tool's tabular rendering of the instance from \cref{Sinstance} is shown in
\cref{fig:OPL_pic}.  Because a unary representation of the integers is computationally inefficient, for expediency
the employee salaries in the \FPQL\ program have been reduced compared to \cref{Sinstance}.  
A more efficient axiomatization of the integers, such as using binary, can also be used.





\appendix

\section{Componentwise composition and exponentiation in \texorpdfstring{$\Data$}{Data}}
  \label{sec.componentwise}

We defined composition of bimodules and 2-cells in \cref{bimodulecomposition,2cellcomp}
and exponentiation of bimodules in \cref{Dataisclosed}. In \cref{componentwisebimodule} we saw
that bimodules can be equivalently defined in several components, separating the entity and type
sides of the structure. It is natural to ask what composition and exponentiation (and as special
cases, the data migration functors) look like in this decomposed view.

In fact, when first working out the ideas presented in this paper, we used componentwise
formulas to understand all the constructions. In writing it up, we decided that the coend formulas
were more succinct and often easier to work with; however, the machinery below still turns out to be
useful in certain cases, so we present it without proof for the interested reader.

Recall the left tensor $\otimes$ defined in \cref{def:otimes}, which `preserves algebraicity' of
profunctors on the right.

\begin{proposition}
    \label{decomposed_bimod_comp}
  The composition $\bimod{M}\odot\bimod{N}$ of two bimodules $\schema{R} \xtickar{\bimod{M}}
  \schema{S} \xtickar{\bimod{N}} \schema{T}$ in $\Data$ \cref{bimodulecomposition} is equivalently
  given in components as follows: $\bnodes{(M\odot N)} = \bnodes{M}\odot\bnodes{N}$ in $\dProf$, and
  the rest of the components are given using a pushout, as in the following diagram in the category
  $\Fun{\op{\snodes{R}},\TypeAlg}$:
  \begin{equation*}
    \begin{tikzcd}[labels={inner sep=.2ex}
                  ,column sep={between origins,4em}
                  ,row sep={between origins,4em}
                  ,bend angle=30]
      \satts{R} \ar[dr,"\bret{M}"] \ar[ddrr,bend right=40,"\bret{(M\odot N)}"']
      && \bnodes{M}\otimes\satts{S}
        \ar[dl,"\batts{M}"']
        \ar[dr,"\bnodes{M}\otimes\bret{N}"']
        \ar[dd,phantom,sloped,"\displaystyle\ulcorner"{pos=.9,rotate=45}]
      && \bnodes{M}\otimes\bnodes{N}\otimes\satts{T} \ar[dl,"\bnodes{M}\otimes\batts{N}"']
      \ar[ddll,bend left=40,"\batts{(M\odot N)}"] \\
      & \bterms{M} \ar[dr]
      && \bnodes{M}\otimes\bterms{N} \ar[dl] & \\
      && \bterms{(M\odot N)}
    \end{tikzcd}
  \end{equation*}
\end{proposition}

This follows from the following lemma, which can be proven using \cref{prop:inverse_equiv}:

\begin{lemma}
    \label{cor:Res_of_odot}
  Let $L\colon A_0\tickar A_1$, $M\colon B_0\tickar B_1$, and $N\colon C_0\tickar C_1$ be proarrows
  in an equipment $\dcat{D}$ with extensive collages and local finite colimits. Let
  $X\in\Simp{L}{M}$ and $Y\in\Simp{M}{N}$ be simplices, and let $P\colon\scol{L}\tickar\scol{M}$ and
  $Q\colon\scol{M}\tickar\scol{N}$ be proarrows such that $X\iso\res{L}{M}(P)$ and
  $Y\iso\res{M}{N}(Q)$ (see \cref{res_functor}). Then the components of $\res{L}{N}(P\odot Q)$ can
  be computed by pushout:
  \begin{equation} \label{eq:simplex_composition}
    \begin{tikzcd}
      X_{i,0}\odot M\odot Y_{1,j} \ar[r] \ar[d] \ar[dr,phantom,"\ulcorner" pos=.9]
      & X_{i,0}\odot Y_{0,j} \ar[d] \\
      X_{i,1}\odot Y_{1,j} \ar[r]
      & \comp{i}_{A_i}\odot(P\odot Q)\odot\conj{i}_{C_j}
    \end{tikzcd}
  \end{equation}
  Moreover, the 2-cells of $\res{L}{N}(P\odot Q)$ are found using these pushouts in an evident way.
\end{lemma}

\begin{proposition}
    \label{decomposed_2cell_comp}
  The horizontal composition of 2-cells in $\Data$
  \[
    \begin{tikzcd}
      \schema{R} \ar[r,tick,"\bimod{M}" domA] \ar[d,"\map{F}"']
      & \schema{S} \ar[r,tick,"\bimod{N}" domB] \ar[d,"\map{G}"]
      & \schema{T} \ar[d,"\map{H}"] \\
      \schema{R}' \ar[r,tick,"\bimod{M}'"' codA]
      & \schema{S}' \ar[r,tick,"\bimod{N}'"' codB]
      & \schema{T}'
      \twocellA{\twoCell{\theta}}
      \twocellB{\twoCell{\phi}}
    \end{tikzcd}
  \]
  is given (\cref{decomposed_2cell}) by $\twonodes{(\theta\odot\phi)}=
  \twonodes{\theta}\odot\twonodes{\phi}$, while $\twoterms{(\theta\odot\phi)}$ is induced by the
  diagram
  \[
    \begin{tikzcd}[column sep=4.2em]
      \bterms{M} \ar[d,"\twoterms{\theta}"']
      & \bnodes{M}\otimes\satts{S} \ar[d,"\twonodes{\theta}\otimes\matts{G}"]
      \ar[l,"\batts{M}"'] \ar[r,"\bnodes{M}\otimes\bret{N}"]
      & \bnodes{M}\otimes\bterms{N} \ar[d,"\twonodes{\theta}\otimes\twoterms{\phi}"] \\
      \comp{\mnodes{F}}\otimes\bterms{M}'
      & \comp{\mnodes{F}}\otimes\bnodes{M}'\otimes\satts{S}'
      \ar[l,"\id\otimes\batts{M}'"] \ar[r,"\id\otimes\bnodes{M}'\otimes\bret{N}'"']
      & \comp{\mnodes{F}}\otimes\bnodes{M}'\otimes\bterms{N}'
    \end{tikzcd}
  \]
  where $\bterms{(M\odot N)}$ and $\bterms{(M'\odot N')}$ are the pushouts of the top and bottom
  rows respectively, by \cref{decomposed_bimod_comp}.
\end{proposition}

\begin{proposition}
    \label{decomposed_exponent}
  Let $\bimod{N}\colon\schema{S}\tickar\schema{T}$ and $\bimod{P}\colon\schema{R}\tickar\schema{T}$
  be bimodules. The bimodule $\bimod{N}\rhd\bimod{P}\colon\schema{R}\tickar\schema{S}$ is given as
  follows: the entity component $\bnodes{(N\rhd P)}$ is computed by a pointwise pullback, for any
  objects $s\in\snodes{S}$, $r\in\snodes{R}$,
  \[
    \begin{tikzcd}
      \bnodes{(N\rhd P)}(r,s) \ar[r] \ar[dd] \ar[ddr,phantom,"\lrcorner" very near start]
      & \Set^{\snodes{T}}\bigl(\bnodes{N}(s,\textrm{--}),\bnodes{P}(r,\textrm{--})\bigr)
        \ar[d,"\batts{P}"] \\
      & \Set^{\snodes{T}}\Bigl(\bnodes{N}(s,\textrm{--}),
          \TypeAlg\bigl(\satts{T}(\textrm{--}),\bterms{P}(r)\bigr)\Bigr)
      \ar[d,"\iso"] \\
      \TypeAlg\bigl(\bterms{N}(s),\bterms{P}(r)\bigr) \ar[r,"\batts{N}"]
      & \TypeAlg\bigl((\bnodes{N}\otimes\satts{T})(s),\bterms{P}(r)\bigr).
    \end{tikzcd}
  \]
  Equivalently, $\bnodes{(N\rhd P)}(r,s)=\sinst{T}\bigl(N(s),P(r)\bigr)$. The other components are
  $\bterms{(N\rhd P)}=\bterms{P}$, $\bret{(N\rhd P)}=\bret{P}$, and $\batts{(N\rhd P)}$ is the
  composition
  \[
    \bnodes{(N\rhd P)}\otimes\satts{S} \to
    \TypeAlg\bigl(\bterms{N}(\textrm{--}),\bterms{P}(\textrm{--})\bigr)\otimes\bterms{N}
    \to \bterms{P}.
  \]
\end{proposition}

\section{Errata}
  \label{sec.errata}

This section describes certain minor errors revealed by the authors of \cite{gorenroig2024presentingprofunctors}.  Their work is still ongoing; see 
\cite{err1} , \cite{err2} and \cite{err3} for details, and 
 \cite{gorenroig2024presentingprofunctors} discusses the simple case of an empty typeside.

\subsection{Non-equivalence of syntactic and semantic categories}

\cref{rmk:canonical_presentation} claims that the functor
$\Cxt{}\colon\APr\to\ATh$ is fully faithful, and thus an equivalence.
\cref{rmk:canonical_presentation_alg} similarly claims that the
category of $\T$-algebra presentations as given in
\cref{def:presented_algebra} is equivalent to $\T\alg$.  Neither of these claims are true:

\begin{example}\label{ex:non-equiv_theories}
  Let $P$ be the algebraic presentation with one sort, $p$, one
  function symbol, $f:(p)\to p$, and no equations.  Let $Q$ be the
  algebraic presentation with one sort, $q$, one function symbol,
  $g:(q)\to q$, and one equation $x:q\vdash g(x)=x$.
  Now define morphisms $F,G:P\to Q$ by $F(p),G(p)\coloneqq q$ and
  $F(f)\coloneqq (x:q\vdash g(x))$, $G(f)\coloneqq (x:q\vdash x)$.
  Then $F\neq G$ while $\Cxt{}(F)=\Cxt{}(G)$, so $\Cxt{}$ is not
  faithful.
\end{example}

\begin{example}
  Let $P$ be as in \cref{ex:non-equiv_theories}.  Define algebra
  presentations $(\Gamma, E)$ and $(\Gamma', E')$ on $P$, where
  $\Gamma\coloneqq (x:p)$, $E\coloneqq\varnothing$,
  $\Gamma'\coloneqq (y:p)$, and $E'\coloneqq\{f(y)=y\}$.
  Now define morphisms $F,G:(\Gamma, E)\to (\Gamma', E')$ by $F=[x\coloneqq y]$ and
  $G=[x\coloneqq f(y)]$.
  Then $F\neq G$ while the images of $F$ and $G$ in $\T\alg$ are
  equal.
\end{example}

There are two possible ways to address this error (stated in the case
of theories, but should work \emph{mutatis mutandis} for the case of
instances):

\begin{itemize}
\item Define what it means for two morphisms in $\APr$ to be
  ``provably equal''.  Show that ``provable equality'' is a congruence
  $\approx$.  Then the quotient category $\APr/\approx$ will be
  equivalent to $\ATh$.  See
  \cite{gorenroig2024presentingprofunctors} Prop. 2.12 for how to do this
  in the special case of category presentations.
\item Make $\APr$ into a bicategory where, given morphisms $f,g$,
  there is a unique morphism $f\Rightarrow g$ iff $f\approx g$.
  Consider $\ATh$ a bicategory in a trivial sense.  Then $\APr$ and
  $\ATh$ should be equivalent as bicategories.
\end{itemize}

The latter approach may seem excessive for this simple example, but points to an essential bicategorical nature of syntax taken up in \cite{gorenroig2024presentingprofunctors}-- see Remark 3.19 for a
case in which this viewpoint is more clearly fruitful.

\subsection{Strict and Weak Bimodules}

The definition of \emph{bimodule} in \cref{def:bimodule} requires that
the square \eqref{eqn:bimodule_form2} commutes up to \emph{equality},
i.e. $M(\tau)=\yoneda(\tau)$ for any $\tau\in\Type$.  However, the
formula \eqref{bimodulecomposition} given for bimodule composition in
\cref{def:Data} results in a functor for which
\eqref{eqn:bimodule_form2} commutes only up to \emph{isomorphism},
i.e. we only have an isomorphism $(M\odot N)(\tau)\cong\yoneda(\tau)$,
natural in $\tau$, not an equality.  Thus bimodules as defined do not
form a double category, as stated.

To examine this discrepancy, let us call bimodules as defined in
\cref{def:bimodule} ``strict bimodules'', and define a weaker notion
as follows:

\begin{definition}
    \label{def:weakbimodule}
  Let $\schema{R}$ and $\schema{S}$ be database schemas. A \emph{weak bimodule} $\bimod{M} \colon
  \schema{R} \tickar \schema{S}$ is a functor
  $M\colon \op{\scol{R}} \to \sinst{S}$ along with a natural isomorphism
  
  \begin{equation}\label{eqn:bimodule_form2-fixed}
    \begin{tikzcd}
      \op{\Type} \ar[r,"\op{\inctypes}"] \ar[d,"\op{\inctypes}"']
      & \op{\scol{S}} \ar[d,"\yoneda"] \\
      \op{\scol{R}} \ar[r,"M"'] \ar[ur, phantom, "\scriptstyle\Uparrow\alpha_M"{description}]
      & \sinst{S}
    \end{tikzcd}
  \end{equation}
  or succinctly, $\alpha_M:M(\tau)\cong\yoneda(\tau)$
  natural in $\tau\in\Type$.

  A morphism of $(\schema{R},\schema{S})$-bimodules $\twoCell{\phi}\colon\bimod{M}\to\bimod{N}$ is a natural
  transformation $\phi\colon M\Rightarrow N$ such that $\alpha_N\circ\phi_\tau=\alpha_M$ for all $\tau\in\Type$.
\end{definition}

If $\alpha_M$ is the identity, then $\bimod{M}$ is strict.  Note that we need to include the isomorphism $\alpha_M$ as
data in the definition of $\bimod{M}$; it is not enough to say ``such
that ... is isomorphic''.  In fact, changing this isomorphism can
result in substantially different bimodules (the possible choices of
$\alpha_M$ are in 1-1 correspondence with the automorphisms
of $\Type$).

The  characterization of strict bimodules in
\cref{prop:semidecomposed_bimodule} is often useful.  Here
is a similar characterization of weak bimodules:

\begin{lemma}
  A weak bimodule $\bimod{M}\colon\schema{R} \tickar \schema{S}$ can be equivalently
  represented as a quadruple
  \begin{align*}
    \batts{M}&\colon\op{\snodes{R}}\to\sinst{S},\\
    \bterms{M}&\colon\op{\Type}\to\sinst{S},\\
    \bret{M}&\colon \satts{R}\Rightarrow\sinst{S}(\bterms{M},\batts{M}),\\
    \alpha_M&\colon\yoneda\circ\op{\inctypes}\cong\bterms{M}
  \end{align*}
\end{lemma}

If the bimodule is strict, then this reduces to \cref{prop:semidecomposed_bimodule}.  Define ``strictification'' to relate strict and weak bimodules.

\begin{definition}
  Let $\bimod{M}$ be a weak bimodule.  Then define the
  \emph{strictification} $\widetilde{\bimod{M}}$ of $\bimod{M}$ as the
  following strict bimodule (utilizing
  \cref{prop:semidecomposed_bimodule}):
  \begin{align*}
    \widetilde{\batts{M}}&\coloneqq\batts{M}\\
    \widetilde{\bret{M}}&\coloneqq \left(\satts{R}\xRightarrow{\bret{M}}\sinst{S}(\bterms{M},\batts{M})\xRightarrow{(\alpha_M^{-1})^*}\sinst{S}(\yoneda\circ\op{\inctypes},\batts{M})\cong\batts{M}(-,\inctypes) \right)
  \end{align*}
\end{definition}

\begin{lemma}
  \begin{itemize}
    \item There is a canonical isomorphism
      $\bimod{M}\cong\widetilde{\bimod{M}}$.
    \item Strictification $\bimod{M}\mapsto\widetilde{\bimod{M}}$ extends to a
      functor from the category of weak bimodules to the category of
      strict bimodules.
    \item The aforementioned canonical isomorphism is natural in $\bimod{M}$.
    \item The strictification functor is quasi-inverse to the
      forgetful functor from strict to weak bimodules, as witnessed by
      the canonical isomorphism.
    \item Thus the categories of strict and weak bimodules are equivalent.
  \end{itemize}
\end{lemma}

To resolve this problem, the authors of \cite{gorenroig2024presentingprofunctors} consider two
possibilities:

\begin{itemize}
\item Only use strict bimodules, and append a strictification step to
  the definition of bimodule composition.  This would work in theory,
  and it is indeed what is done when composing presentations of
  bimodules in practice (see \cite[Section
    4.3.3]{Schultz_Wisnesky_2017}, also see \ref{bimodulepresentationserratasection} for a general
  comment on bimodule presentations).  However, it is mathematically
  inelegant, perhaps even ``evil'' (see
  \cite{nlab:principle_of_equivalence}), similar to insisting that a
  functor $F:C\to\text{Set}$ is not representable unless it is
  \emph{equal} to $C(c,-)$ for some $c$, notwithstanding that it may
  be isomorphic to $C(c,-)$.  What is lost here are not merely
  aesthetic, because many key insights, such as weak associativity of
  composition, cease to be intuitive once there are strictifications
  everywhere.
\item Use weak bimodules.  But then everything gets a lot more
  complicated, as $\alpha_M$ must always be accounted for.
  Although perhaps this complexity subsides.
\end{itemize}

\subsection{Right Closures of Bimodules}

The formula given for right closure in \cref{Dataisclosed} is
incorrect.  The proof of this proposition is erroneous as it omits the
requirement that $M\circ\op{\inctypes}\cong
\yoneda\circ\op{\inctypes}$ through some $\alpha_M$ as discussed above--- the central application of the
$\Lambda_N\dashv\Gamma_N$ adjunction does not go through anymore if
this requirement is observed.  On the other hand, the formula given
for right closure in \cref{decomposed_exponent} is correct, and
non-equivalent to the formula in \cref{Dataisclosed}.

\subsection{Unexpected Properties of Bimodule Presentations}\label{bimodulepresentationserratasection}

Bimodule presentations, as defined in this paper, have the following unfortunate property:

\begin{lemma}
  There exist finite bimodule presentations $(\Omega, E)$ and
  $(\Omega', E')$, presenting composable bimodules $\kappa[\Omega]/E$
  and $\kappa[\Omega']/E'$, such that the composite bimodule
  $\kappa[\Omega]/E \odot \kappa[\Omega']/E'$ has no finite
  presentation.
\end{lemma}
\begin{proof}
  Let the typeside be empty.  Then schemas are categories and
  bimodules are profunctors.  Bimodule presentations are now
  ``uncurried profunctor presentations'', as defined in \cite{gorenroig2024presentingprofunctors} Section
    3.1.  The lemma follows by
  \cite{gorenroig2024presentingprofunctors} Prop 3.9.
\end{proof}

On the contrary, the ``uberflowers'' of \cite{Schultz_Wisnesky_2017}
are a non-equivalent way of presenting bimodules, and a construction
is given for composition of finite uberflowers, resulting in a finite
uberflower.  It has not been proven to our knowledge that this
construction agrees with bimodule composition, however we have proved
the correctness of a composition algorithm in the empty-typeside case,
for which ``uberflowers'' become the ``curried profunctor presentations'',
of \cite{gorenroig2024presentingprofunctors}; see Thm. 3.26.

\bibliographystyle{acm}
\bibliography{Library}

\end{document}